\documentclass[twoside,12pt]{article}

\usepackage{amsmath}
\usepackage{latexsym,amssymb}
\usepackage{xcolor}
\usepackage{amsthm}
\usepackage{amsfonts}
\usepackage{dsfont}
\usepackage{float}
\usepackage{enumitem}
\usepackage{subcaption}

\usepackage{hyperref}
\usepackage{cleveref}
\usepackage{geometry}
\usepackage{graphicx}

\def\rmd{\mathrm{d}}
\def\Law{\mathrm{Law}}
\def\1{\mathds{1}}
\def\cU{\mathcal{U}}
\def\cB{\mathcal{B}}

\newtheorem{theorem}{Theorem}
\newtheorem{lemma}[theorem]{Lemma}
\newtheorem{corollary}[theorem]{Corollary}
\newtheorem{proposition}[theorem]{Proposition}
\newtheorem{definition}[theorem]{Definition}
\newtheorem{remark}[theorem]{Remark}
\newtheorem{assumption}{Assumption}

\begin{document}

\title{Convergence of kinetic Langevin samplers for non-convex potentials}

%\author{\name Katharina Schuh\hspace{0.3mm} \email katharina.schuh@tuwien.ac.at\\
%       \addr Institute of Analysis and Scientific Computing\\
%       TU Wien, Wiedner Hauptstra\ss e 8--10, 1040 Wien, Austria
%       \AND
%    \name Peter A. Whalley\hspace{0.3mm}\email peter.whalley@math.ethz.ch\\
%       \addr Seminar for Statistics\\
%       ETH Z{\"u}rich, Z{\"u}rich, Switzerland 
%       }

%\editor{}

\author{Katharina Schuh \thanks{Institute of Analysis and Scientific Computing, TU Wien, Wiedner Hauptstra\ss e 8--10, 1040 Wien, Austria, katharina.schuh@tuwien.ac.at}
\thanks{The author acknowledges partial support from the Austrian Science Fund (FWF), grant F65.} \and Peter A. Whalley \thanks{Seminar for Statistics, Department of Mathematics, ETH Z{\"u}rich, Switzerland,  peter.whalley@math.ethz.ch)} \thanks{The author acknowledges the support of the Engineering and Physical Sciences
Research Council Grant EP/S023291/1 (MAC-MIGS Centre for Doctoral Training)} }

\date{}

\maketitle

\begin{abstract}
We study three kinetic Langevin samplers including the Euler discretization, the BU and the UBU splitting scheme. We provide contraction results in $L^1$-Wasserstein distance for non-convex potentials. 
These results are based on a carefully tailored distance function and an appropriate coupling construction. 
Additionally, the error in the $L^1$-Wasserstein distance between the true target measure and the invariant measure of the discretization scheme is bounded. To get an $\varepsilon$-accuracy in $L^1$-Wasserstein distance, we show complexity guarantees of order $\mathcal{O}(\sqrt{d}/\varepsilon)$ for the Euler scheme and $\mathcal{O}(d^{1/4}/\sqrt{\varepsilon})$ for the UBU scheme under appropriate assumptions on the target measure. %To get an $\varepsilon$-accuracy in $L^1$-Wasserstein distance, we show complexity guarantees of order $\mathcal{O}(\sqrt{d}/\varepsilon)$ for the Euler scheme and $\mathcal{O}(d^{1/4}/\sqrt{\varepsilon})$ for the UBU scheme under appropriate regularity assumptions on the target measure. %, we observe an asymptotic bias of order one whereas for UBU scheme an asymptotic bias of order two under appropriate assumptions on the step size.
The results are applicable to interacting particle systems and provide bounds for sampling probability measures of mean-field type. 
\end{abstract}

%\begin{keywords}
%   Markov Chain Monte Carlo; Langevin diffusion; Wasserstein convergence; numerical analysis of SDEs
%\end{keywords}

\section{Introduction}

We are interested in the long-time behaviour of discretizations of the kinetic Langevin dynamics on $\mathbb{R}^{2d}$ given by 
\begin{align} \label{eq:LD}
\begin{cases}
\rmd X_t = V_t \rmd t \\
\rmd V_t = -\nabla U(X_t) \rmd t -\gamma V_t  \rmd t + \sqrt{2\gamma} \rmd B_t,
\end{cases}
\end{align}
where $U$ is a twice-differential potential, $\gamma>0$ denotes the friction parameter and $(B_t)_{t\ge 0}$ is a $d$-dimensional standard Brownian motion. Apart from its origin to model phenomena occurring in physics, this dynamics is applied to sample a given probability distribution $\mu^x(\rmd x) \propto \exp(-U(x))\rmd x$ on $\mathbb{R}^d$. The distribution is the marginal of the Boltzmann-Gibbs distribution $\mu_{\infty}(\rmd x)\propto \exp(-U(x)-|v|^2/2)\rmd x \rmd v$ on $\mathbb{R}^{2d}$ forming the stationary distribution of \eqref{eq:LD}. As observed in \cite{Cao2023}, these dynamics provide a faster convergence behaviour than for example the overdamped Langevin dynamics $\rmd X_t=-\nabla U(X_t)\rmd t+\sqrt{2}\rmd B_t$ with stationary distribution $\mu^x$, which after an appropriate time-rescaling is the high-friction limit of \eqref{eq:LD}.

As in simulations, the exact dynamics are not accessible in general, in practice the dynamics need to be discretized by a numerical scheme. There are many different choices of numerical scheme that have been proposed for kinetic Langevin dynamics from the molecular dynamics community (see \cite{durmus2021uniform,leimkuhler2016computation,leimkuhler2013rational,bussi2007accurate,skeel2002impulse,brunger1984stochastic,melchionna2007design,leimkuhler2013robust}) as well as the machine learning and the MCMC community (see \cite{shen2019randomized,cheng2018underdamped,foster2021shifted,foster2024high,bou2023randomized}). We consider three discretization schemes of the kinetic Langevin dynamics. First, we analyse the explicit Euler-Maruyama scheme which is a simple but straightforward scheme to implement. Secondly, we consider the two more elaborate splitting schemes, BU and UBU, which were introduced in \cite{skeel1999integration,alamo2017word}. These two splitting schemes rely on the fact that the dynamics are split into two parts, i.e., $\cU$ and $\cB$, and each of them is integrated exactly (in the weak sense). As the names of the splitting schemes suggest, for discretization parameter $h>0$, the schemes are constructed by performing a $B$ step before a $U$ step for the BU scheme. Alternatively, a half $\cU$ step (of size $h/2$), then a full $\cB$ step and finally another half $\cU$ step of size $h/2$ is realised for the UBU scheme  (see \cite{skeel1999integration,alamo2017word,sanz2021wasserstein,chada2023unbiased}).

In the analysis of the long-time performance of the three schemes, we aim to understand the influences of a non-convex potential including multi-well potentials. Here, we will assume that $U$ is $\kappa$-strongly convex outside of a Euclidean ball with radius $R>0$ and has an $L$-Lipschitz continuous gradient. This setting is very important in applications in molecular dynamics and Bayesian inference (see for example \cite{leimkuhler2015molecular} and \cite{jasra2005markov}), where the potentials can be highly multimodal.

%\medskip

%\paragraph{Literature}
Before presenting our own contribution, let us first highlight the existing results for the kinetic Langevin sampler.
Already the study of the long-time behaviour of the continuous dynamics has attracted enormous interest through many techniques including Lyapunov techniques (see \cite{wu2001large,mattingly2002ergodicity}), hypocoercivity techniques (see \cite{villani2006hypocoercivity,dolbeault2009hypocoercivity,dolbeault2015hypocoercivity,bakry2006diffusions,baudoin2016wasserstein,baudoin2017bakry,albritton2019variational,Cao2023,brigati2023construct}),  and coupling techniques (see \cite{eberle2019couplings,schuh2022global}). 

It is also important to quantify the long-time behaviour of the discretizations of kinetic Langevin dynamics, as these are implemented in practice in many applications. In combination with bias estimates these result in non-asymptotic guarantees for the implemented algorithms. The long-time behaviour of kinetic Langevin dynamics discretizations have also been studied by various techniques including Lyapunov arguments (see \cite{durmus2021uniform,leimkuhler2015molecular}), coupling methods (see \cite{cheng2018underdamped,dalalyan2020sampling,monmarche2021high,sanz2021wasserstein,leimkuhler2023contractiona,gouraud2022hmc,leimkuhler2023contractionb} in the convex setting and \cite{cheng2018sharp,chak2023reflection} in the non-convex setting) and recently hypocoercivity approaches have been extended to certain discretizations (see \cite{monmarche2022entropic} and \cite{camrud2023second}). 

The focus of this paper is using coupling methods to study the long-time behaviour of the kinetic Langevin sampler. This is a recently popular approach to study sampling algorithms and their non-asymptotic guarantees. It includes overdamped Langevin dynamics-based sampling methods and Hamiltonian dynamics-based sampling methods, which can be united by variants of the OBABO integrator (see \cite{gouraud2022hmc} and \cite{chak2023reflection} for generalized Hamiltonian Monte Carlo). Non-asymptotic guarantees for Hamiltonian Monte Carlo methods have been studied in \cite{mangoubi2017rapid,bou2020coupling,bou2023convergence,bou2023nonlinear,bou2022unadjusted,gouraud2022hmc,camrud2023second,monmarche2022entropic} and also Langevin Monte Carlo methods in \cite{Dalalyan2017,Durmus2017}.

Our main contribution is twofold. First, we establish contraction in Wasserstein distance for each scheme. To our knowledge, this is the first contraction result in a non-globally convex setting for these discretization schemes. More precisely, we adapt the idea of the coupling of \cite{chak2023reflection} for OBABO to the discretization schemes considered here, construct a distance function $\rho$ which is based on the one in \cite{schuh2022global} and show that contraction in $L^1$-Wasserstein distance with respect to the distance $\rho$ holds for the Euler and the BU scheme. The results hold provided the step size is sufficiently small and the friction parameter $\gamma$ is sufficiently large. This is consistent with the observations in the continuous case. Since $\rho$ is equivalent to the Euclidean distance and  contraction for the UBU scheme can be deduced up to an additional factor from the contraction for the BU scheme, we obtain exponential decay in the classical $L^1$ Wasserstein distance for all three schemes, i.e.,  
\begin{align*}
     \mathcal{W}_{1}(\mu_k, \nu_k)\le \mathbf{M }e^{-ck h} \mathcal{W}_1(\mu_0,\nu_0),
\end{align*}
where $\mu_k$ and $\nu_k$ denotes the law after $k$ discretization steps. The contraction rate $c>0$ and the constant $\mathbf{M}>0$ are independent of the step size $h$ and the dimension $d$. This result implies existence of a unique invariant measure and convergence towards it for each scheme. Secondly, we give an error analysis and establish complexity guarantees for each scheme. More precisely we bound the asymptotic bias $\mathcal{W}_{1}(\mu_{h,\infty},\mu_{\infty})$, where $\mu_{h,\infty}$ is the invariant measure of the discretized process with step size $h$ and $\mu_{\infty}$ is the invariant measure of \eqref{eq:LD}. Inspired by the strong convergence of numerical solutions of SDEs \cite[Theorem 1.1]{Milstein2004} we are able to only lose an order of $1/2$ accuracy from local to global strong error estimates in terms of stepsize $h>0$. That is an asymptotic bias of order one for the Euler-Maruyama scheme and order two under additional smoothness assumptions for the UBU scheme. For the UBU integrator, we achieve second-order asymptotic bias estimates inspired by the work of \cite{sanz2021wasserstein}, but in comparison to \cite{sanz2021wasserstein}'s approach we achieve this by using the independence of the Brownian increments during each iteration and average over multiple steps, which is the approach we use for analysis of both integrators to achieve bias estimates in the constructed distance function.

Combining these results for the Euler-Maruyama scheme we have complexity guarantees of order $\mathcal{O}(\sqrt{d}/\epsilon)$ to reach an accuracy of $\epsilon>0$ in $\mathcal{W}_{1}$ for the Euler-Maruyama scheme and additionally for the UBU scheme when combining the respective UBU results. Under additional smoothness assumptions, we can achieve complexity guarantees of order $\mathcal{O}(\sqrt{d}/\sqrt{\epsilon})$ for the UBU scheme and under a stronger smoothness condition complexity guarantees of order $\mathcal{O}(d^{1/4}/\sqrt{\epsilon})$, which is true for many applications of interest.

Finally, we remark that the contraction and complexity results can be carried over to interacting particle models with pairwise interactions (see \Cref{rem:particle} and \Cref{rem:particle_comp}). These models play an important role for instance in modelling granular media in physics (see \cite{benedetto1998granular}), in molecular dynamics problems using a harmonic or Morse interaction potential (see for example 
\cite[Chapter 1]{leimkuhler2015molecular}) or two-layer neural networks in deep learning \cite{Hu2021,Mei2018, Rotskoff2022, Sirignano2022}. 
As the number of particles in the model tends to infinity, the target measure becomes the stationary measure of a distribution-dependent version of the kinetic Langevin dynamics (see Equation \eqref{eq:limitmeas}). Using particlewise adaptations of the coupling and the distance function, contraction results independent of the particle number are proven both for the continuous kinetic Langevin dynamics in \cite{bolley2010trend,kazeykina2020, guillin2022, schuh2022global} and for kinetic samplers in \cite{camrud2023second, bou2023nonlinear}.  The two latter papers also provide complexity guarantees of sampling this type of measure. Our paper contributes to the analysis of kinetic samplers by providing bounds in this setting for the Euler, BU and UBU discretization. 

%\medskip

\paragraph{Notation}
We denote by $\mathcal{B}(\mathbb{R}^{2d})$ the Borel $\sigma$-algebra of the space $\mathbb{R}^{2d}$  and by $\mathcal{P}(\mathbb{R}^{2d})$ the space of all probability measures on $(\mathbb{R}^{2d}, \mathcal{B}(\mathbb{R}^{2d}))$. A coupling $\omega$ of two probability measures $\nu,\eta \in \mathcal{P}(\mathbb{R}^{2d})$ is a probability measure on the space $(\mathbb{R}^{2d}\times \mathbb{R}^{2d}, \mathcal{B}(\mathbb{R}^{2d})\otimes \mathcal{B}(\mathbb{R}^{2d}))$ with marginals $\nu$ and $\eta$. The $L^1$-Wasserstein distance with respect to a distance function $\rho:\mathbb{R}^{2d}\times\mathbb{R}^{2d}\to[0,\infty)$ is given by
\begin{align*}
    \mathcal{W}_{\rho}(\nu, \eta)=\inf_{\omega\in \Pi(\nu, \eta)}\int_{\mathbb{R}^{2d}\times \mathbb{R}^{2d}} \rho(z,z') \omega( \rmd z \ \rmd z'),
\end{align*}
where $\Pi(\eta, \nu)$ denotes the set of all couplings of $\nu$ and $\eta$. If we consider the Euclidean distance for the distance function we write $\mathcal{W}_1$.

%\medskip

\paragraph{Outline of the paper}
The paper is organized in the following way. In \Cref{sec:prelim}, we define rigorously the discretization schemes and state the precise framework. In \Cref{sec:mainresults} the contraction results are stated for the different discretization schemes followed by the accuracy analysis of these schemes and numerical illustrations of the contraction results. The metric and coupling construction and the proofs are postponed to \Cref{sec:dist_coupl} and \Cref{sec:proofs}.

\section{Discretization schemes and preliminaries} \label{sec:prelim}

%The simplest discretization of \eqref{eq:LD} is a simple Euler discretization of the dynamics which has strong local error of the order $3/2$ and global strong error of order $1$ (see \cite{Milstein2004}[Theorem 1.1]).
\subsection{Euler-Maruyama discretization}
The simplest discretization of \eqref{eq:LD} is a simple explicit Euler discretization of the dynamics. For given discretization parameter $h>0$, the scheme is given by 
\begin{align} \label{eq:EM}
\begin{cases}
\mathbf{X}_{k+1}= \mathbf{X}_k+h\mathbf{V}_{k} \\
\mathbf{V}_{k+1}= \mathbf{V}_{k} - h \nabla U(\mathbf{X}_{k})-h\gamma \mathbf{V}_k+\sqrt{2 \gamma h} \xi_{k+1},
\end{cases}
\end{align}
where $(\xi_k)_{k\in\mathbb{N}}$ is a sequence of independent normally distributed random variables.
This discretization scheme has strong local error of the order $3/2$ and global strong error of order $1$ (see \cite[Theorem 1.1]{Milstein2004}).

\begin{remark}
    In \cite{cheng2018underdamped,NesterovMCMC,zhang2023improved} they consider the stochastic Euler scheme, which is derived from freezing the force and solving the dynamics exactly, this allows one to use analysis techniques based on Girsanov's theorem which are not generally applicable to more sophisticated discretizations.
    In practice, one would prefer to resort to more accurate second order discretizations which only require one gradient evaluation per step (see \cite{leimkuhler2013rational}).
    
    We introduce the Euler-Maruyama discretization as its simple expression allow for illustration of the complexity analysis in a more comprehensible way before moving on to second order discretizations. Although one would argue that second order methods are a better choice in practice, the Euler-Maruyama discretization is often used for its ease of implementation, for example, within the context of stochastic gradient HMC (see \cite{chen2014stochastic}). 
\end{remark}

% \pete{Regarding R2 point 3, I have added the above remark. Feel free to edit any of it.}
% \kati{Thanks. I think this remark also addresses the second paragraph of the main comments and concerns in R3.}

\subsection{Splitting Methods and the UBU discretization}

More advanced numerical schemes of \eqref{eq:LD} can be made by the use of splitting methods, where the dynamics are split into different components (deterministic and stochastic) which can be integrated exactly in the weak sense. We refer the reader to \cite{mclachlan2002splitting} for a comprehensive introduction to splitting methods. By careful design of the appropriate splittings one can create high order numerical methods in the strong and the weak sense, as discussed in \cite{leimkuhler2013rational, bussi2007accurate}. 

A class of splitting schemes which are typically used in molecular dynamics (see \cite{leimkuhler2015molecular}) are based on splitting the SDE \eqref{eq:LD} in the following way
\[
\begin{pmatrix}
\rmd x \\
\rmd v
\end{pmatrix} = \underbrace{\begin{pmatrix}
0 \\
-\nabla U(x)\rmd t
\end{pmatrix}}_{\mathcal{B}}+ \underbrace{\begin{pmatrix}
v\rmd t \\
0
\end{pmatrix}}_{\mathcal{A}} +\underbrace{\begin{pmatrix}
0 \\
-\gamma v \rmd t + \sqrt{2\gamma}\rmd B_t
\end{pmatrix}}_{\mathcal{O}},
\]
where the $\mathcal{B}$, $\mathcal{A}$ and $\mathcal{O}$ parts can be integrated exactly over a time interval of size $h > 0$ and composed in different orders to produce different splitting methods. These include the popular integrators BAOAB, OBABO and OABAO, where palindromic sequences produce weak order two numerical methods (see \cite{leimkuhler2013rational}). The OBABO and OABAO integrators have been studied in the context of non-asymptotic guarantees in \cite{chak2023reflection,monmarche2021high,camrud2023second,gouraud2022hmc,bou2023mixing}. These methods are weak order two, but are only strong order one. 

Strong order methods are particularly important in the context of multilevel Monte Carlo (see \cite{giles2015multilevel}) and recently in unbiased estimation in \cite{chada2023unbiased}. An alternative splitting first introduced in \cite{skeel1999integration} requires only one gradient evaluation per step, yet surprisingly is globally strong order two. It is based on splitting the SDE \eqref{eq:LD} into the following components
\[
\begin{pmatrix}
\rmd x \\
\rmd v
\end{pmatrix} = \underbrace{\begin{pmatrix}
0 \\
-\nabla U(x)\rmd t
\end{pmatrix}}_{\mathcal{B}} +\underbrace{\begin{pmatrix}
v\rmd t \\
-\gamma v \rmd t + \sqrt{2\gamma}\rmd B_t
\end{pmatrix}}_{\mathcal{U}},
\]
which can be integrated in the weak sense exactly over an interval of size $h>0$. As in the earlier methods, we can compose the maps corresponding to the exact integration of the $\mathcal{B}$ and $\mathcal{U}$ parts to design numerical integrators of kinetic Langevin dynamics. \cite{alamo2017word} consider the BUB and the UBU methods. The UBU integrator with step size $h > 0$ is defined by a half step in $\mathcal{U}$ (of size $h/2$) followed by a full $\mathcal{B}$ step (of size $h$), followed by a half $\mathcal{U}$ step (of size $h/2$). Let us define $\eta = \exp{\left(-\gamma h\right)}$ then the operators corresponding to these maps are given by
\begin{equation}\label{eq:Bdef}
\mathcal{B}(x,v,h) = (x,v - h\nabla U(x)),
\end{equation}
and
\begin{equation}\label{eq:Udef}
\begin{split}
\mathcal{U}(x,v,h,\xi^{(1)},\xi^{(2)}) &= \Big(x + \frac{1-\eta}{\gamma}v + \sqrt{\frac{2}{\gamma}}\left(\mathcal{Z}^{(1)}\left(h,\xi^{(1)}\right) - \mathcal{Z}^{(2)}\left(h,\xi^{(1)},\xi^{(2)}\right) \right),\\
& v\eta + \sqrt{2\gamma}\mathcal{Z}^{(2)}\left(h,\xi^{(1)},\xi^{(2)}\right)\Big),
\end{split}
\end{equation}
where 
\begin{equation}\label{eq:Z12def}
\begin{split}
\mathcal{Z}^{(1)}\left(h,\xi^{(1)}\right) &= \sqrt{h}\xi^{(1)},\\
\mathcal{Z}^{(2)}\left(h,\xi^{(1)},\xi^{(2)}\right) &= \sqrt{\frac{1-\eta^{2}}{2\gamma}}\Bigg(\sqrt{\frac{1-\eta}{1+\eta}\cdot \frac{2}{\gamma h}}\xi^{(1)} + \sqrt{1-\frac{1-\eta}{1+\eta}\cdot\frac{2}{\gamma h}}\xi^{(2)}\Bigg),
\end{split}
\end{equation}
and $\xi^{(1)}, \xi^{(2)} \sim \mathcal{N}\left(0,I_{d}\right)$ are independent standard normal random variables. $\mathcal{Z}^{(1)}$ with stepsize $h>0$ is equivalent to {$\int^{h}_{0}dB_{s}$}, and $\mathcal{Z}^{(2)}$ is equivalent to {$\int^{h}_{0}e^{-(h-s)\gamma}dB_{s}$}.
UBU integration scheme with stepsize $h >0 $ is defined by 
\begin{equation}\label{eq:PhUBU}
\begin{split}
\left(x_{k+1},v_{k+1}\right)&=\mathcal{UBU}\left(x_k,v_k,h,\xi^{(1)}_{k+1},\xi^{(2)}_{k+1},\xi^{(3)}_{k+1},\xi^{(4)}_{k+1}\right)
\\
&=\mathcal{U}\left(\mathcal{B}\left(\mathcal{U}\left(x_{k},v_{k},h/2,\xi^{(1)}_{k+1},\xi^{(2)}_{k+1}\right),h\right),h/2,\xi^{(3)}_{k+1},\xi^{(4)}_{k+1}\right),
\end{split}
\end{equation}
where $\xi^{(i)}_{k+1} \sim \mathcal{N}(0,I_{d})$ for all $i = 1,...,4$ and $k \in \mathbb{N}$. The UBU integration scheme's non-asymptotic guarantees were first studied in \cite{sanz2021wasserstein}, where discretization analysis provided global strong order two estimates under an additional smoothness assumption. This integrator was also studied in \cite{chada2023unbiased,paulin2024sampling} with stochastic gradient methods which allow strong order $3/2$ and $2$ with non-asymptotic guarantees.

\begin{remark}
   \cite{shen2019randomized} and \cite{bou2023randomized} introduce discretizations for kinetic Langevin dynamics which randomize the point at which the gradient is evaluated resulting in improved dimension-dependence without additional smoothness. The methods we develop in this work could equally be used to prove convergence of these methods by coupling the times at which the force evaluations take place. This could then be combined with discretization analysis to achieve non-asymptotic guarantees, but this was not the focus of the current work.
\end{remark}

\subsection{Assumptions}

We impose the following assumption on the potential $U$.

\begin{assumption}\label{ass}
The potential $U$ is $\kappa$-strongly convex outside a Euclidean ball with radius $R\ge 0$, i.e., there exist $\kappa>0$ and $R\ge 0$ such that
\begin{align*}
(\nabla U(x)-\nabla U(y))\cdot(x-y)\ge \kappa|x-y |^2 \qquad \text{for } x, y\in\mathbb{R}^d \text{ such that } |x-y|>R.
\end{align*}
Moreover, $\nabla U$ is $L$-Lipschitz continuous. 
\end{assumption}

From this condition on $U$ we deduce that $U$ can be split in an quadratic term and in a term which is convex outside an Euclidean ball, i.e., 
$U(x)=\frac{1}{2}x^T Kx + G(x)$, where the function $G$ satisfies
\begin{align*}
(\nabla G(x)-\nabla G(y))\cdot(x-y)\ge 0 \qquad \text{for } x, y\in\mathbb{R}^d \text{ such that } |x-y|>R.
\end{align*} 
and the matrix $K$ is positive-definite and symmetric with smallest eigenvalue $\kappa>0$.
We denote by $L_G$ the Lipschitz constant of the gradient of the function $G$ and by  $L_K$ the Lipschitz constant of the function $x\mapsto Kx $. We note that this splitting is not unique and a possible choice for $K$ is always given by $K=\kappa I_d$, which is not necessarily the optimal one. 

\subsection{Sketch of the distance function and the coupling}

Next, we give a rough sketch of the construction of the coupling and the accompanying metric and demonstrate how they lead to contraction in Wasserstein distance for the discretization schemes. 

For the Euler discretization, consider two sets of normally distributed random variables $(\xi_k)_{k\in \mathbb{N}}$ and $(\xi_k')_{k\in \mathbb{N}}$. Given initial values $(x,v), (x',v')\in \mathbb{R}^{2d}$, let $((\mathbf{X}_k,\mathbf{V}_k),(\mathbf{X}_k',\mathbf{V}_k'))_{k\in \mathbb{N}}$ be the coupling of two solutions to \eqref{eq:EM} with $(\xi_k)_{k\in \mathbb{N}}$ and $(\xi_k')_{k\in \mathbb{N}}$, respectively.  Denote by $(Z_k,W_k)_{k\in \mathbb{N}}=(\mathbf{X}_k-\mathbf{X}_k',\mathbf{V}_k-\mathbf{V}_k')_{k\in \mathbb{N}}$ the difference process of the two copies. 
If the difference process is far apart from the origin, we consider a twisted $2$-norm $r_l(k)$ as in the continuous setting (see \cite{schuh2022global}). Using a synchronous coupling, i.e., $\xi_k=\xi_k'$, the noise cancels in the difference process and we show local contraction for this norm thanks to the strong convexity outside a Euclidean ball with radius $R$ of the potential $U$ (\Cref{ass}) and the discretized part of the Ornstein-Uhlenbeck part of the Langevin dynamics.
If the difference process is close to the origin, we consider a distance function of the form $r_s(k)=\alpha |Z_k|+ |q_k|$ with $q_k=Z_k+\gamma^{-1} W_k$. It holds for the process $(Z_k,q_k)_{k\in \mathbb{N}}$
\begin{align*}
    \begin{cases}
         Z_{k+1}=Z_{k}+ h \gamma(q_k-Z_k)
        \\  q_{k+1}=q_k - h\gamma^{-1} (\nabla U(\mathbf{X}_k)-\nabla U(\mathbf{X}_k'))+ \sqrt{2\gamma^{-1} h}(\xi_{k+1}-\xi_{k+1}'). 
    \end{cases}
\end{align*}
Then, for $q_k=0$ and $\xi_k=\xi_k'$, the first term in the distance function decreases due to the first equation and the contribution of the second term in the distance function can be controlled by the first one by choosing the parameter $\alpha$ sufficiently large. Note that thanks to the synchronous coupling we do not have disturbance by the noise. Apart from $q_k=0$ we want to use the noise to get closer to the line $q_k=0$. In contrast to the continuous case we do not use a completely reflecting coupling, since in the case the process is already close to the line $q_k=0$, the reflected noise can not be controlled. We rather consider the coupling introduced for HMC in \cite{bou2020coupling}, which is applied to analyse OBABO in \cite{chak2023reflection}. Here we have a transition from the synchronous coupling on the line  $q_k=0$ to a reflection coupling if the difference process is far away from the line. Note that if $h$ tends to zero, we recover the coupling from \cite{eberle2019couplings}.
If the difference process is close to the origin, i.e., $\alpha |Z_k|+ |q_k|<R_1$ for some constant $R_1$, local contraction on average is shown for the distance $f(r_s(k))$, where $f$ is an appropriate concave function. As in \cite{schuh2022global}, the two distances $f(r_s(k))$ and $r_l$ are continuously glued to a joint distance function such that the two local contraction results imply a global one. The transition between the two distances is illustrated in \Cref{fig:metric}.
\begin{figure}
    \centering
    \includegraphics[scale=0.5]{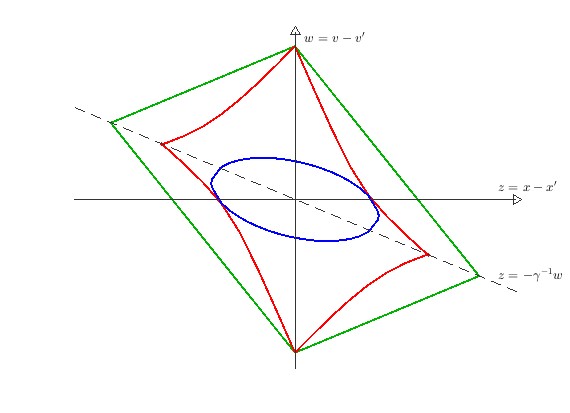}
    \caption{\footnotesize Level sets of the distances: The green rectangle shows a level set of the distance $r_s$, while the blue ellipse shows a level set of the distance $r_l$. The red deformed rectangle illustrates the transition between the two distances. Outside the red rectangle and on the gray dashed line the synchronous coupling is considered while inside the red rectangle the coupling from \cite{bou2020coupling} is applied.  }
    \label{fig:metric}
\end{figure}

For the BU scheme, two sets of normally distributed random variables $(\xi_k^{(1)},\xi_k^{(2)})_{k\in \mathbb{N}}$ and $(\xi_k^{(1)'},\xi_k^{(2)'})_{k\in \mathbb{N}}$ are considered for the coupling. In this case, $\xi_k^{(1)}$ and $\xi_k^{(1)'}$ are coupled as $\xi_k$ and $\xi_k'$ in the Euler scheme, whereas the random variables $\xi_k^{(2)}$ and $\xi_k^{(2)'}$ are coupled synchronously. For the same distance function as in the Euler scheme, global contraction is shown by proving two local contraction results. 

The UBU scheme is analysed using the coupling and the result for the BU scheme. Namely for given $k\in \mathbb{N}$, $k$ $\mathcal{U}_{1/2}\mathcal{B}\mathcal{U}_{1/2}$ steps can be split to 
\begin{align*}
    (\mathcal{U}_{1/2}\mathcal{B}\mathcal{U}_{1/2})^k=\mathcal{U}_{1/2}(\mathcal{B}\mathcal{U})^{k-1}(\mathcal{B}\mathcal{U}_{1/2}),
\end{align*}
here $\mathcal{U}_{1/2}$ denotes a $\mathcal{U}$ step with size $h/2$.
Then for $(k-1)$ $\mathcal{B}\mathcal{U}$ steps the known contraction result is used. The missing steps $\mathcal{U}_{1/2}$ and $\mathcal{B}\mathcal{U}_{1/2}$ are controlled using a synchronous coupling for the random variables and noting that the distance can be controlled for these steps by the fact that $\nabla U$ is $L$-Lipschitz continuous.

\section{Main results} \label{sec:mainresults}

\subsection{Contraction for the Euler-Maruyama scheme}
We state the contraction result for the Euler-Maruyama scheme given by \eqref{eq:EM}.
\begin{theorem}[Contraction for the Euler-Maruyama scheme] \label{thm:contra_EM}
Let $U$ satisfy \Cref{ass}. Let $(x,v), (x',v')\in \mathbb{R}^{2d}$. Let $\gamma>0$ and $h>0$ satisfy
\begin{align} 
&\gamma \ge 4 L_G/\sqrt{\kappa} \qquad  \text{and}   \label{eq:gamma_cond}
\\ &   L\gamma^{-1} h \le \min\Big(\frac{1}{8LR_1^2}, \frac{1}{256\cdot 75(2L\gamma^{-2}+1) }, \frac{L\gamma^{-2}}{8}, \frac{L}{32L_K}\Big) . \label{eq:condition_h} 
\end{align}
Then there exists a distance $\rho:\mathbb{R}^{2d}\times \mathbb{R}^{2d}\to[0,\infty)$ such that for all $k\in \mathbb{N}$
\begin{align}\label{eq:contraction_new}
\mathbb{E}[\rho((\mathbf{X}_{k},\mathbf{V}_{k}),(\mathbf{X}_{k}',\mathbf{V}_{k}'))]\le (1-c h)^k\rho ((x,v),(x',v')),
\end{align}
where the contraction rate $c$ is given by 
\begin{align} \label{eq:c}
c= \min\Big(f'(R_1)\frac{\epsilon \kappa \gamma^{-1}}{8} \mathcal{E}, f'(R_1)\frac{\epsilon \gamma}{16} \mathcal{E},f'(R_1)\frac{\gamma}{8},f'(R_1)\frac{\gamma \alpha}{2},\frac{9\hat{c}}{640}, \frac{\hat{c}}{32(4\alpha+1)} \Big).
\end{align}
The definition of the distance function $\rho:\mathbb{R}^{2d}\times \mathbb{R}^{2d}\to[0,\infty)$, the function $f:[0,\infty)\to[0,\infty)$ and the constants $\epsilon$, $\alpha$, $\hat{c}$, $\mathcal{E}$ and $R_1$ are independent of $d$ and $h$ and are given in Section~\ref{sec:dist}.
\end{theorem}

\begin{proof}
The proof is given in Section~\ref{sec:proofs_euler}.
\end{proof}

Consider two probability measures $\mu_0$, $\nu_0$ on $\mathbb{R}^{2d}$. Let $(\mathbf{X}_0,\mathbf{V}_0)$ and $(\mathbf{X}_0',\mathbf{V}_0')$ be distributed with respect to $\mu_0$ and $\nu_0$, respectively.
Then, we denote by $\mu_k=\Law(\mathbf{X}_k,\mathbf{V}_k)$ and $\nu_k=\Law(\mathbf{X}_k',\mathbf{V}_k')$ the laws of the discretised kinetic Langevin sampler after $k\in\mathbb{N}$ steps, respectively.

\begin{corollary}[Contraction in $L^1$-Wasserstein distance]
Suppose Assumption~\ref{ass}, \eqref{eq:gamma_cond} and \eqref{eq:condition_h} hold.
Then, for $k\in\mathbb{N}$
\begin{align}
    & \mathcal{W}_{\rho}(\mu_k, \nu_k)\le \exp(-ckh)  \mathcal{W}_{\rho}(\mu_0, \nu_0), \label{eq:contr_k_step}
    \\ & \mathcal{W}_{1}(\mu_k, \nu_k)\le \mathbf{M}\exp(-ckh)  \mathcal{W}_{1}(\mu_0, \nu_0) ,
\end{align}
where $c$ is given by \eqref{eq:c}, and $\mathbf{M}$ is given in \eqref{eq:M}. Both constants are independent of $d$ and $h$.
%\begin{align}
%    \mathbf{M}=f'(R_1)^{-1}\frac{2\max(\gamma(1+2L\gamma^{-2}), 1)}{\epsilon\min(\sqrt{2\kappa},1) }, \label{eq:M}
%\end{align}
%where  the function $f$ and the constants $R_1$ and $\epsilon$ are given in \Cref{sec:dist}.
Moreover, existence of a unique invariant measure $\mu_{h,\infty}$ and convergence towards its holds, i.e., for $k \in \mathbb{N}$
\begin{align*}
    & \mathcal{W}_{\rho}(\mu_k, \mu_{h,\infty})\le \exp(-ckh)  \mathcal{W}_{\rho}(\mu_0, \mu_{h,\infty}),
    \\ & \mathcal{W}_{1}(\mu_k, \mu_{h,\infty})\le \mathbf{M}\exp(-ckh)  \mathcal{W}_{1}(\mu_0, \mu_{h,\infty}).
\end{align*}
\end{corollary}

\begin{proof}
The results follow immediately from Theorem~\ref{thm:contra_EM} and the fact that $\rho$ is equivalent to the Euclidean distance by \eqref{eq:dist_equiv}. Existence of a unique invariant measure holds by Banach fixed point theorem.
\end{proof}

\begin{remark}
    We note that for the strongly convex case the contraction rate reduces to $$c= \min(\kappa/(8\gamma^2),1/16)\gamma$$ and is maximized for $\kappa= \gamma^2/2$. If $L_G \le \sqrt{8} \kappa$, the condition $\gamma\ge 4L_G/\sqrt{\kappa}$ is satisfied. Hence, for small perturbations of the Gaussian case, we obtain that the contraction rate given by $c= \gamma/16= \sqrt{2\kappa}/16$ is of optimal order in $\kappa$ (see for example \cite{gouraud2022hmc}).
\end{remark}
% \pete{Do you think it would be good to add a remark regarding the restriction on the friction parameter (I am referring to point 2 of R2). We can point out that this is the Lipschitz constant of $\nabla G$ and not $\nabla U$.}
% \kati{I thought about this too. We should highlight that we can get good results for small perturbations of the Gaussian case, even when the 'condition number' $L_K/\kappa$ is large. In the paper of Dalalyan and Riou-Durand only the condition number of the full potential is considered.}
\begin{remark}[Particle model] \label{rem:particle}
    As in the continuous case (\cite{schuh2022global}), the convergence result can be carried over to show convergence for a particle system with $N\in\mathbb N$ particles, where the potential $U: \mathbb{R}^{dN}\to \mathbb R$ is of the form 
    \begin{align*}
        U(\mathbf{x})=\sum_{i=1}^N \Big(V(x^i)+\frac{1}{N}\sum_{j=1 , j\neq i}^N W(x^i-x^j)\Big)
    \end{align*}
    with confining potential $V:\mathbb R^d\to \mathbb R$ and interaction potential $W:\mathbb R^d\to \mathbb R$.
    In this case, contraction can be shown in $L^1$-Wasserstein distance with respect to a particlewise adaptation of the distance $\rho$ and the Euclidean distance, i.e. 
    \begin{align*}
        \rho_N ((\mathbf{x},\mathbf{v}), (\mathbf{x}',\mathbf{v}'))=\frac{1}{N}\sum_{i=1}^N \rho((x^i,v^i),(x^{i'},v^{i'})), \quad \text{and } \quad \ell^N_1=\frac{1}{N}\sum_{i=1}^N |(x^i,v^i)-(x^{i'},v^{i'})|
    \end{align*}
    provided that the interaction potential has Lipschitz continuous gradient and the Lipschitz constant is sufficiently small compared to the strong convexity constant $\kappa$, i.e., $W$ is a small perturbation compared to the confining potential. Following \cite{schuh2022global}, the condition on the Lipschitz coefficient is needed in this approach to handle the interaction as an additional perturbation. In this case, contraction for the particle system is of the form 
    \begin{align*}
        \mathcal{W}_{\rho_N}(\mu^N_k, \nu^N_k)\le \exp(-ckh) \mathcal{W}_{\rho_N}(\mu^N_0, \nu^N_0),
    \end{align*}
    where $\mu^N_k$ and $\nu^N_k$ denote the law of the discretized Langevin dynamics for the $N$-particle system and $c$ is the contraction rate which is independent of $N$ and which will be of the same form as the contraction rate in \eqref{eq:c} up to some constant prefactor.    
        
\end{remark}

\subsection{Complexity guarantees for the Euler-Maruyama scheme}

Next, we bound the distance between the target measure $\mu_{\infty}(\rmd x)\propto \exp(-U(x)-|v|^2/2)\rmd x $ and the invariant measure $\mu_{h,\infty}$ of \eqref{eq:EM}.

\begin{theorem}[Strong accuracy] \label{thm:strong_accuracy}
    Suppose Assumption~\ref{ass}, \eqref{eq:gamma_cond} and \eqref{eq:condition_h} hold. Then,
    \begin{align*}
        \mathcal{W}_{\rho}(\mu_{\infty},\mu_{h,\infty})& \le h\Big(1+\frac{\gamma(1+ 2L\gamma^{-2})}{c}\Big)20 L \gamma^{-2} \sqrt{d}, %\Big( (1+1/\alpha)e^{\gamma^4/L^2}(\gamma \mathbf{M}_1+\alpha \mathbf{M}_2)+h\frac{\gamma \alpha}{2}\mathbf{M}_1\Big) %\Big(9\frac{L^2 h^2 \gamma^{-2} }{8}\mathbf{C}_X+ 2\gamma^{-1}L h^2\mathbf{C}_V +L\gamma^{-2}\frac{37 h}{8}\sqrt{2\gamma hd}\Big),
        %c^{-1} h \int_{\mathbb{R}^{2d}} \Big(\frac{9}{4} L^2\gamma^{-2}|x|+ L\gamma^{-1}\Big(\frac{\sqrt{L}}{4\gamma}+\frac{3}{2}\Big)|v|\Big)\rmd \mu_{\infty}(x,v)  +  c^{-1}\frac{5}{4}L\gamma^{-2}\sqrt{2\gamma hd} .
    \end{align*}
    where $c$ is given in \eqref{eq:c}. % and the moment bounds $\mathbf{C}_X=\int_{\mathbb{R}} |x|\rmd \mu_{\infty}(x,v)$ and $\mathbf{C}_V=\int_{\mathbb{R}} |v|\rmd \mu_{\infty}(x,v)$. %$\mathbf{C}_X$ and $\mathbf{C}_V$ are the uniform moment bounds of $(X_t,V_t)_t$
\end{theorem}

\begin{proof}
The proof is given in Section~\ref{sec:proofs}.
\end{proof}
Given a probability measure $\nu_0$ on $\mathbb{R}^{2d}$ let $\nu_k$ denote the law of the Euler discretization after $k\in\mathbb{N}$ steps. 
Using the strong accuracy result, we bound the distance between $\nu_k$ and the target measure $\mu_{\infty}$.

\begin{theorem} \label{thm:complexity_guarantee}
    Suppose Assumption~\ref{ass}, \eqref{eq:gamma_cond} and \eqref{eq:condition_h} hold. Then for $k\in\mathbb{N}$,
    \begin{align*}
        \mathcal{W}_{\rho}(\mu_{\infty},\nu_k)& \le h\Big(1+\frac{\gamma(1+ 2L\gamma^{-2})}{c}\Big) 20 L \gamma^{-2} \sqrt{d}
        %\Big( (1+1/\alpha)e^{\gamma^4/L^2}(\gamma \mathbf{M}_1+\alpha \mathbf{M}_2)+h\frac{\gamma \alpha}{2}\mathbf{M}_1\Big) %(1-e^{-c})^{-1} \Big(9\frac{L^2 h^2 \gamma^{-2} }{8}\mathbf{C}_X+ 2\gamma^{-1}L h^2\mathbf{C}_V +L\gamma^{-2}\frac{37 h}{8}\sqrt{2\gamma hd}\Big)
          +e^{-chk}\mathcal{W}_{\rho}(\mu_{h,\infty}, \nu_0),
        \\  \mathcal{W}_{1}(\mu_{\infty},\nu_k)& \le \mathbf{N} h\Big(1+\frac{\gamma(1+ 2L\gamma^{-2})}{c}\Big) 20 L \gamma^{-2} \sqrt{d}+\mathbf{M}e^{-chk}\mathcal{W}_{\rho}(\mu_{h,\infty}, \nu_0),
    \end{align*}  
    where $c$ is given in \eqref{eq:c}, $\mathbf{M}$ and $\mathbf{N}$ are given in \eqref{eq:M}-\eqref{eq:N}. %and $\mathbf{N}=f'(R_1)^{-1}\gamma(\epsilon \min(\sqrt{\kappa},\sqrt{1/2)})^{-1}$.
\begin{proof}
The proof is given in Section~\ref{sec:proofs}.
\end{proof}
    
\end{theorem}

\begin{remark}[Complexity guarantees] \label{remark:complexity_euler} To obtain an $\varepsilon$-accuracy in $\mathcal{W}_1$ distance, we have to choose $h\propto \varepsilon/ \sqrt{d} $ and the number of steps $k$ of order $k\propto \log(\Delta(0)/\varepsilon)/(ch)$. Here $\Delta(0)=\mathcal{W}_{1} (\mu_{h,\infty}, \nu_0)$. Since in each step there is one gradient evaluation, the number of gradient evaluation for $\varepsilon$-accuracy is of order $\sqrt{d}/\epsilon$.
\end{remark}

\begin{remark}[Particle model and propagation of chaos]\label{rem:particle_comp}
Bounds on the strong accuracy can also be considered for the particle model (see Remark~\ref{rem:particle}). Note that due to the fact that a normalized distance $\rho_N$ is considered, the bound between the target measure $\mu_{\infty}^N$ of the particle system and the law after $k$-steps of the discretized Langevin dynamics is independent of the particle number $N$. Further, combining this bound with the propagation of chaos result given in \cite[Theorem 17]{schuh2022global}, it holds under the assumptions in Theorem~\ref{thm:complexity_guarantee} and a smallness assumption for the Lipschitz constant of the gradient of the interaction potential in the particle system that the distance between the law after $k$-steps of the discretized Langevin dynamics with $N$ particles and the stationary measure $\mu_*$ of the limit process given by 
\begin{align} \label{eq:limitmeas}
    \mu_* (\rmd x) \propto \exp\left(-V(x)-\int_{\mathbb R^d}W(x-y)\mu_*(\rmd y)\right) \rmd x,
\end{align}
is bounded by
\begin{align*}
    \mathcal{W}_{\ell^N_1}(\mu_k^N, \mu_*^{\otimes N})\le {C} \Big(e^{-chk} \mathcal{W}_1(\mu_0, \mu_*) + \sqrt{d} h + N^{-1/2}\Big).
\end{align*}
Here, the constant $ C>0$ is independent of $N$, $d$ and $h$.    

\end{remark}

\subsection{Contraction for the BU and UBU scheme}

Consider the Markov chain $(\mathbf{X}_k, \mathbf{V}_k)_{k\in\mathbb{N}}$ generated by the BU discretization scheme.
\begin{theorem}[Contraction for the BU discretization scheme] \label{thm:BU_contr}
    Let $U$ satisfy \Cref{ass}. Let $(x,v), (x',v')\in \mathbb{R}^d$. Let $\gamma >0$ and $h>0$ satisfy
    \begin{align} 
        &\gamma\ge \sqrt{13L_G^2/\kappa} \qquad \text{and} \label{eq:gamma_BU}
        \\  & L\gamma^{-1} h \le \min\Big(\frac{1}{8LR_1^2}, \frac{1}{256\cdot 75(2L\gamma^{-2}+1) }, \frac{L\gamma^{-2}}{15},\frac{L}{55 L_K} \Big). \label{eq:condition_h_BU} 
    \end{align}
    Then there exists a distance $\rho:\mathbb{R}^{2d}\times \mathbb{R}^{2d}\to[0,\infty)$ such that for any $k\in \mathbb{N}$ it holds 
    \begin{align} \label{eq:contr_BU}
        \mathbb{E}[\rho((\mathbf{X}_k,\mathbf{V}_k),(\mathbf{X}_k',\mathbf{V}_k'))]\le (1-ch)^k \rho((x,v), (x',v')) ,
    \end{align}
    where the contraction rate $c$ is given by 
    \begin{align} \label{eq:c_BU}
        c=\min\Big(f'(R_1)\frac{7\epsilon \kappa \gamma^{-1} }{96}\mathcal{E},f'(R_1)\frac{7\epsilon \gamma }{256}\mathcal{E},f'(R_1) e^{-\gamma h}\frac{\gamma}{16},f'(R_1) e^{-\gamma h}\frac{\gamma \alpha}{4},\frac{9\hat{c}}{640},\frac{\hat{c}}{32(4\alpha +1)}  \Big).
    \end{align}
    The construction of the distance function $\rho$ and the constants $R_1$, $\epsilon$, $\mathcal{E}$ and $\alpha$ are independent of $d$ and $h$ and are given in Section~\ref{sec:dist}.
\end{theorem}

\begin{proof}
    The proof is given in \Cref{sec:proofs_UBU}.
\end{proof}

Next, consider the Markov chain $(\mathbf{X}_k, \mathbf{V}_k)_{k\in\mathbb{N}}$ given by the UBU discretization scheme.
\begin{theorem}[Contraction for the UBU discretization scheme] \label{thm:UBU_contr}
    Let $U$ satisfy \Cref{ass}. Let $(x,v), (x',v')\in \mathbb{R}^d$. Let $h>0$ and $\gamma>0$ satisfy \eqref{eq:condition_h_BU} and  \eqref{eq:gamma_BU}.
    Then for any $k\in \mathbb{N}$ it holds 
    \begin{align} \label{eq:contr_UBU}
        \mathbb{E}[\rho((\mathbf{X}_k,\mathbf{V}_k),(\mathbf{X}_k',\mathbf{V}_k'))]\le \mathbf{C}(1-ch)^k \rho((x,v), (x',v')),
    \end{align}
    where the contraction rate $c$ is given by \eqref{eq:c_BU} and the additional constant $\mathbf{C}>0$ is given by
    \begin{align} \label{eq:contr_UBU_C}
        \mathbf{C}= \Big(1+\frac{\gamma h}{16}\Big)\max\Big(\Big(1+\alpha\gamma h/2\Big)^2, 1+\gamma h \max(1, L_K\gamma^{-2})\Big).
    \end{align}
    % \begin{align} \label{eq:c_UBU}
    %     c=...
    % \end{align}
\end{theorem}

\begin{proof}
    The proof is given in \Cref{sec:proofs_UBU}.
\end{proof}

Similarly to the Euler scheme, we obtain convergence in $L^1$-Wasserstein distance for the UBU scheme due to the equivalence of the distance $\rho$ and the Euclidean distance.
Consider two probability measures $\mu_0$, $\nu_0$ on $\mathbb{R}^{2d}$. Let $(\mathbf{X}_0,\mathbf{V}_0)$ and $(\mathbf{X}_0',\mathbf{V}_0')$ be distributed with respect to $\mu_0$ and $\nu_0$, respectively.
Then, we denote by $\mu_k=\Law(\mathbf{X}_k,\mathbf{V}_k)$ and $\nu_k=\Law(\mathbf{X}_k',\mathbf{V}_k')$ the laws of the discretised kinetic Langevin sampler after $k\in\mathbb{N}$ UBU-steps, respectively. 

\begin{corollary}[Convergence in $L^1$-Wasserstein distance]
Suppose Assumption~\ref{ass}, \eqref{eq:gamma_BU} and \eqref{eq:condition_h_BU} hold.
Then, for $k\in\mathbb{N}$
\begin{align*}
    & \mathcal{W}_{\rho}(\mu_k, \nu_k)\le \mathbf{C} \exp(-ckh)  \mathcal{W}_{\rho}(\mu_0, \nu_0),
    \\ & \mathcal{W}_{1}(\mu_k, \nu_k)\le \mathbf{C}\mathbf{M}\exp(-ckh)  \mathcal{W}_{1}(\mu_0, \nu_0) ,
\end{align*}
where $c$ is given by \eqref{eq:c}, $\mathbf{C}>0$ is given in \eqref{eq:contr_UBU_C}, and $\mathbf{M}$ is given in \eqref{eq:M}. Moreover, existence of a unique invariant measure $\mu_{h,\infty}$ and convergence towards its holds, i.e., for $k \in \mathbb{N}$
\begin{align*}
    & \mathcal{W}_{\rho}(\mu_k, \mu_{h,\infty})\le \mathbf{C}\exp(-ckh)  \mathcal{W}_{\rho}(\mu_0, \mu_{h,\infty}),
    \\ & \mathcal{W}_{1}(\mu_k, \mu_{h,\infty})\le \mathbf{C}\mathbf{M}\exp(-ckh)  \mathcal{W}_{1}(\mu_0, \mu_{h,\infty}).
\end{align*}
    
\end{corollary}

\begin{remark}
    As for the Euler scheme, we observe that in the strong convex case, i.e., $R=0$, the rate $c$ reduces to $c=\min(\kappa \gamma^{-1}/24, \gamma/64)$ and the constant $\mathbf{C}$ is given by $\mathbf{C}=1+2\gamma h \max(1, L_K \gamma^{-2})$. Assuming that $L_G$ satisfies $L_G\le \sqrt{3/104}\kappa$ and choosing $\gamma=\sqrt{3/8\kappa}$, we obtain a contraction rate of order $\sqrt{\kappa}$. Note that compared to the convergence result of the continuous Langevin dynamics \cite[Remark 2]{schuh2022global}, we only loose a constant prefactor in the contraction rate and in the condition for the smallness of the perturbation to the Gaussian case.
\end{remark}

\subsection{Assumptions for error analysis}

Assumption \ref{ass} is the only condition necessary for the convergence analysis and order one bias estimates for the numerical schemes we consider. However to establish higher-order bounds with an improved dimension dependence we introduce Assumptions \ref{ass:hessian_lip} and \ref{ass:strong_hessian_lip} and some motivating examples. We will state the main results in the paper in cases where each of the Assumptions is satisfied.

\begin{assumption}\label{ass:hessian_lip}
The potential $U:\mathbb{R}^{d} \to \mathbb{R}$ is three times continuously differentiable and there exists $L_{1} > 0$ such that for all $x,y \in \mathbb{R}^{d}$,
\[
\|\nabla^{2}U(x) - \nabla^{2}U(y)\| \leq L_{1}\|x-y\|,
\]
this implies that
\[
\|\nabla^{3}U(x)[v,v']\| \leq L_{1}\|v\|\|v'\|,
\]
which was used in \cite{sanz2021wasserstein}.
\end{assumption}

Assumption \ref{ass:hessian_lip} is not strong enough to achieve an improved dimension dependence, which is observed in many applications of interest (see \cite{chada2023unbiased}). A stronger assumption can be used, specifically, the strongly Hessian Lipschitz property introduced in \cite{chen2023does} and used in \cite{paulin2024}, which uses the following tensor norm.
\begin{definition}
    For $A \in \mathbb{R}^{d \times d \times d}$, let us define
    \[
    \|A\|_{\{1,2\}\{3\}} = \sup_{x\in \mathbb{R}^{d\times d}, y\in \mathbb{R}^d}\left\{\left.\sum^{d}_{i,j,k=1}A_{ijk}x_{ij}y_{k} \right| \sum_{i,j=1}^{d}x^{2}_{ij}\leq 1, \sum^{d}_{k=1}y^{2}_{k}\leq 1\right\}.
    \]
\end{definition}

\begin{assumption}\label{ass:strong_hessian_lip}
$U:\mathbb{R}^{d} \to \mathbb{R}$ is three times continuously differentiable and strongly Hessian Lipschitz if there exists a $L^{s}_{1} > 0$ such that
    \[
    \|\nabla^{3}U(x)\|_{\{1,2\}\{3\}} \leq L^{s}_{1}
    \]
    for all $x \in \mathbb{R}^{d}$.    
\end{assumption}
% \kati{$M_1=L_1$ and $M_1^s=L_1^s$? Both notations would be fine, but we should be consistent. }
It is easy to show that Assumption \ref{ass:hessian_lip} is equivalent to a uniform bound on the matrix norm defined by
\[
\|\nabla^{3}U(x)\|_{\{1\},\{2\},\{3\}} :=  \sup\left\{\left.\sum^{d}_{i,j,k=1}[\nabla^{3}U(x)]_{ijk}x_{i}y_{j}z_{k} \right| \sum_{i}x_i^2\leq 1,\sum_{j}y_j^2\le 1, \sum_{k}z_k^2\le 1\right\}
\]
for all $x \in \mathbb{R}^{d}$. Due to \cite[Lemma 8]{paulin2024} we have the following equivalency relationship of the norms
\[
\|\cdot\|_{\{1\},\{2\},\{3\}} \leq \|\cdot\|_{\{1,2\},\{3\}} \leq \sqrt{d}\|\cdot\|_{\{1\},\{2\},\{3\}}.
\]
\begin{remark}
    [Examples]
    It is easy to show that Assumption \ref{ass:strong_hessian_lip} is satisfied without additional dimension dependency for product distributions as the tensor $\nabla^{3}U(\cdot)$ is diagonal.

    In \cite[Section 6]{chen2023does} they introduce a wide range of applications which have a small strongly Hessian Lipschitz constant including Bayesian statistical models such as Bayesian ridge regression and generalised linear models. They also show logistic regression problems and two-layer neural networks satisfy Assumption \ref{ass:strong_hessian_lip} with a small constant. An explicit estimate of the constant for Bayesian multinomial regression found in \cite[Lemma H.6]{chada2023unbiased}. The two-layer neural network problem is of particular interest as an application of interacting particle system-based methods (see \cite{mei2018mean,Hu2021,bou2023nonlinear}). % (also see \cite[Section 1.6.3]{schuh2022convergence}).

    In the case of interacting particle systems, examples of interaction potential taken from Molecular dynamics problems which satisfy Assumption \ref{ass:hessian_lip} are the Morse potential and the harmonic bonding potential (see \cite[Chapter 1]{leimkuhler2015molecular}).
\end{remark}

\subsection{Complexity guarantees for the UBU scheme}
Under the respective assumptions we bound the distance between the invariant measure $\mu_{h,\infty}$ of the UBU scheme and the target measure $\mu_{\infty}$. 

\begin{theorem}\label{thm:compl_UBU}
    Let $U$ satisfy Assumption \ref{ass} and let $h>0$ satisfy \eqref{eq:condition_h_BU} and $\gamma > 0$ satisfy \eqref{eq:gamma_BU} then
    \begin{align*}
        &\mathcal{W}_{\rho}(\mu_{\infty},\mu_{h,\infty})\leq \frac{15\mathbf{C}h\sqrt{d}}{c}\left(\gamma^{-1}L + \alpha (5L + \gamma L^{1/2})h\right),
    \end{align*}
    if $U$ satisfies additionally Assumption \ref{ass:hessian_lip} then
    \begin{align*}
        &\mathcal{W}_{\rho}(\mu_{\infty},\mu_{h,\infty})\leq \frac{12{\mathbf{C}}h^{2}\sqrt{d}}{c}\left(2\gamma^{-1}(\sqrt{3}L_{1}\sqrt{d}+L^{3/2} + \gamma L) + \alpha\left(5L + \gamma L^{1/2}\right)\right),
    \end{align*}
    and further if $U$ satisfies Assumption \ref{ass:strong_hessian_lip} then
    \begin{align*}
        &\mathcal{W}_{\rho}(\mu_{\infty},\mu_{h,\infty})\leq \frac{12{\mathbf{C}}h^{2}\sqrt{d}}{c}\left(2\gamma^{-1}(\sqrt{3}L^{s}_{1}+L^{3/2} + \gamma L) + \alpha\left(5L + \gamma L^{1/2}\right)\right),
    \end{align*}
    for distance function $\rho$ given in Section \ref{sec:dist}, contraction rate $c$ defined by \eqref{eq:c_BU} and preconstant $\mathbf{C}$ defined by \eqref{eq:contr_UBU_C}.
\end{theorem}

Let $\nu_k$ denote the law of the UBU scheme after $k$ steps with initial distribution $\nu_0$ on $\mathbb{R}^{2d}$. By the previous result, we bound the distance between $\nu_k$ and $\mu_{\infty}$.

\begin{theorem}\label{thm:compl_UBU_2}
    Let $U$ satisfy Assumption \ref{ass} and let $h>0$ satisfy \eqref{eq:condition_h_BU} and $\gamma > 0$ satisfy \eqref{eq:gamma_BU} then
    \begin{align*}
        &\mathcal{W}_{1}(\nu_k,\mu_{h,\infty})\leq \frac{15\mathbf{C}\mathbf{N} h\sqrt{d}}{c}\left(\gamma^{-1}L + \alpha (5L + \gamma L^{1/2})h\right)+ \mathbf{M} \mathbf{C}e^{-chk} \mathcal{W}_{1}(\nu_0, \mu_{h,\infty}),
    \end{align*}
    if $U$ satisfies additionally Assumption \ref{ass:hessian_lip} then
    \begin{align*}
        \mathcal{W}_{1}(\nu_k,\mu_{h,\infty})&\leq \frac{12{\mathbf{C}}\mathbf{N}h^{2}\sqrt{d}}{c}\left(2\gamma^{-1}(L_{1}\sqrt{d}+L^{3/2} + \gamma L) + \alpha\left(5L + \gamma L^{1/2}\right)\right)
        \\ & + \mathbf{M}\mathbf{C}e^{-chk} \mathcal{W}_{1}(\nu_0, \mu_{h,\infty}),
    \end{align*}
    and further if $U$ satisfies Assumption \ref{ass:strong_hessian_lip} then
    \begin{align*}
        \mathcal{W}_{1}(\nu_k,\mu_{h,\infty})&\leq \frac{12\mathbf{C}\mathbf{N}h^{2}\sqrt{d}}{c}\left(2\gamma^{-1}(L^{s}_{1}+L^{3/2} + \gamma L) + \alpha\left(5L + \gamma L^{1/2}\right)\right)
        \\ & + \mathbf{M}\mathbf{C}e^{-chk} \mathcal{W}_{1}(\nu_0, \mu_{h,\infty}),
    \end{align*}
    for distance function $\rho$ given in Section \ref{sec:dist}, contraction rate $c$ defined by \eqref{eq:c_BU}, the preconstants $\mathbf{C}$, $\mathbf{M}$ and $\mathbf{N}$ are defined by \eqref{eq:contr_UBU_C},\eqref{eq:M} and \eqref{eq:N} respectively.

%    \kati{complexity analysis as in the Euler scheme is missing}
%    \pete{I have added this below, similarly to what you did for the Euler scheme, thank you.\\
%    Should we state the above theorem in terms of $\mathcal{W}_{1}$ as well, as in the Euler case? }
    % \kati{thank you! I changed the statement in the theorem in terms of the $W_1$ distance, since this is easier for the reader to compare it with other results.}
\end{theorem}

\begin{remark}[Complexity guarantees]\label{remark:complexity_UBU} We establish the same complexity guarantees as the Euler scheme when we do not have additional smoothness (only Assumption \ref{ass} holds). When Assumption \ref{ass:hessian_lip} holds by the same reasoning as Remark \ref{remark:complexity_euler}, if we choose $h \propto \sqrt{\epsilon/d}$  and the number of steps $k$ of order $k\propto \log(\Delta(0)/\varepsilon)/(ch)$, where $\Delta(0)=\mathcal{W}_{1} (\mu_{h,\infty}, \nu_0)$. We then achieve a $\varepsilon$-accuracy in the order of $\sqrt{d/\epsilon}$ steps. If additionally Assumption \ref{ass:strong_hessian_lip} holds we can achieve  $\varepsilon$-accuracy in an improved order of $d^{1/4}/\sqrt{\epsilon}$ steps.

\end{remark}

\begin{remark}[Particle model and propagation of chaos for the UBU scheme]
    As for the Euler scheme (see Remark~\ref{rem:particle_comp}), we can also get bounds on the strong accuracy for the UBU scheme applied to the particle model. Together with the propagation of chaos result given in \cite[Theorem 17]{schuh2022global}, it is possible to bound under the assumptions in Theorem~\ref{thm:compl_UBU} and a smallness assumption on the interaction potential the distance between the law $\mu_k^N$ of the particle system with $N$ after $k$-UBU steps and the measure $\mu_*$ given in \eqref{eq:limitmeas} by 
    \begin{align*}
        \mathcal{W}_{\ell^N_1}(\mu_k^N, \mu_*^{\otimes N})\le {C} \Big(e^{-chk} \mathcal{W}_1(\mu_0, \mu_*) + \sqrt{d} h^2 + N^{-1/2}\Big)
    \end{align*}
    for some constant $ C>0$ which is independent of $N$, $d$ and $h$ and rate $c>0$ which is up to a constant factor of the same form as the rate in \eqref{eq:c_BU} . We note that this bound has a better order in $h$ compared to the complexity bounds of nonlinear HMC given in \cite{bou2023nonlinear}.
\end{remark}

\subsection{Numerical illustration of the coupling}

We implement synchronous and reflection coupling as defined in the analysis (see Section \ref{sec:BU_coupling}) for two model problems and the BU scheme (and UBU equivalently). These model problems were also considered in \cite{bou2020coupling,bou2023convergence} and contour plots of their respective potentials are given in Figure \ref{fig:contour}. First of which is the banana-shaped potential, which is defined for $(x,y) \in \mathbb{R}^{2}$ by 
\begin{align*}
U(x,y) = (1 - x)^2 + 10(y - x^2)^2,
\end{align*}
which is unimodal, but has a very flat minima and does not satisfy the strong-convexity assumption. The second is a 10-mode Gaussian mixture model where each of the ten Gaussian distributions has standard deviation $\sigma = 0.5$ and mean  given as in \cite[Table 1]{Liang2001}. 

In Figure \ref{fig:banana} and Figure
\ref{fig:GMM} the contraction properties are illustrated for the two model cases under synchronous and reflection coupling. Interestingly we observe very slow convergence for synchronous coupling on the Banana potential model for small values of the friction parameter, in particular when the coupled chains are close together in the very flat basin. However, when reflection coupling is used this is not the case. The initial bump in \ref{fig:banana} corresponds to the additional prefactor we have in the convergence results in Wasserstein distance with respect to the Euclidean distance. After this initial phase, we exhibit exponential convergence, which is illustrated in Figure \ref{fig:banana} and \ref{fig:GMM} on the log-scale, we remark that at small distances on the log-scale the estimators have very high relative variance, hence noise is present even after 100,000 independent runs.

\begin{figure*}[t!]
    \centering
    \begin{subfigure}[t]{0.5\textwidth}
        \centering
        \includegraphics[height=2.0in]{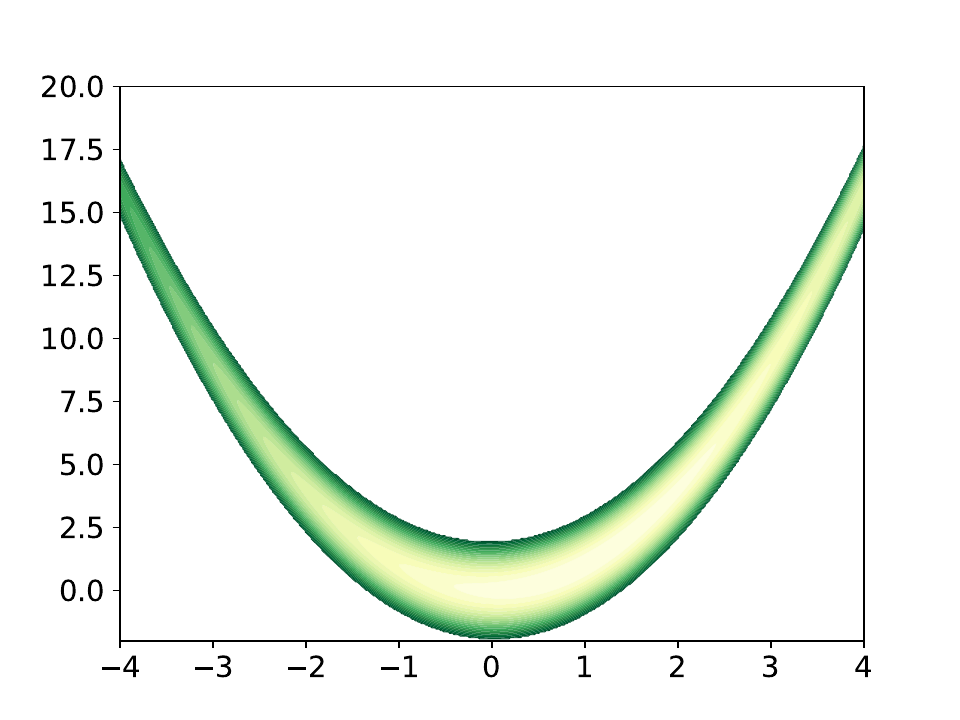}
        \caption{}
    \end{subfigure}%
    ~ 
    \begin{subfigure}[t]{0.5\textwidth}
        \centering
        \includegraphics[height=2.0in]{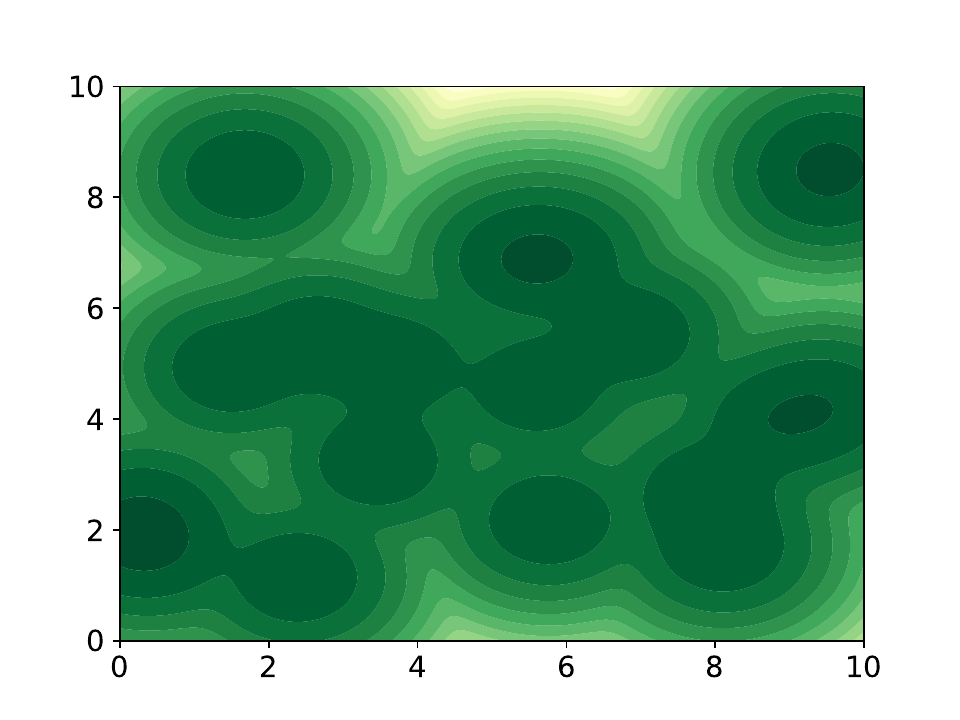}
        \caption{}
    \end{subfigure}
    \caption{a) Contour plot of banana potential, where area outside of basin has been coloured white. b) Contour plot of potential for the Gaussian mixture model.}
    \label{fig:contour}
\end{figure*}

% \begin{figure}[H]
% 	\centering
%         \begin{subfigure}
%             \includegraphics[width=0.45\textwidth]{banana.pdf}
%         \caption{Contour plot of banana potential, where area outside of basin has been coloured white.}
%         \end{subfigure}
%         \begin{subfigure}
%             \includegraphics[width=0.45\textwidth]{contour_multimodal.pdf}
%         \caption{Contour plot of potential for the Gaussian mixture model.}
%         \end{subfigure}
% \end{figure}
\begin{figure}[H]
	\centering
	\includegraphics[width=0.45\textwidth]{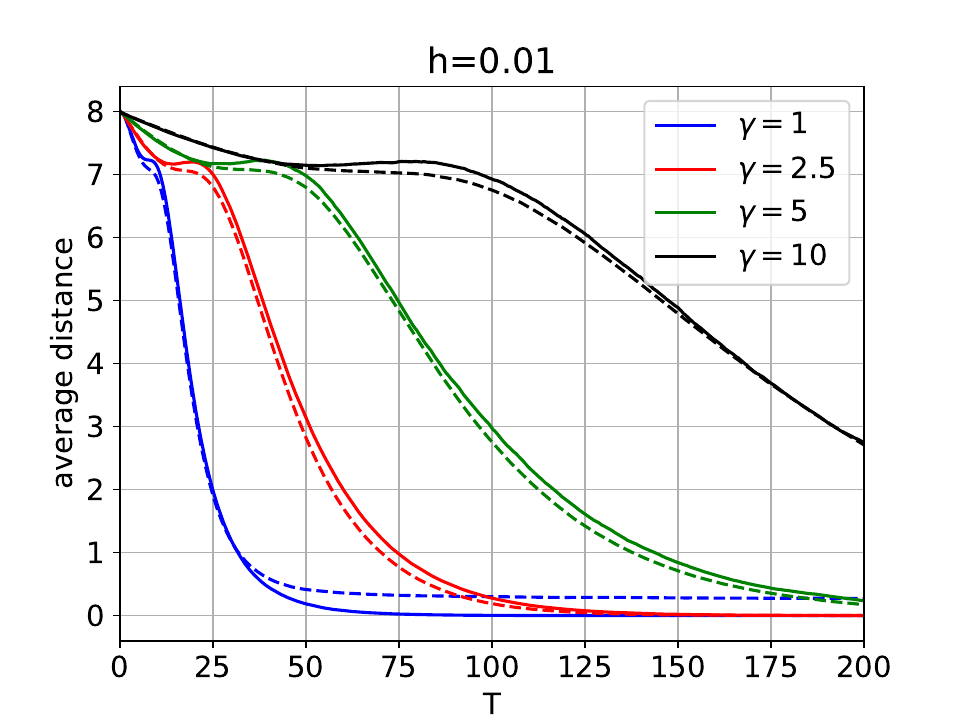}
        \includegraphics[width=0.45\textwidth]{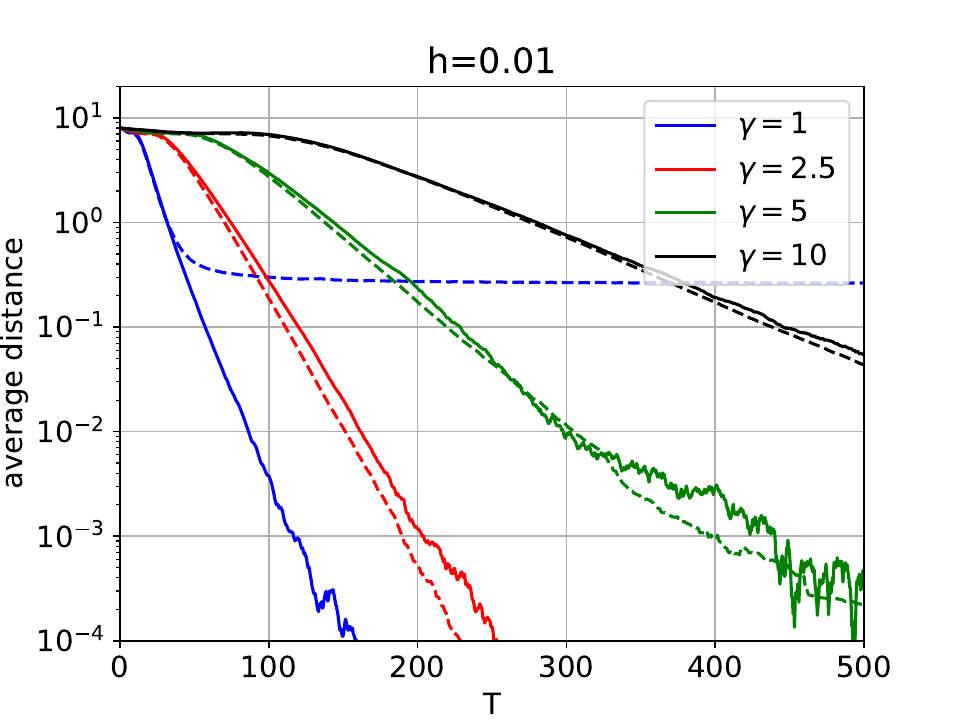}
        \caption{Banana potential: Bold lines on the plot are where reflection coupling is used and dashed lines are where synchronous coupling is used. Initializing two trajectories coupled using synchronous coupling and our reflection coupling construction for the BU scheme initialized at $[4,16]$ and $[-4,16]$ respectively for the banana potential model. Plotting the average distance between the trajectories versus time (number of iterations multiplied by stepsize) where we have averaged the results over 100,000 independent runs. Different colours correspond to different values of the friction parameter $\gamma >0$, which are provided in the legend.}
        \label{fig:banana}
\end{figure}

\begin{remark}
    In Figure \ref{fig:banana} for small values of the friction, close to the Hamiltonian regime, synchronous coupling performs very poorly. The basin is flat and the coupled dynamics exhibit very oscillatory behaviour as illustrated in the Figure \ref{fig:low-friction}; a snapshot of the dynamics of two synchronously coupled particles in the low-friction regime. 

    However, reflection coupling performs much better in this regime; a typical coupling event is illustrated in Figure \ref{fig:low-friction}.

    \begin{figure*}[t!]
    \centering
    \begin{subfigure}[t]{0.5\textwidth}
        \centering
        \includegraphics[height=2.0in]{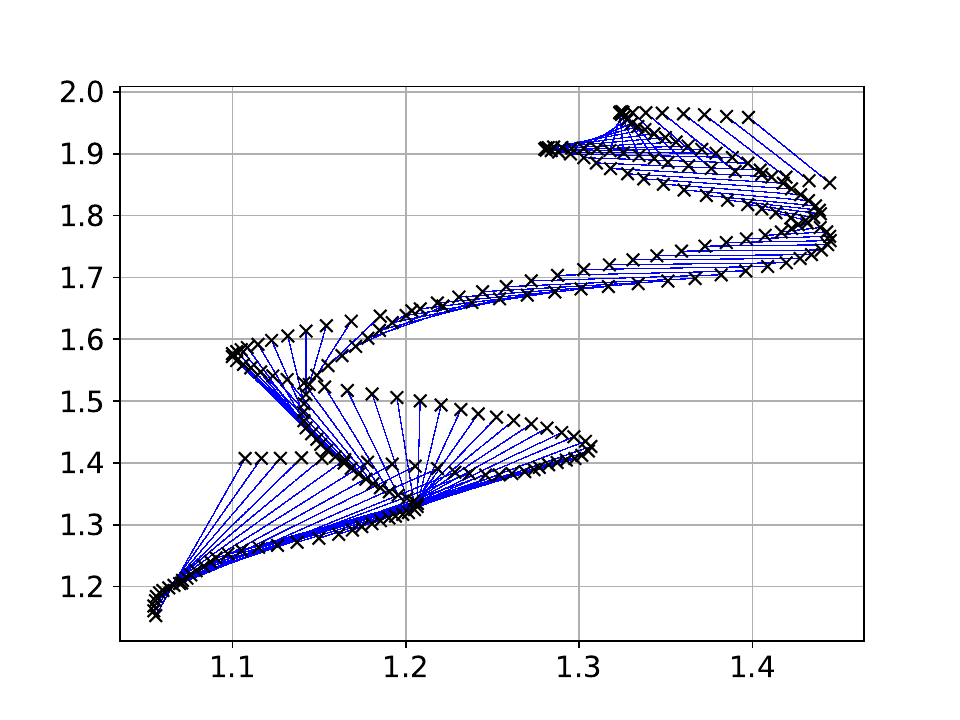}
        \caption{}
    \end{subfigure}%
    ~ 
    \begin{subfigure}[t]{0.5\textwidth}
        \centering
        \includegraphics[height=2.0in]{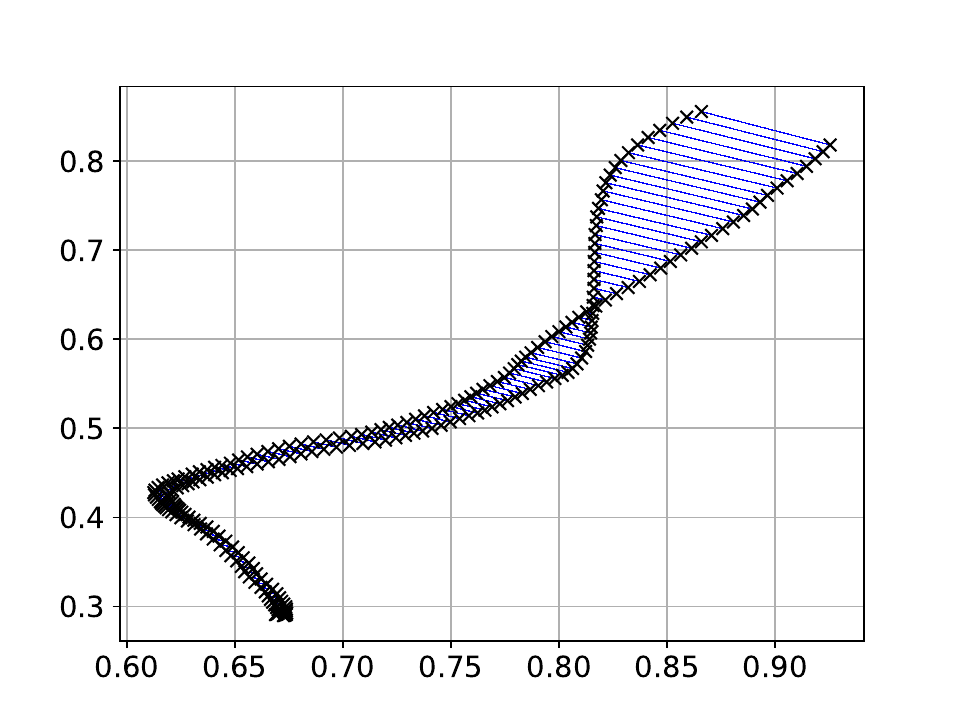}
        \caption{}
    \end{subfigure}
    \caption{Snapshots of the coupled dynamics in the low-friction regime for the Banana potential. a) Synchronously coupled dynamics. b) Reflection coupled dynamics.}
        \label{fig:low-friction}
\end{figure*}

\end{remark}

\begin{figure}[H]
	\centering
	\includegraphics[width=0.45\textwidth]{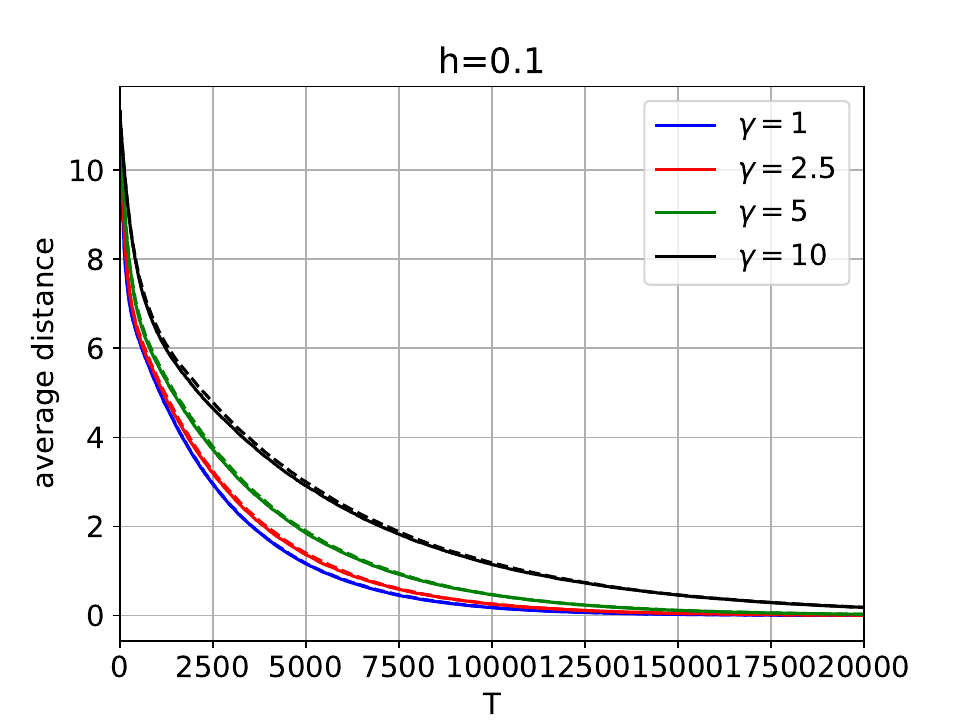}
        \includegraphics[width=0.45\textwidth]{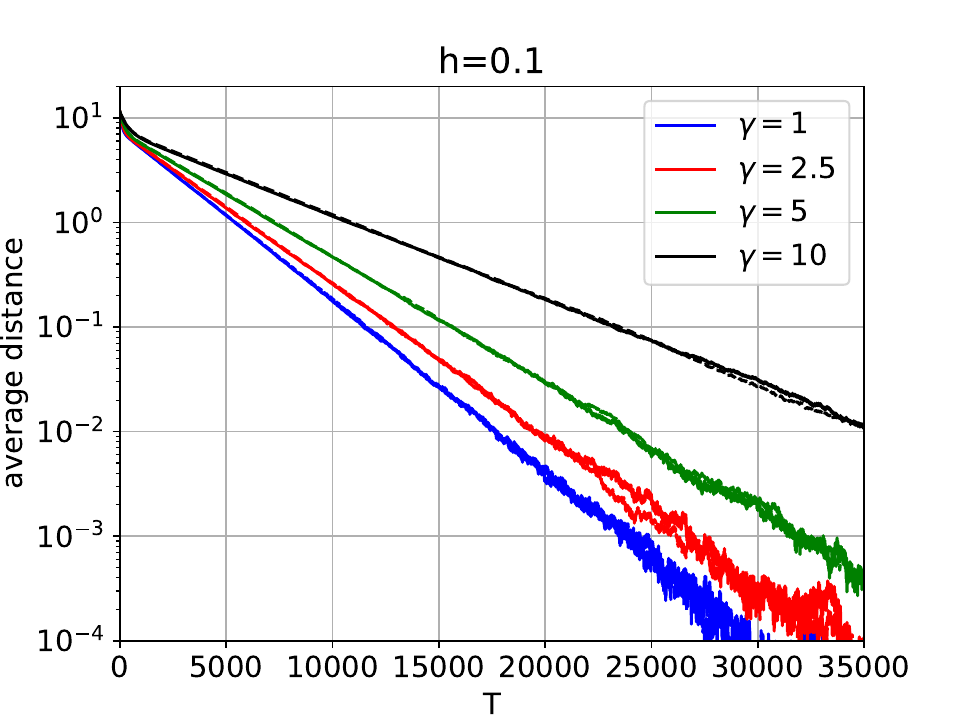}
        \caption{Gaussian Mixture Model: Bold lines on the plot are where reflection coupling is used and dashed lines are where synchronous coupling is used. Initializing two trajectories coupled using synchronous coupling and our reflection coupling construction for the BU scheme initialized at $[1,1]$ and $[9,9]$ respectively for the Gaussian mixture model. Plotting the average distance between the trajectories versus time (number of iterations multiplied by stepsize) where we have averaged the results over 100,000 independent runs. Different colours correspond to different values of the friction parameter $\gamma >0$, which are provided in the legend.}
        \label{fig:GMM}
\end{figure}

\section{Coupling and distance construction}\label{sec:dist_coupl}
\subsection{Distance function} \label{sec:dist}
% \pete{R2 wanted to some intuition for how the distance function is chosen. Perhaps we could refer to Figure \ref{fig:metric} and describe the gluing here as well? What do you think?}
% \kati{I would not repeat the gluing procedure here. But we could refer to it?}
% \pete{That makes sense, good idea.}
As in the construction of  the distance function in \cite[Section 4]{schuh2022global} {(see also \Cref{fig:metric})}, we consider two metrics
$r_l,r_s:\mathbb{R}^{2d}\times \mathbb{R}^{2d}\to [0,\infty)$ given by 
\begin{align*}
r_l((x,v),(x',v'))^2&:=\gamma^{-2} (x-x')^T K (x-x') +\frac{1}{2}|(1-2\tau)(x-x')+\gamma^{-1}(v-v')|^2
\\ & +\frac{1}{2}\gamma^{-2}|v-v'|^2,
\end{align*}
and
\begin{align*}
r_s((x,v),(x',v')):=\alpha |x-x'|+|x-x'+\gamma^{-1}(v-v')|,
\end{align*}
for $(x,v),(x',v')\in\mathbb{R}^{2d}$, where the constant $\tau$ and $\alpha$ are given by 
\begin{align}
&\tau:= \min(1/8, \gamma^{-2}\kappa/4) % \gamma^{-2}\kappa/2-\gamma^{-4}L_G^2) 
\qquad \text{and} \label{eq:tau}
\\ &\alpha:=2L\gamma^{-2}. \label{eq:alpha}
\end{align}
Note that the assumption we impose on $\gamma$ in \Cref{sec:mainresults} guarantees that $\tau>0$.
The matrix $K$ is a positive-definite matrix with smallest eigenvalue $\kappa>0$ and is given by splitting $\nabla U(x)=Kx+\nabla G(x)$, where $G$ is a convex function outside a ball of radius $R$.
Before constructing the metric $\rho:\mathbb{R}^{2d}\times\mathbb{R}^{2d}\to[0,\infty)$ we note that the distances $r_l$ and $r_s$ are equivalent, i.e., it holds
\begin{align}
2\epsilon r_l((x,v),(x',v'))\le r_s((x,v),(x',v')) \le \mathcal{E}^{-1} r_l((x,v),(x',v')) \label{eq:equi_dist}
\end{align}
with
\begin{align}
\epsilon &:= \frac{1}{2}\min(1, \frac{2 \alpha}{3 \sqrt{L_K}\gamma^{-1}},\alpha) \label{eq:epsilon}
\\ \mathcal{E} & := \min(\frac{\sqrt{\kappa}\gamma^{-1}}{\sqrt{8}\alpha},\frac{1}{2}). \label{eq:mathcalE}
\end{align}
{For a detailed calculation, we refer to \cite[Section 4.1]{schuh2022global}.}
Define 
\begin{align*}
\Delta((x,v),(x',v')):=r_s((x,v),(x',v'))-\epsilon r_l((x,v),(x',v')),
\end{align*}
for $(x,v),(x',v')\in\mathbb{R}^{2d}$ and 
\begin{align*}
D_{\mathcal{K}}:=\sup_{ \substack{((x,v),(x',v'))\in\mathbb{R}^{4d} \\ (x-x',v-v') \in\mathcal{K} }} \Delta((x,v),(x',v')),
\end{align*}
where the compact set $\mathcal{K}$ on $\mathbb{R}^{2d}$ is given by
\begin{align*}
\mathcal{K}:=\{(z,w)\in\mathbb{R}^{2d}: \kappa\gamma^{-2}|z|^2+ (1/2)|z+\gamma^{-1} w|^2+(1/2)|\gamma^{-1} w|^2 \le \mathcal{R} \}
\end{align*}
with
\begin{align} \label{eq:mathcalR}
\mathcal{R}:= \tau^{-1}L_G R^2\gamma^{-2}.
\end{align}
Further, we define the constant $R_1$ by
\begin{align} \label{eq:R_1}
R_1:=\sup_{ \substack{((x,v),(x',v')): \\ \Delta((x,v),(x',v')) \le D_{\mathcal{K}} }} r_s((x,v),(x',v')).
\end{align}

Then, we define the metric $\rho:\mathbb{R}^{2d}\times\mathbb{R}^{2d}\to[0,\infty)$ by
\begin{align} \label{eq:rho}
\rho ((x,v),(x',v'))= f(\Delta((x,v),(x',v'))\wedge D_{\mathcal{K}}+\epsilon r_l((x,v),(x',v')))
\end{align}
for $(x,v),(x',v')\in\mathbb{R}^{2d}$. 
{We refer to \Cref{fig:metric} to illustrate this construction.}
The function $f$ is an increasing concave function with $f(0)=0$ and is defined by
\begin{align}\label{eq:f}
f(r)=\int_0^r \phi(s)\psi(s) \rmd s,
\end{align}
where 
%\kati{the function $\phi$ needs to be modified and the precise constant needs to be determined in the proof (instead of $\exp (-\frac{\alpha \gamma^2}{4}\frac{(s\wedge R_1)^2}{2})$).}
\begin{align*}
& \phi(s)=\exp \Big(-128 \alpha \gamma^2\frac{(s\wedge R_1)^2}{2}\Big) && \Phi(s)=\int_0^s \phi(x)\rmd x
\\ & \psi(s)= 1-\frac{\hat{c}}{2}\gamma \int_0^{s\wedge R_1} \Phi(x)\phi(x)^{-1} \rmd x	&& \hat{c}=\frac{1}{\gamma \int_0^{R_1} \Phi(s)\phi(s)^{-1} \rmd s}.
\end{align*}
We note that for $r\in[0,R_1)$, the function $f$ satisfies
\begin{align} \label{eq:f_conseq}
%f'(r) \alpha \gamma r + 4\gamma^{-1} f''(r)\le -(1/2)\hat{c}f(r).
f''(r)=-128 \alpha \gamma^2 r f'(r) -\frac{\hat{c}\gamma}{2}\int_0^r \phi(s) \rmd s.
\end{align}
Further, for all $r\ge 0$
\begin{align} \label{eq:f_conseq4}
f'(R_1)r\le f'(r) r \le f(r) \le \Phi(r) \le r,
\end{align}
since $\psi(r)\in[1/2,1]$.
We refer to \cite{schuh2022global} where a proof that $\rho$ defines indeed a metric is given. In particular it holds
\begin{align} \label{eq:dist_equiv}
    |(x,v)-(x',v')|\le \mathbf{N} \rho ((x,v),(x',v'))\le  \mathbf{M}|(x,v)-(x',v')| 
\end{align}
with
\begin{align}
    & \mathbf{M}= f'(R_1)^{-1}\frac{2 \max(\gamma(1+\alpha),1)}{\epsilon\min(\sqrt{2\kappa},1)}  \label{eq:M}
    \\ &\mathbf{N}=f'(R_1)^{-1} \frac{\gamma}{\epsilon \min(\sqrt{\kappa}, \sqrt{1/2})}. \label{eq:N}
\end{align}

Further, observe that in the strong convex case, i.e., $R=0$, the construction of the distances reduces to $\rho((x,v),(x',v'))=r_l((x,v),(x',v'))$.

\subsection{Coupling construction for the Euler scheme}
Next, we define the coupling by using the idea of the coupling construction from \cite{bou2020coupling}. 
Consider two states $(x,v)$ and $(x',v')$. We define the next coupling step of two copies given by \eqref{eq:EM} depending whether the two current positions $(x,v)$ and $(x',v')$ are close to each other or far apart. 

In particular, the chains are coupled through the sequence of random variables $(\xi_k,\xi_k')_{k\in \mathbb{N}}$ on a common probability space and which satisfy $\xi_k,\xi_k'\sim \mathcal{N}(0,I_d)$ for all $k\in\mathbb{N}$.

\paragraph{Synchronous coupling:}
Given $(\mathbf{X}_k,\mathbf{V}_k),(\mathbf{X}_k',\mathbf{V}_k')\in\mathbb{R}^{2d}$. 
If 
\begin{align*}
D_{\mathcal{K}}+\epsilon r_l((\mathbf{X}_k,\mathbf{V}_k),(\mathbf{X}_k',\mathbf{V}_k')) \le  r_s ((\mathbf{X}_k,\mathbf{V}_k),(\mathbf{X}_k',\mathbf{V}_k') ),
\end{align*}
i.e., if the two states are far apart, we take the same random variables $\xi_{k+1}=\xi_{k+1}'$ and the next step $((\mathbf{X}_{k+1},\mathbf{V}_{k+1}),(\mathbf{X}_{k+1}',\mathbf{V}_{k+1}'))$ is given by
\begin{align*}
&\begin{cases}
\mathbf{X}_{k+1}= \mathbf{X}_k+h\mathbf{V}_{k} \\
\mathbf{V}_{k+1}= \mathbf{V}_{k} - h \nabla U(\mathbf{X}_{k})-h\gamma \mathbf{V}_k+\sqrt{2 \gamma h} \xi_{k+1},
\end{cases}
\\&
\begin{cases}
\mathbf{X}_{k+1}'= \mathbf{X}_k'+h\mathbf{V}_{k}' \\
\mathbf{V}_{k+1}'= \mathbf{V}_{k}' - h \nabla U(\mathbf{X}_{k}')-h\gamma \mathbf{V}_k'+\sqrt{2 \gamma h} \xi_{k+1}.
\end{cases}
\end{align*} 

\paragraph{Contractive coupling:}
If $D_{\mathcal{K}}+\epsilon r_l((\mathbf{X}_k,\mathbf{V}_k),(\mathbf{X}_k',\mathbf{V}_k')) >  r_s ((\mathbf{X}_k,\mathbf{V}_k),(\mathbf{X}_k',\mathbf{V}_k') )$, then let $\xi_{k+1}$ be a normally distributed random variable. Let $\mathcal{U}$ be an independent uniformly distributed random variable on $[0,1]$ and $\beta$ be given by 
\begin{align} \label{eq:beta}
\beta=\frac{1}{\sqrt{2\gamma^{-1} h}}. 
\end{align}
We define $\xi_{k+1}'$ by
\begin{align}\label{eq:coupl_contr}
\xi_{k+1}'=\begin{cases} \xi_{k+1}+\beta q_k & \text{if } \mathcal{U}\le \frac{\varphi_{0,1}(e_k\cdot\xi_{k+1}+\beta|q_k|)}{\varphi_{0,1}(e_k\cdot \xi_{k+1})},
\\ \xi_{k+1}-2(e_k\cdot \xi_{k+1})e_k & \text{otherwise},
\end{cases}
\end{align}
where $q_k=\mathbf{X}_k-\mathbf{X}_k'+\gamma^{-1}(\mathbf{V}_{k}-\mathbf{V}_k')$, $e_k=q_k/|q_k|$, and $\varphi_{0,1}$ denotes the density of the standard normal distribution.
We set $\hat{q}_k=\beta q_k$ and 
\begin{align*}
\Xi_{k+1}=\xi_{k+1}-\xi_{k+1}'=\begin{cases} -\hat{q}_k & \text{if } \mathcal{U}\le \frac{\varphi_{0,1}(e_k\cdot\xi_{k+1}+\beta|q_k|)}{\varphi_{0,1}(e_k\cdot \xi_{k+1})},
\\ 2(e_k\cdot \xi_{k+1})e_k & \text{otherwise}.
\end{cases}
\end{align*}
Note that on the line $q_k=0$ this coupling simplifies to a synchronous coupling. This corresponds to the coupling for the time-continuous Langevin dynamics in \cite{eberle2019couplings, schuh2022global}. Moreover, for $h\to 0$ the above-constructed coupling converges to the one in \cite{eberle2019couplings, schuh2022global}. Moreover, if $q_k\neq 0$, it holds $\xi_{k+1}-\xi_{k+1}'=-\beta q_k$ with maximal probability. Otherwise, we consider a reflection coupling, which is reflected at the hyperspace $q_k=0$.
We note that this indeed defines a coupling, see \cite[Section 2.3.2.]{bou2020coupling}.

\subsection{Coupling construction for the BU scheme}\label{sec:BU_coupling}

Similarly, the construction for the coupling for the BU scheme relies on the idea of the coupling from \cite{chak2023reflection}. Consider two states $(x,v)$ and $(x',v')$. The coupled chain $(\mathbf{X}_k,\mathbf{V}_k, \mathbf{X}_k',\mathbf{V}_k')_{k\in \mathbb{N}}$ of two copies of the BU scheme is given by coupling the sequence of random variables $(\xi^{(1)}_k,\xi^{(2)}_k)_{k\in \mathbb{N}}$ and  $(\xi^{(1)'}_k,\xi^{(2)'}_k)_{k\in \mathbb{N}}$ on a common probability space such that $(\xi^{(1)}_k,\xi^{(2)}_k),(\xi^{(1)'}_k,\xi^{(2)'}_k)\sim \mathcal{N}(0_{2d},I_{2d})$ for all $k\in\mathbb{N}$.

Given $(\mathbf{X}_{k},\mathbf{V}_k), (\mathbf{X}_k',\mathbf{V}_k')\in \mathbb{R}^{2d}$. Let $(\xi^{(1)}_{k+1},\xi^{(2)}_{k+1})\sim \mathcal{N}(0_{2d}, I_{2d})$. If 
\begin{align*}
    D_{\mathcal{K}}+\epsilon r_l((\mathbf{X}_k,\mathbf{V}_k), (\mathbf{X}_k',\mathbf{V}_k'))\le r_s((\mathbf{X}_k,\mathbf{V}_k),(\mathbf{X}_k',\mathbf{V}_k')), 
\end{align*}
we couple the random variables synchronously, i.e., we set $\xi^{(1)'}_{k+1}=\xi^{(1)}_{k+1}$ and $\xi^{(2)'}_{k+1}=\xi^{(2)}_{k+1}$.
Hence, for $Z_k=\mathbf{X}_k-\mathbf{X}_k'$ and $W_k=\mathbf{V}_k-\mathbf{V}_k'$ it holds
\begin{align*}
    &\begin{cases}
        Z_{k+1}=Z_k+ \frac{1-\exp(-\gamma h )}{\gamma}W_k -\frac{1-\eta}{\gamma}h(\nabla U(\mathbf{X}_{k})-\nabla U(\mathbf{X}'_{k}))
        \\ W_{k+1}=\exp(-\gamma h ) W_k - h \exp(-\gamma h) (\nabla U(\mathbf{X}_k)-\nabla U(\mathbf{X}_k')).
    \end{cases}
    %\\ & \begin{cases}
    %    \mathbf{X}_{k+1}'=\mathbf{X}_k'+ \frac{1-\exp(-\gamma h )}{\gamma}\mathbf{V}_k'
    %    \\ \mathbf{V}_{k+1}'=\exp(-\gamma h ) \mathbf{V}_k' - h \exp(-\gamma h) \nabla U(\mathbf{X}_k'). 
    %\end{cases}
\end{align*}

If 
$ D_{\mathcal{K}}+\epsilon r_l((\mathbf{X}_k,\mathbf{V}_k), (\mathbf{X}_k',\mathbf{V}_k'))> r_s((\mathbf{X}_k,\mathbf{V}_k),(\mathbf{X}_k',\mathbf{V}_k'))$, 
we set $\xi^{(2)'}_{k+1}=\xi^{(2)}_{k+1}$ and construct $\xi^{(1)'}_{k+1}$ in the following way: Let $\mathcal{U}\sim \text{Unif}[0,1]$ be an independent uniformly distributed random variable and let $\beta$ be given by \eqref{eq:beta}. We define $\xi^{(1)'}$ as in \eqref{eq:coupl_contr}, i.e., 
\begin{align}\label{eq:coupl_contr_BU}
\xi_{k+1}^{(1)'}=\begin{cases} \xi_{k+1}^{(1)}+\beta q_k & \text{if } \mathcal{U}\le \frac{\varphi_{0,1}(e_k\cdot\xi^{(1)}_{k+1}+\beta|q_k|)}{\varphi_{0,1}(e_k\cdot \xi^{(1)}_{k+1})},
\\ \xi_{k+1}^{(1)}-2(e_k\cdot \xi_{k+1}^{(1)})e_k & \text{otherwise},
\end{cases}
\end{align}
where $q_k=\mathbf{X}_k-\mathbf{X}_k'+\gamma^{-1}(\mathbf{V}_{k}-\mathbf{V}_k')$, $e_k=q_k/|q_k|$, and $\varphi_{0,1}$ denotes the density of the standard normal distribution. Further, as for the Euler scheme, we set $\hat{q}_k=\beta q_k$ and the difference of $\xi_{k+1}^{(1)}$ and $\xi_{k+1}^{(1)'}$ satisfies
\begin{align*}
\Xi_{k+1}=\xi_{k+1}^{(1)}-\xi_{k+1}^{(1)'}=\begin{cases} -\hat{q}_k & \text{if } \mathcal{U}\le \frac{\varphi_{0,1}(e_k\cdot\xi_{k+1}^{(1)}+\beta|q_k|)}{\varphi_{0,1}(e_k\cdot \xi_{k+1}^{(1)})},
\\ 2(e_k\cdot \xi_{k+1}^{(1)})e_k & \text{otherwise}.
\end{cases}
\end{align*}
This construction defines a coupling, see \cite[Section 2.3.2.]{bou2020coupling}.
% \pete{I think R3 is referring to the definitions of $\xi$ and $\Xi$ as being repetitive. However, I think it is nice as it is. I think R3 is complete now, I looked through all the points you didn't highlight and I think they are addressed.}

\section{Proofs}\label{sec:proofs}
\subsection{Euler-Maruyama} \label{sec:proofs_euler}

To prove \Cref{thm:contra_EM}, we first show local contraction for the distance $r_l$ if the distance is sufficiently large.  
\begin{proposition} \label{thm:case1}
Let the potential $U$ be of the form $U := x^{T}Kx + G(x)$, where the symmetric and positive definite matrix $K$ satisfies $\kappa I_{d}\prec K \prec L_{K}I_{d}$ and $G$ is convex outside a Euclidean ball, i.e., $(\nabla G(x)-\nabla G(y))\cdot (x-y) \ge 0$ for all $x,y \in \mathbb{R}^{d}$ such that $|x-y| > R$, now consider two iterates of the Euler-Maruyama scheme $(\mathbf{X}_{k},\mathbf{V}_{k})_{k \in \mathbb{N}}$ and $(\mathbf{X}'_{k},\mathbf{V}'_{k})_{k \in \mathbb{N}}$ with synchronously coupled noise increments and metric $r_{l}$ between the iterates.  If $r^{2}_{l}((\mathbf{X}_{k},\mathbf{V}_{k}),(\mathbf{X}'_{k},\mathbf{V}'_{k}))\ge \mathcal{R}$ at iteration $k \in \mathbb{N}$ with $\mathcal{R}$ given in \eqref{eq:mathcalR}, $ h < \min\{\frac{\gamma}{32L_K},\frac{1}{8\gamma }\}$ and $L_G\gamma^{-2}  \le \kappa/(16 L_G)$ we have that
\[
r^{2}_{l}((\mathbf{X}_{k+1},\mathbf{V}_{k+1}),(\mathbf{X}'_{k+1},\mathbf{V}'_{k+1})) \leq (1-\tau\gamma h)r^{2}_{l}((\mathbf{X}_{k},\mathbf{V}_{k}),(\mathbf{X}'_{k},\mathbf{V}'_{k})),
\]
where $\tau = \min\{\frac{\kappa}{4\gamma^{2}},\frac{1}{8}\}$. If $R=0$, $\mathcal{R}=0$ and the restriction on $\gamma$ improves to $(4+\frac{3}{4})L_G \gamma^{-2}\le 1$.
\end{proposition}
Due to controlling the additional discretization error the bound on $\gamma$ is worse than the bound in the continuous dynamics given in \cite{schuh2022global}. For $h$ tending to zero, it is possible to adapt the proof such that in the limit we can actually recover the condition for $\gamma$ from the continuous dynamics.

\begin{proof}
    We have that
\[
r_{l}^{2}((\mathbf{X}_{k+1},\mathbf{V}_{k+1}),(\mathbf{X}'_{k+1},\mathbf{V}'_{k+1})) = (Z_{k},W_{k}) P^{T} M P \cdot (Z_{k},W_{k}),
\]
where 
\[
M = \begin{pmatrix}
  \gamma^{-2} K + (1-2\tau)^{2}/2I_d & (1-2\tau)/2\gamma I_d\\
      (1-2\tau)/2\gamma I_d &  \gamma^{-2}I_d
\end{pmatrix},
\textnormal{ and }
P = \begin{pmatrix}
  I_d & hI_d\\
      -h(K + Q) &  (1-\gamma h) I_d
\end{pmatrix}.
\]
The matrix $K$ is given by the quadratic term in the potential and $Q$ is defined by 
\[
Q = \int^{1}_{t=0}\nabla^{2}G(\mathbf{X}_{k} + t(\mathbf{X}'_{k}-\mathbf{X}_{k}))\rmd t,
\]
where $G$ is the non-quadratic term in the potential and $Q \succ 0$ for $|\mathbf{X}_{k}-\mathbf{X}'_{k}|>R$ and $L_{G}I_{d} \succ Q \succ -L_{G} I_{d}$ otherwise. 
It holds
\begin{align*}
    r_{l}^{2}((\mathbf{X}_{k+1},\mathbf{V}_{k+1}),(\mathbf{X}'_{k+1},\mathbf{V}'_{k+1})) &  = r_l^2((\mathbf{X}_{k},\mathbf{V}_{k}),(\mathbf{X}'_{k},\mathbf{V}'_{k}))
    \\ & +(Z_k,W_k)^T ( h (M P_1+P_1^T M)+ h^2 P_1^T M P_1  ) (Z_k,W_k)
\end{align*}
with 
\begin{align*}
    P_1=  \begin{pmatrix}
        0 & I_d \\ -(K+Q) & -\gamma I_d
    \end{pmatrix} .
\end{align*}
It is sufficient to show that for all $(z,w)\in \mathbb{R}^{2d}$ with $r_l^2((z,w))\ge \mathcal{R}$, $(z,w)^T ( h (M P_1+P_1^T M)+ h^2 P_1^T M P_1  ) (z,w)\le - \gamma \tau h (z,w)^TM(z,w) $.
It holds
\begin{align*}
    h (M P_1+P_1^T M)&= h \begin{pmatrix}
        -\frac{1-2\tau}{\gamma}(K+Q) & -\gamma^{-2}Q-\tau (1-2\tau)I_d \\ -\gamma^{-2}Q-\tau (1-2\tau)I_d& (-2\tau \gamma^{-1} -\gamma^{-1}) I_d
    \end{pmatrix}
    \\ & \prec h \begin{pmatrix}
        -\frac{1-2\tau}{\gamma}(K+Q)+ 2\gamma^{-3} Q^2 & -\tau (1-2\tau)I_d \\ -\tau (1-2\tau)I_d& (-2\tau \gamma^{-1} -\frac{\gamma^{-1}}{2}) I_d
    \end{pmatrix},
\end{align*}
since for all $(z,w)\in \mathbb{R}^{2d}$, $z^T(-Q\gamma^{-2})w\le \gamma^{-3}z^TQ ^2 z + 1/4 \gamma^{-1}  |w|^2 $ and  $w^T(-Q\gamma^{-2})z\le \gamma^{-3}z^TQ ^2 z + 1/4 \gamma^{-1} |w|^2$.
Further,
\begin{align*}
    h^2 P_1^T M P_1 & = h^2 \begin{pmatrix}
        (K+Q)^2\gamma^{-2} & %\frac{-1-\tau}{\gamma}(K+Q) 
        {\frac{1+2\tau}{2\gamma}(K+Q)} \\ %\frac{-1-\tau}{\gamma}(K+Q) 
        {\frac{1+2\tau}{2\gamma}(K+Q)} & K \gamma^{-2} + \frac{1+ 4\tau^2}{2}I_d
    \end{pmatrix}
    \\ & \prec h^2
    \begin{pmatrix}
        (K+Q)^2\gamma^{-2}+ \frac{1}{2\gamma^{2}}(K+Q)^2 & 0 \\ 0 & K \gamma^{-2} + \frac{3+ 4\tau^2}{2}I_d
    \end{pmatrix}
    \\ & \prec h^2
    \begin{pmatrix}
        (3/2)(K+Q)^2\gamma^{-2} & 0 \\ 0 & K \gamma^{-2} + 2I_d
    \end{pmatrix},
\end{align*}
since { $z^T \frac{1+2\tau}{2\gamma}(K+Q) w \le \frac{(1+\tau)^2}{8\gamma^{2}}z^T(K+Q)^2z+ \frac{1}{2}|w|^2$ and $\tau \le 1/8$}. Further, we observe that $h^2 (3/2) (K+Q)^2\prec h^2 (3 K^2  + 3 Q^2) \prec h^2 (3 L_K K + 3 Q^2)$, where we used in the second step that $K$ is symmetric and positive definite and hence we can take the square root of $K$. Putting the previous estimates together, we obtain
\begin{align*}
    & h (M P_1+P_1^T M)+ h^2 P_1^T M P_1 
    \\ & \prec h \begin{pmatrix}
        \frac{1-2\tau}{\gamma}(-K-Q)+ 2\gamma^{-3} Q^2 & -\tau (1-2\tau)I_d \\ -\tau (1-2\tau)I_d& (-2\tau \gamma^{-1} -\frac{\gamma^{-1}}{2}) I_d
    \end{pmatrix} + h^2
    \begin{pmatrix}
        \frac{3L_K}{\gamma^2} K + \frac{3}{ \gamma^2 } Q^2  & 0 \\ 0 & \frac{K}{\gamma^2} + 2I_d
    \end{pmatrix}.
\end{align*}
By the condition on $h$ it holds $ h^2 \frac{3 L_K}{\gamma^2}K\prec \frac{3(1-4\tau)}{64\gamma}K $, $h(L_K \gamma^{-2} +2)\le \frac{1}{2\gamma}$ and $3 h L_G^2 \gamma^{-2} \le \frac{3}{8} \gamma^{-3} L_G^2$, and hence
% \pete{should it be $\frac{(1-4\tau)}{4\gamma}K$ below?}
% \kati{yes, you are right!}
\begin{align*}
    & h (M P_1+P_1^T M)+ h^2 P_1^T M P_1 
    \\ & \prec h \begin{pmatrix}
        \frac{(1-2\tau)}{\gamma}(-K-Q)+\frac{3(1-4\tau)}{64\gamma}K + (2+\frac{3}{8})\gamma^{-3} L_G^2 I_d & -\tau (1-2\tau)I_d \\ -\tau (1-2\tau)I_d& -2\tau \gamma^{-1} I_d
    \end{pmatrix}.
\end{align*}
By assumption on $G$ and $\gamma$ and the choice of $\tau$, we observe  
\begin{align*}
    \frac{(1-2\tau)}{\gamma}z^T (-Q) z \le \frac{(1-2\tau)}{\gamma} L_G \1_{|z|\le R}|z|^2 \le L_G \gamma^{-1} R^2
\end{align*}
and 
% \pete{should it be $-2\tau\gamma\frac{(1-2\tau)^{2}}{2}$ below?}
% \kati{yes, you are right! thanks}
\begin{align*}
    -\frac{61(1-4\tau)}{64\gamma}z^T K z +\frac{19}{8}\gamma^{-3} L_G^2|z|^2 \le -\frac{61}{128\gamma}z^T K z +\frac{19}{8}\gamma^{-3} L_G^2|z|^2&\le -\frac{1}{4\gamma}\kappa |z|^2 
    \\ & \le -2 \tau \gamma \frac{(1-2\tau)^2}{2}|z|^2,
\end{align*}
since by \eqref{eq:tau}, $\tau \gamma(1-2\tau)^2\le \tau \gamma\le \kappa/(4\gamma)$ in the last step.
Hence, 
\begin{align*}
    &(z,w)^T\Big(h (M P_1+P_1^T M)+ h^2 P_1^T M P_1\Big) (z,w)
    \\ & \le h(z,w)^T  \begin{pmatrix}
        -2\tau \gamma \Big( \gamma^{-2}K +\frac{(1-2\tau)^2}{2}\Big)  & -\tau (1-2\tau)I_d \\ -\tau (1-2\tau)I_d& -2\tau \gamma^{-1} I_d
    \end{pmatrix} (z,w) + hL_G \gamma^{-1} R^2
    \\ & \le -2\tau \gamma h  (z,w)^TM(z,w)+ hL_G \gamma^{-1} R^2 \le -\tau \gamma h  (z,w)^TM(z,w),
\end{align*}
where the last step holds by the definition of $\mathcal{R}$ given in \eqref{eq:mathcalR}. We note that for $R=0$, $Q$ is positive definite and we can bound directly $-\frac{1-2\tau}{\gamma}Q+(2+\frac{3}{8})\gamma^{-3} Q^2\prec 0$, which yields the result.
\end{proof}

\begin{proof}[Proof of Theorem~\ref{thm:contra_EM}]
Fix $k\in \mathbb{N}$. Consider $\mathbf{X}_k,\mathbf{V}_k,\mathbf{X}_k',\mathbf{V}_k'\in \mathbb{R}^d$.
We write $Z_k=\mathbf{X}_k-\mathbf{X}_k'$, $W_k=\mathbf{V}_k-\mathbf{V}_k'$ and $q_k=Z_k+\gamma^{-1}W_k$.

\textit{Case 1:} If $D_{\mathcal{K}}+\epsilon r_l((\mathbf{X}_k,\mathbf{V}_k),(\mathbf{X}_k',\mathbf{V}_k')) \le  r_s ((\mathbf{X}_k,\mathbf{V}_k),(\mathbf{X}_k',\mathbf{V}_k') )$ holds, then the distance satisfies $\rho((\mathbf{X}_k,\mathbf{V}_k),(\mathbf{X}_k',\mathbf{V}_k'))=f(D_{\mathcal{K}}+\epsilon r_l((\mathbf{X}_k,\mathbf{V}_k),(\mathbf{X}_k',\mathbf{V}_k')))$ and we consider a synchronous coupling, i.e., 
$\xi_{k+1}=\xi_{k+1}'$. 
We observe
\begin{align*}
\rho_{k+1}:=\rho((\mathbf{X}_{k+1},\mathbf{V}_{k+1}),(\mathbf{X}_{k+1}',\mathbf{V}_{k+1}'))\le f(D_{\mathcal{K}}+\epsilon r_l(k+1)), 
\end{align*}
where we abbreviate $r_l(k+1)=r_l((\mathbf{X}_{k+1},\mathbf{V}_{k+1}),(\mathbf{X}_{k+1}',\mathbf{V}_{k+1}'))$. 
By Proposition~\ref{thm:case1} it holds
\begin{align*}
r_l(k+1) & \le \sqrt{1-\tau \gamma h}r_l(k) \leq \left(1-\frac{\tau \gamma h}{2}\right)r_l(k).
\end{align*}
Note that the condition on $h$ follows directly from \eqref{eq:condition_h}.
Hence, since $f$ is concave and $f'(R_1)\le f'(D_{\mathcal{K}}+\epsilon r_l(k))$,
\begin{equation} \label{eq:contr_largedist}
\begin{aligned}
\mathbb{E}[\rho_{k+1}-\rho_k]&\le \mathbb{E}[f(D_{\mathcal{K}}+\epsilon r_l(k+1))-f(D_{\mathcal{K}}+\epsilon r_l(k))]
\\ & \le f'(R_1)\epsilon\mathbb{E}[r_l(k+1)-r_l(k)]
 \le -f'(R_1) c_1 h r_l(k),
\end{aligned}
\end{equation}
where $c_{1} = \epsilon \tau \gamma/2$. By \eqref{eq:equi_dist} and $f(r) \le r$,
\begin{align*}
-f'(R_1) c_1 h r_l(k)\le -f'(R_1)h \frac{c_1 r_l(k)}{D_{\mathcal{K}}+\epsilon r_l(k)}\rho_k\le -f'(R_1) c_1 h \mathcal{E}\rho_k.
\end{align*}
Hence, 
\begin{align*}
    \mathbb{E}[\rho_{k+1}]\le (1-f'(R_1) c_1 \mathcal{E} h) \rho_k=\Big(1-h\min\Big(f'(R_1) \epsilon \mathcal{E}  \frac{\gamma}{16},f'(R_1) \epsilon \mathcal{E} \frac{\kappa\gamma^{-1}}{8}\Big)\Big) \rho_k.
\end{align*}

\textit{Case 2:} If $D_{\mathcal{K}}+\epsilon r_l((\mathbf{X}_k,\mathbf{V}_k),(\mathbf{X}_k',\mathbf{V}_k')) >  r_s ((\mathbf{X}_k,\mathbf{V}_k),(\mathbf{X}_k',\mathbf{V}_k') )$, the coupling \eqref{eq:coupl_contr} is applied and $\rho((\mathbf{X}_k,\mathbf{V}_k),(\mathbf{X}_k',\mathbf{V}_k'))=f( r_s((\mathbf{X}_k,\mathbf{V}_k),(\mathbf{X}_k',\mathbf{V}_k')))$.
To show contraction for small distances, the proof is divided in three steps, i.e., first we consider $|Z_k|\ge 4|q_k|$, then $|Z_k|< 4|q_k|$ and $|q_k|\ge \sqrt{2\gamma^{-1} h}$ and finally $|Z_k|< 4|q_k|$ and $|q_k|< \sqrt{2\gamma^{-1} h}$. In the following we abbreviate $r_s(k)=r_s((\mathbf{X}_k,\mathbf{V}_k),(\mathbf{X}_k',\mathbf{V}_k'))$.

\textit{Step 1:} Let $|Z_k|\ge 4|q_k|$. Note that
\begin{align*}
    \alpha|Z_{k+1}|&-\alpha|Z_k|+(1/2)\gamma h \alpha |Z_k|
    \\ & = \alpha[|(1-\gamma h) Z_k + \gamma h (Z_k + \gamma^{-1} W_k)|- (1-\gamma h) |Z_k |- (h\gamma/2)|Z_k|]
    \\ & \le \gamma h \alpha|q_k| - (\alpha h \gamma /2)|Z_k|.
\end{align*}
By concavity of $f$, \Cref{ass}, \eqref{eq:EM} and \eqref{eq:alpha}
\begin{align}
\mathbb{E}[\rho_{k+1}-\rho_k]&\le \mathbb{E}[f(r_s(k+1))-f(r_s(k))]
\le f'(r_s(k)) \mathbb{E}[r_s(k+1)-r_s(k)] \nonumber
\\& \le f'(r_s(k)) \mathbb{E}\Big[-\frac{h\gamma}{2}\alpha |Z_k|+ \alpha h \gamma |q_k| + \sqrt{2 \gamma^{-1}  h}|\hat{q}_k + \Xi_{k+1}|-|q_k|\Big] \nonumber
\\ & \le f'(r_s(k)) \Big(-\frac{h\gamma}{2}\alpha |Z_k|+ \alpha h \gamma |q_k| + \sqrt{2 \gamma^{-1}  h}\mathbb{E}[|\hat{q}_k + \Xi_{k+1}|]-|q_k|\Big). \label{eq:expec_step1}
\end{align}
We set $\hat{r}_k=|\hat{q}_k|$. For the term in expectation it holds
\begin{align}
\mathbb{E}[|\hat{q}_k + \Xi_{k+1}|] & = \int_{-\hat{r}_k/2}^{\infty} |\hat{r}_k+2 u|(\varphi_{0,1}(u)-\varphi_{0,1}(u+\hat{r}_k))\rmd u \nonumber
\\ & = \int_{-\hat{r}_k/2}^{\infty} (\hat{r}_k+2 u)\varphi_{0,1}(u)\rmd u- \int_{\hat{r}_k/2}^{\infty} (2 u-\hat{r}_k)\varphi_{0,1}(u)\rmd u \nonumber
\\ & = \int_{-\hat{r}_k/2}^{\hat{r}_k/2} (\hat{r}_k+2 u)\varphi_{0,1}(u)\rmd u+ \int_{\hat{r}_k/2}^{\infty} 2\hat{r}_k\varphi_{0,1}(u)\rmd u \nonumber
\\ & = \int_{-\hat{r}_k/2}^{\hat{r}_k/2} \hat{r}_k\varphi_{0,1}(u)\rmd u+ \int_{\hat{r}_k/2}^{\infty} 2\hat{r}_k\varphi_{0,1}(u)\rmd u= \hat{r}_k=|\hat{q}_k|. \label{eq:exp_1moment}
\end{align}
Inserting this estimate in \eqref{eq:expec_step1} and using twice $|Z_k|\ge 4|q_k|$, it holds 
\begin{align*}
\mathbb{E}[\rho_{k+1}-\rho_k]& %\le  f'(r_s(k)) \Big(-\frac{h\gamma}{2}\alpha |Z_k|+ \alpha h \gamma |q_k| + \sqrt{2 \gamma^{-1}  h}\mathbb{E}[|\hat{q}_k + \Xi_{k+1}|]-|q_k|\Big)
 \le  -f'(r_s(k)) \frac{h\gamma}{4}\alpha |Z_k| \le - f'(r_s(k))\min\Big( \frac{h\gamma }{8}, \frac{h\gamma\alpha}{2} \Big)r_s(k).
\end{align*}
By \eqref{eq:f_conseq4}, it holds 
\begin{align*}
\mathbb{E}[\rho_{k+1}]&\le \Big(1-f'(R_1) \min\Big( \frac{h\gamma }{8}, \frac{h\gamma\alpha}{2} \Big)\Big)\rho_k.
\end{align*}
%with
%\begin{align*}
%c_2= -f'(R_1) \min\Big( \frac{h\gamma }{8}, \frac{h\gamma\alpha}{2} \Big).
%\end{align*}
\textit{Step 2:} Let $|Z_k| < 4 |q_k|$ and $|q_k|\ge \sqrt{2\gamma^{-1} h}$.
By \eqref{eq:EM}, Lipschitz continuity of $\nabla U$, \eqref{eq:alpha}, \eqref{eq:exp_1moment} and Taylor expansion
\begin{align}
\mathbb{E}[\rho_{k+1}-\rho_k]&\le \mathbb{E}[f(\alpha |Z_k|+ \alpha h \gamma |q_k|+\sqrt{2 \gamma^{-1}  h}|\hat{q}_k + \Xi_{k+1}|)- f(\alpha|Z_k|+|q_k|) ] \nonumber
\\ &= f'(r_s(k)) \mathbb{E}[(\alpha h \gamma |q_k|+\sqrt{2 \gamma^{-1}  h}|\hat{q}_k + \Xi_{k+1}|-|q_k|)] \nonumber
\\ & + \mathbb{E}\Big[\int_{r_s(k)}^{\bar{r}_s(k)} (\bar{r}_s(k)-t)f''(t)\rmd t\Big] \nonumber
\\ & = f'(r_s(k)) (\alpha h \gamma |q_k|)
+ \mathbb{E}\Big[\int_{r_s(k)}^{\bar{r}_s(k)} (\bar{r}_s(k)-t)f''(t)\rmd t\Big], \label{eq:2nd_case}
%\\ & + \frac{f''(r_s(k))}{2} \mathbb{E}[(\alpha h \gamma |q_k|+\sqrt{2 \gamma^{-1}  h}|\hat{q}_k + \Xi_{k+1}|-|q|)^2]
%\\ & +  \mathbb{E}\Big[\int_{\alpha |Z_k|+ |q_k|}^{\alpha |Z_k|+ \alpha h \gamma |q_k| \qquad + \sqrt{2\gamma^{-1} h} |\hat{q}_k+\Xi_{k+1}|} \frac{f'''(t)}{2} \rmd t \Big].
\end{align}
where $\bar{r}_s(k)=\alpha |Z_k|+ \alpha h \gamma |q_k|+\sqrt{2 \gamma^{-1}  h}|\hat{q}_k + \Xi_{k+1}|$.
To bound the term in expectation, we split 
\begin{align}
\mathbb{E}\Big[\int_{r_s(k)}^{\bar{r}_s(k)} (\bar{r}_s(k)-t)f''(t)\rmd t\Big]&= \mathbb{E}\Big[\int_{r_s(k)}^{\bar{r}_s(k)} (\bar{r}_s(k)-t)f''(t)\rmd t \1_{\bar{A}}\Big] \nonumber
\\ & +\mathbb{E}\Big[\int_{r_s(k)}^{\bar{r}_s(k)} (\bar{r}_s(k)-t)f''(t)\rmd t \1_{\bar{A}^c}\Big], \label{eq:exp_secderiv}
\end{align}
where
\begin{align*}
\bar{A}=\Big\{\Xi_{k+1}=-\hat{q}_k \Big\} \quad \text{and} \quad \bar{A}^c=\Big\{\Xi_{k+1}\neq -\hat{q}_k \Big\}=\Big\{\Xi_{k+1}= 2(e_k\cdot \xi_{k+1})e_k \Big\}.
\end{align*}
For the first term in \eqref{eq:exp_secderiv} we observe
\begin{align} 
&\mathbb{E}\Big[\int_{r_s(k)}^{\bar{r}_s(k)} (\bar{r}_s(k)-t)f''(t)\rmd t \1_{\bar{A}}\Big] \nonumber
\\&=\mathbb{E}\Big[\int_{\frac{1}{2}(r_s(k)+\bar{r}_s(k))}^{\bar{r}_s(k)} (\bar{r}_s(k)-t)f''(t)\rmd t \ \1_{\bar{A}}\Big] 
 +\mathbb{E}\Big[\int_{r_s(k)}^{\frac{1}{2}(r_s(k)+\bar{r}_s(k))} (\bar{r}_s(k)-t)f''(t)\rmd t \ \1_{\bar{A}}\Big] \nonumber
 \\ & \le \mathbb{E}\Big[\frac{3(\bar{r}_s(k)-r_s(k))^2}{8}\max_{t\in[\frac{1}{2}(\bar{r}_s(k)+r_s(k)),r_s(k)]}f''(t) \ \1_{\bar{A}}\Big] \nonumber
 \\ & + \mathbb{E}\Big[\frac{(\bar{r}_s(k)-r_s(k))^2}{8}\max_{t\in[\bar{r}_s(k),\frac{1}{2}(\bar{r}_s(k)+r_s(k))]}f''(t) \ \1_{\bar{A}}\Big] \nonumber
 \\ & \le \frac{3(2\gamma^{-1} h)(1-\alpha \gamma h)^2|\hat{q}_k|^2}{8} \max_{t\in[r_s(k)-(1-\alpha h \gamma)/2|q_k|,r_s(k)]}f''(t)\mathbb{E}[ \ \1_{\bar{A}}]. \label{eq:case2_contrcoup}
\end{align}
Note that in the last step we ignored the last summand since $f''(t)$ is negative.

For the second term in \eqref{eq:exp_secderiv} we define the set 
\begin{align*}
A'=\Big\{\sqrt{2\gamma^{-1} h }|\hat{q}_k+\Xi_{k+1}|\le |q_k|-\sqrt{2\gamma^{-1} h}  \Big\}.
\end{align*}
Then,
\begin{align}
\mathbb{E}& \Big[\int_{r_s(k)}^{\bar{r}_s(k)} (\bar{r}_s(k)-t)f''(t)\rmd t \ \1_{\bar{A}^c}\Big] \nonumber
\\ & = \mathbb{E}\Big[\int_{r_s(k)}^{\frac{1}{2}(\bar{r}_s(k)+r_s(k))} (\bar{r}_s(k)-t)f''(t)\rmd t \ \1_{\bar{A}^c}\Big]+\mathbb{E}\Big[\int_{\frac{1}{2}(\bar{r}_s(k)+r_s(k))}^{\bar{r}_s(k)} (\bar{r}_s(k)-t)f''(t)\rmd t \ \1_{\bar{A}^c}\Big] \nonumber
\\ & \le \mathbb{E}\Big[\frac{3(\bar{r}_s(k)-r_s(k))^2}{8} \max_{t\in[\frac{1}{2}(\bar{r}_s(k)+r_s(k)), r_s(k)]}f''(t) \ \1_{\bar{A}^c}\ \1_{A'}\Big] \nonumber
\\ & +\mathbb{E}\Big[\frac{(\bar{r}_s(k)-r_s(k))^2}{8} \max_{t\in[\bar{r}_s(k), \frac{1}{2}(\bar{r}_s(k)+ r_s(k))]}f''(t)\ \1_{\bar{A}^c}\ \1_{A'}\Big] \nonumber
\\ & \le \mathbb{E}\Big[\frac{3(\bar{r}_s(k)-r_s(k))^2}{8} \max_{t\in[\frac{1}{2}(\bar{r}_s(k)+r_s(k)), r_s(k)]}f''(t) \ \1_{\bar{A}^c}\ \1_{A'}\Big]. \label{eq:case2_reflcoup}
\end{align}
To bound the second derivative of $f$ in \eqref{eq:case2_contrcoup} and \eqref{eq:case2_reflcoup}, we observe for $t\in [\frac{1}{2}(\bar{r}_s(k)+r_s(k)), r_s(k)]$ the bounds $t\ge (1/2)r_s(k)$, $\exp(-128 \alpha \gamma^2 \frac{t^2}{2})\ge \exp(-128 \alpha \gamma^2\frac{r_s(k)^2}{2})$ and $\psi(t)\ge \frac{1}{2}\psi(r_s(k))$. Further for $s>t$, $\exp(-128 \alpha \gamma^2\frac{t^2}{2})\ge \exp(-128 \alpha \gamma^2\frac{s^2}{2})$. Hence by \eqref{eq:f_conseq} for $t\in [\frac{1}{2}(\bar{r}_s(k)+r_s(k)), r_s(k)]$,
\begin{align}
f''(t) & =-128 \alpha \gamma^2 t f'(t)-\frac{\hat{c}\gamma}{2}\int_0^t \phi(s) \rmd s \le -128 \alpha \gamma^2 \frac{r_s(k)}{2} f'(t)-\frac{\hat{c}\gamma}{2}\int_0^{\frac{r_s(k)}{2}} \phi(s) \rmd s \nonumber
\\ &\le -128 \alpha \gamma^2 \frac{r_s(k)}{2}\frac{1}{2} f'(r_s(k))-\frac{\hat{c}\gamma}{4}\int_0^{r_s(k)} \phi(s) \rmd s \nonumber
\\ & \le -128 \alpha \gamma^2 \frac{r_s(k)}{4} f'(r_s(k))-\frac{\hat{c}\gamma}{4}f(r_s(k)). \label{eq:conseq_f2}
\end{align}
Inserting these estimates in  \eqref{eq:case2_contrcoup} and \eqref{eq:case2_reflcoup}, applying $|q_k|\ge \sqrt{2\gamma^{-1} h}$ and using that by \eqref{eq:condition_h} $\alpha\gamma h \le 1/2$, we obtain
\begin{align} \label{eq:2nd_case1}
\mathbb{E} & \Big[\int_{r_s(k)}^{\bar{r}_s(k)} (\bar{r}_s(k)-t)f''(t)\rmd t \1_{\bar{A}}\Big] \nonumber
\\ & = 
\mathbb{E}\Big[\int_{r_s(k)}^{\bar{r}_s(k)} (\bar{r}_s(k)-t)f''(t)\rmd t \1_{\bar{A}}\Big]
 + \mathbb{E} \Big[\int_{r_s(k)}^{\bar{r}_s(k)} (\bar{r}_s(k)-t)f''(t)\rmd t \ \1_{\bar{A}^c}\Big]  \nonumber
\\ & \le \frac{3(2\gamma^{-1} h)}{32} \Big(-128 \alpha \gamma^2 \frac{r_s(k)}{4} f'(r_s(k))-\frac{\hat{c}\gamma}{4}f(r_s(k))\Big) \mathbb{E}[ \ \1_{\bar{A}}] \nonumber
\\ &  +\Big(-128 \alpha \gamma^2\frac{r_s(k)}{4} f'(r_s(k))-\frac{\hat{c}\gamma}{4}f(r_s(k))\Big) \mathbb{E}\Big[\frac{3(\bar{r}_s(k)-r_s(k))^2}{8}  \ \1_{\bar{A}^c}\ \1_{A'}\Big].
\end{align}
%and
%\begin{align}\label{eq:2nd_case2}
%\mathbb{E}& \Big[\int_{r_s(k)}^{\bar{r}_s(k)} (\bar{r}_s(k)-t)f''(t)\rmd t \ \1_{\bar{A}^c}\Big] \nonumber
%\\ & \le  \Big(-128 \alpha \gamma^2\frac{r_s(k)}{4} f'(r_s(k))-\frac{\hat{c}\gamma}{4}f(r_s(k))\Big) \mathbb{E}\Big[\frac{3(\bar{r}_s(k)-%r_s(k))^2}{8}  \ \1_{\bar{A}^c}\ \1_{A'}\Big].
%\end{align}
For the expectation in the last term, it holds by \eqref{eq:coupl_contr} and the definition of $A'$
\begin{align*}
 \mathbb{E} & \Big[\frac{3(\bar{r}_s(k)-r_s(k))^2}{8}  \ \1_{\bar{A}^c}\ \1_{A'}\Big] 
 \\ &  = \int_{-\infty}^{\infty}\frac{3(\alpha h \gamma |q_k|+\sqrt{2\gamma^{-1}h} ||\hat{q}_k|+2u|-|q_k|)^2}{8} \ \1_{\{||\hat{q}_k|+2u|\le |\hat{q}_k|-1\}} (\varphi(u)-\varphi(u+|\hat{q}_{k}|))^+\rmd u
 \\ & = \int_{-\infty}^{-1/2}\frac{3(\alpha h \gamma |q_k|+\sqrt{2\gamma^{-1}h} 2u )^2}{8} (\varphi(u)-\varphi(u+|\hat{q}_{k}|))^+\rmd u
 \\ & = \frac{3\gamma^{-1} h}{4}\int_{-\infty}^{-1/2}(\alpha h \gamma |\hat{q}_k|+ 2u )^2 (\varphi(u)-\varphi(u+|\hat{q}_{k}|))^+\rmd u
  \\ & \ge  \frac{3\gamma^{-1} h}{4}\int_{-\infty}^{-1/2}(\alpha \sqrt{h \gamma/2} \gamma R_1 + 2u )^2 (\varphi(u)-\varphi(u+|\hat{q}_{k}|))^+\rmd u
\\ & \ge  \frac{3\gamma^{-1} h}{4}\int_{-\infty}^{-1/2}(1/2 + 2u )^2 (\varphi(u)-\varphi(u+|\hat{q}_{k}|))^+\rmd u
\\ & \ge  \frac{3\gamma^{-1} h}{4}\int_{-\infty}^{-1/2}(1/4 ) (\varphi(u)-\varphi(u+|\hat{q}_{k}|))^+\rmd u,
\end{align*}
since by \eqref{eq:R_1}, $|q_k|\le r_s(k)\le R_1$ and by the assumption \eqref{eq:condition_h} on $h$,  $\alpha \sqrt{h\gamma/2}\gamma R_1\le 1/2$. Then, by \eqref{eq:coupl_contr}
\begin{align*}
 \mathbb{E} & \Big[\frac{3(\bar{r}_s(k)-r_s(k))^2}{8}  \ \1_{\bar{A}^c}\ \1_{A'}\Big] + \frac{3(2\gamma^{-1} h) }{32}\mathbb{E}[\1_{\bar{A}}]
 \\ & \ge  \frac{3\gamma^{-1} h}{4}\int_{-\infty}^{-1/2}\frac{1}{4} (\varphi(u)-\varphi(u+|\hat{q}|))^+\rmd u +\frac{3(2\gamma^{-1} h )}{32}\int_{-\infty}^{\infty} (\varphi(u)\wedge \varphi(u+|\hat{q}|))\rmd u
 \\ & \ge \frac{3 \gamma^{-1} h}{16}\int_{-\infty}^{-1/2}  \varphi(u)\rmd u \ge \frac{9}{160}\gamma^{-1} h.
\end{align*}
Inserting this estimate into the sum of \eqref{eq:2nd_case1} and plugging it back into \eqref{eq:2nd_case}, we obtain
\begin{align*}
\mathbb{E}[\rho_{k+1}-\rho_k]%&\le f'(r_s(k)) (\alpha h \gamma |q_k|)+ \mathbb{E}\Big[\int_{r_s(k)}^{\bar{r}_s(k)} (\bar{r}_s(k)-t)f''(t)\rmd t \1_{\bar{A}}\Big] 
%\\ & + \mathbb{E} \Big[\int_{r_s(k)}^{\bar{r}_s(k)} (\bar{r}_s(k)-t)f''(t)\rmd t \ \1_{\bar{A}^c}\Big]
 & \le  f'(r_s(k)) (\alpha h \gamma |q_k|)+ \Big(-128 \alpha \gamma^2 \frac{r_s(k)}{2}\frac{1}{2} f'(r_s(k))-\frac{\hat{c}\gamma}{4}f(r_s(k))\Big) \frac{9}{160}\gamma^{-1} h
\\ & \le -\frac{9}{640}\hat{c}h f(r_s(k))=-\frac{9}{640}\hat{c}h \rho_k.
\end{align*}

\textit{Step 3:} Let $|Z_k| < 4 |q_k|$ and $|q_k|< \sqrt{2\gamma^{-1} h}$.
By \eqref{eq:EM} and \Cref{ass}, it holds
\begin{align}
\mathbb{E}&[\rho_{k+1} - \rho_k]\le \mathbb{E}[f(r_s(k+1))- f(r_s(k))] \nonumber
\\ &\le \mathbb{E}[f(\alpha|Z_k|+\alpha h \gamma |q_k| +\sqrt{2\gamma^{-1} h } |\hat{q}_k+ \Xi_{k+1}|))- f(\alpha|Z_k|+|q_k|)] \nonumber
\\ & \le \mathbb{E}[f'(r_s(k))(\alpha h \gamma |q_k|+\sqrt{2\gamma^{-1} h } |\hat{q}_k+ \Xi_{k+1}|-|q_k| )] 
+ \mathbb{E}\Big[\int_{r_s(k)}^{\bar{r}_s(k) }   (\bar{r}_s(k)-t)f''(t)\rmd t \Big] \nonumber
\\ & \le f'(r_s(k))\alpha h \gamma |q_k| + \mathbb{E}\Big[\int_{r_s(k)}^{\bar{r}_s(k)}  (\bar{r}_s(k)-t)f''(t)\rmd t \Big], \label{eq:3rd_case}
\end{align}
where \eqref{eq:exp_1moment} is applied in the last step and $\bar{r}_s(k)=\alpha |Z_{k}|+\alpha h \gamma |q_k|+\sqrt{2\gamma^{-1}h } |\hat{q}_k+ \Xi_{k+1}|$. To bound the second term, consider the set 
\begin{align*}
A=\Big\{ |q_k|+ 2\sqrt{2\gamma^{-1} h} \le \sqrt{2\gamma^{-1} h }|\hat{q}_k+ \Xi_{k+1}|\le  |q_k|+ 6\sqrt{2\gamma^{-1} h}  \Big\}.
\end{align*}
Then, using the non-positivity of $f''$ we bound the expectation by 
\begin{align*}
\mathbb{E} \Big[\int_{r_s(k)}^{\bar{r}_s(k)}  (\bar{r}_s(k)-t)f''(t)\rmd t \Big]
 &\le \mathbb{E} \Big[\int_{r_s(k)}^{\bar{r}_s(k)}  (\bar{r}_s(k)-t)f''(t)\rmd t \ \1_{A} \Big].
\\ & \le \mathbb{E}\Big[\int_{\frac{1}{2}(r_s(k)+\bar{r}_s(k)) }^{\bar{r}_s(k)}  (\bar{r}_s(k)-t)f''(t)\rmd t \ \1_{A} \Big]
%  + \mathbb{E}\Big[\int_{r_s(k)}^{\frac{1}{2}(r_s(k)+\bar{r}_s(k))}  (\bar{r}_s(k)-t)f''(t)\rmd t \ \1_{A}\Big]
\\ & \le \mathbb{E}\Big[\frac{(\bar{r}_s(k)-r_s(k))^2}{8} \max_{t\in [\frac{1}{2}(r_s(k)+\bar{r}_s(k)), \bar{r}_s(k)] }f''(t)\ \1_{A} \Big].
%\\ & + \mathbb{E}\Big[\frac{3(\bar{r}_s(k)-{r}_s(k))^2}{8}\max_{t\in [r_s(k), \frac{1}{2}(\bar{r}_s(k)+r_s(k))]}f''(t)\rmd t \ \1_{A}\Big]
%\\ & \le \mathbb{E}\Big[\frac{(\bar{r}_s(k)-r_s(k))^2}{8} \max_{t\in [\frac{1}{2}(r_s(k)+\bar{r}_s(k)), \bar{r}_s(k)] }f''(t)\ \1_{A} \Big].
\end{align*} 
%where $\bar{r}_s(k+1)=\alpha |Z_{k}|+\alpha h \gamma |q_k|+\sqrt{2\gamma^{-1}h } |\hat{q}_k+ \Xi_{k+1}|$. 
By construction of $A$, it holds $\frac{1}{2}(r_s(k)+\bar{r}_s(k))\ge\sqrt{2\gamma^{-1} h}+r_s(k)$. Then by \eqref{eq:f_conseq}, it holds  for all $t\in [\frac{1}{2}(r_s(k)+\bar{r}_s(k)), \bar{r}_s(k)]$
\begin{align}
f''(t)
& = \Big(-f'(t)128 \alpha \gamma^2 t- \frac{\hat{c}\gamma}{2}\int_0^t \phi(s)\rmd s \Big) \nonumber
\\ & \le \Big(-\frac{1}{2}\psi(r_s(k))\phi(t)128 \alpha \gamma^2\sqrt{2\gamma^{-1} h} - \frac{\hat{c}\gamma}{2}\int_0^{r_s(k)+\sqrt{2h\gamma^{-1}}}\phi(s)\rmd s\Big) \nonumber
\\ & \le \Big(-\frac{1}{2}\psi(r_s(k))\phi(r_s(k))128 \alpha \gamma^2\frac{\phi(\bar{r}_s(k))}{\phi(r_s(k))}\sqrt{2\gamma^{-1} h} - \frac{\hat{c}\gamma}{2}\int_0^{\sqrt{2h\gamma^{-1}}}\phi(s)\rmd s\Big). \label{eq:f_conseq3}
\end{align}
Note that the first inequality holds since $t\ge \sqrt{2\gamma^{-1} h}$ and $\psi(t)\ge \psi(r_s(k))/2$ since $\psi(x)\in[1/2,1]$ for all $x\ge 0$.
Since by \eqref{eq:condition_h} it holds $8(6+\alpha h \gamma)\le 50$, $(6+\alpha h \gamma)(8+\alpha h \gamma)\le 50$,
$128 \alpha \gamma^2(50\alpha+50)(\gamma^{-1} h) \le 2/3$ and $128 \alpha \gamma^2 h\gamma^{-1} \le 2/3$, we obtain
\begin{align*}
\frac{\phi(\bar{r}_s(k))}{\phi(r_s(k))} & = \exp\Big(- 128 \alpha \gamma^2 \frac{\bar{r}_s(k)^2-r_s(k)^2 }{2}\Big)
\\ & \ge  \exp\Big(- 128 \alpha \gamma^2 \frac{(r_s(k)+ \sqrt{2\gamma^{-1} h}(6+\alpha h \gamma))^2-r_s(k)^2 }{2}\Big)
\\ & \ge  \exp\Big(- 128 \alpha \gamma^2 \frac{(2(4\alpha+1)2\gamma^{-1} h(6+\alpha h \gamma)+ 2\gamma^{-1} h(6+\alpha h \gamma)^2) }{2}\Big)
\\ & \ge  \exp\Big(- 128 \alpha \gamma^2 \frac{(2(4\alpha)2\gamma^{-1} h(6+\alpha h \gamma)+ 2\gamma^{-1} h(6+\alpha h \gamma)(8+\alpha h \gamma)) }{2}\Big)\ge \frac{1}{2},
\end{align*}
and
\begin{align}
\frac{r_s(k)}{\sqrt{2\gamma^{-1}h}}\int_0^{\sqrt{2h\gamma^{-1}}}\phi(s)\rmd s & \ge \frac{r_s(k)}{\sqrt{2\gamma^{-1}h}}\sqrt{2h\gamma^{-1}}\phi(\sqrt{2h\gamma^{-1}})  \ge \int_0^{r_s(k)} \phi(s)\rmd s %\frac{r_s(k)+\sqrt{2h\gamma^{-1}}}{\sqrt{2\gamma^{-1}h}}
\phi(\sqrt{2h\gamma^{-1}}) \nonumber
%\\ & \ge \int_0^{r_s(k)} \phi(s)\rmd s \ \phi(\sqrt{(4\alpha+1)2\gamma^{-1} h}+\sqrt{2h\gamma^{-1}}) \nonumber
\\ & \ge f(r_s(k)) \exp\Big(-128 \alpha \gamma^2 \frac{1}{2}(2h\gamma^{-1})\Big)\ge \frac{1}{2} f(r_s(k)). \label{eq:f_conseq2}
\end{align}
Inserting these two bounds in $f''(t)$ and taking the maximum over $t\in [\frac{1}{2}(r_s(k)+\bar{r}_s(k)), \bar{r}_s(k)]$ it holds
\begin{align*}
&\max_{t\in [\frac{1}{2}(r_s(k)+\bar{r}_s(k)), \bar{r}_s(k)] } f''(t)
\le -\frac{1}{4}f'(r_s(k))128 \alpha \gamma^2\sqrt{2\gamma^{-1} h} - \frac{\hat{c}\gamma}{4}\frac{\sqrt{2\gamma^{-1}h}}{r_s(k)}f(r_s(k))
\end{align*}
which yields
\begin{align*}
\mathbb{E}& \Big[\int_{r_s(k)}^{\bar{r}_s(k)}  (\bar{r}_s(k)-t)f''(t)\rmd t \Big] 
\\ & \le \frac{(\alpha h \gamma+2)^2(2\gamma^{-1} h)}{8} \Big(-32 \alpha \gamma^2 f'(r_s(k)) \sqrt{2\gamma^{-1} h} - \frac{\hat{c}\gamma}{4}\frac{\sqrt{2\gamma^{-1} h}}{(4\alpha +1)|q_k|}
f(r_s(k))\Big)\mathbb{E}\Big[ \1_{A} \Big].
\end{align*}
Note that we used $r_s(k)\le (4\alpha+1) |q_k|$.
For the expectation $\mathbb{E}[ \1_{A}]$ it holds 
\begin{align*}
    \mathbb{E}[\1_{A}]& = \int_{-|\hat{q}_k|/2}^{\infty} \1_{\{2\sqrt{2\gamma^{-1} h}+|{q}_k| \le \sqrt{2\gamma^{-1} h} ||\hat{q}_k|+2 u| \le |{q}_k|+ 6\sqrt{2\gamma^{-1} h}\}}(\varphi(u)-\varphi(u+|\hat{q}_k|))\rmd u
    \\ & =\int_{1}^3 \varphi(u)\rmd u - \int_{1+|q_k|/\sqrt{2\gamma^{-1} h}}^{3+|{q}_k|/\sqrt{2\gamma^{-1} h}} \varphi(u)\rmd u =: F(|q_k|),
\end{align*}
where the second step holds since for $u\ge -|\hat{q}_k|/2$, the restriction on the set $A$ implies $1\le u \le 3$.
By \cite[Lemma 3.4]{chak2023reflection} and since $|q_k|\le \sqrt{2\gamma^{-1} h}$, this term is bounded from below by 
\begin{align}
     \mathbb{E}[\1_{A}]&\ge \min\Big( \frac{F( \sqrt{2\gamma^{-1} h})}{\sqrt{2\gamma^{-1} h}}, F'(0)\Big) |q_k| \nonumber
     \\ & = \min \Big(\int_{1}^3 \varphi(u)\rmd u - \int_{1+1 }^{3+1} \varphi(u)\rmd u , \frac{1}{\sqrt{2\pi}}\Big(e^{-\frac{1^2}{2}}-e^{-\frac{3^2}{2}}\Big)\Big) \frac{|q_k|}{\sqrt{2\gamma h^{-1}}}\ge \frac{1}{8}\frac{|q_k|}{\sqrt{2\gamma h^{-1}}}. \label{eq:expecA}
\end{align}
Hence, we obtain
\begin{align*}
\mathbb{E}& \Big[\int_{r_s(k)}^{\bar{r}_s(k)}  (\bar{r}_s(k)-t)f''(t)\rmd t \Big] 
\\ & \le \frac{1}{8}\frac{(\alpha h \gamma+2)^2(2\gamma^{-1} h)}{8} \Big(-32 \alpha \gamma^2 f'(r_s(k))|q_k| - \frac{\hat{c}\gamma}{4(4\alpha +1)}f(r_s(k))\Big)
\\ & \le \Big(-4(\gamma^{-1} h)\alpha \gamma^2 f'(r_s(k))|q_k| -  \frac{1}{8}(\gamma^{-1} h)\frac{\hat{c}\gamma}{4(4\alpha +1)}f(r_s(k))\Big).
\end{align*}
Inserting this bound in \eqref{eq:3rd_case}, we obtain
\begin{align*}
\mathbb{E}[\rho_{k+1} - \rho_k]  %f'(r_s(k))\alpha h \gamma |q_k| + \mathbb{E}\Big[\int_{r_s(k)}^{\bar{r}_s(k)}  (\bar{r}_s(k)-t)f''(t)\rmd t \Big]
 &\le f'(r_s(k))\alpha h \gamma |q_k|-4(\gamma^{-1} h) \alpha \gamma^2 f'(r_s(k))|q_k| -  (\gamma^{-1} h)\frac{\hat{c}\gamma}{32(4\alpha +1)}f(r_s(k))
\\ & \le -  (\gamma^{-1} h)\frac{\hat{c}\gamma}{32(4\alpha +1)}f(r_s(k)).
\end{align*}
Combining the three steps we obtain for $D_{\mathcal{K}}+\epsilon r_l(k)> r_s(k)$
\begin{align}
    \mathbb{E}[\rho_{k+1}] &\le \mathbb{E}[f(r_s(k+1))]\le (1-c_2h) f(r_s(k))=(1-c_2h)\rho_k
\end{align}
with
\begin{align*}
    c_2=\min\Big(\frac{f'(R_1)\gamma}{8},\frac{f'(R_1)\gamma \alpha}{2},\frac{9\hat{c}}{649}, \frac{\hat{c}}{32(4\alpha +1)} \Big).
\end{align*}
%where we used $\frac{9\hat{c}}{640}>\frac{\hat{c}}{128(4\alpha +1)}$.
Combining this estimate with the first case ($D_{\mathcal{K}}+\epsilon r_l(k)\le r_s(k)$), we obtain 
\begin{align}
     \mathbb{E}[\rho_{k+1}] &\le(1-ch)\rho_k
\end{align}
with $c$ given in \eqref{eq:c}.

\end{proof}

\begin{lemma}\label{lem:apriori_bounds} Consider the continuous kinetic Langevin dynamics $(X_{t},V_{t})_{t\geq 0}$ with initial distribution $\mu_{\infty}$.
    Let $l\in \mathbb{N}$.
    Then, for all $k\in \{0, \ldots , l-1\}$
    \begin{align*}
        &\mathbb{E}\Big[\int_0^h |X_{kh+s}-X_{kh}|\rmd s\Big]\le \frac{h^2}{2}\sqrt{d} %  \frac{h^2}{2} \Big( \frac{Lh}{2}\sqrt{\frac{(L+\kappa)R^2 + d}{\kappa} }+ \frac{\sqrt{d}}{2}+\frac{\sqrt{2\gamma }}{2} \sqrt{hd} \Big),
        \\ & \mathbb{E}\Big[\Big|\sum_{i=0}^k\int_0^h V_{ih+s}-V_{ih}\rmd s\Big|\Big]\le \frac{h^2}{2}l(\sqrt{L}+\gamma) \sqrt{d}+  \sqrt{2\gamma dlh}h.   %h^2l\Big(\frac{L}{4}\sqrt{\frac{(L+\kappa)R^2 + d}{\kappa} } + \frac{Lh+\gamma}{4}\sqrt{d} +\frac{Lh+ \gamma}{4}\sqrt{2\gamma hd}\Big) 
        %\\ & \qquad + \sqrt{2\gamma d lh} h.
    \end{align*}
    %where $\mathbf{C}_X=\int_{\mathbb{R}} |x|\rmd \mu_{\infty}(x,v)$ and $\mathbf{C}_V=\int_{\mathbb{R}} |v|\rmd \mu_{\infty}(x,v)$. %$\mathbf{C}_X$ and $\mathbf{C}_V$ are denote the uniform in time moment bounds of $(X_t,V_t)$ and are given in \Cref{lem:moment_bounds}.
\end{lemma}

\begin{proof}[Proof of \Cref{lem:apriori_bounds}]
    By \eqref{eq:LD}, it holds
    \begin{align} \label{eq:boundX}
        \int_0^h |X_{kh+s}-X_{kh}|\rmd s = \int_0^h |\int_0^s V_{kh+r}\rmd r|\rmd s\le \int_0^t \int_0^s |V_{kh+r}|\rmd r\rmd s.
    \end{align}
    Taking expectation yields 
    \begin{align*}
        \mathbb{E}\Big[\int_0^h |X_{kh+s}-X_{kh}|\rmd s\Big]\le \mathbb{E}\Big[\int_0^t \int_0^s |V_{kh+r}|\rmd r\rmd s\Big]= \int_0^t \int_0^s \mathbb{E}[|V_{kh+r}|]\rmd r\rmd s=\frac{h^2}{2}\sqrt{d},
    \end{align*}
    since for $(X_{kh+r},V_{kh+r})\sim \mu_{\infty}$ for all $k\in\mathbb{N}$ and $ r\in[0,h]$ we have $$\|V_{kh+r}\|_{L^1(\mu_{\infty})}\le \sqrt{d}$$ by \cite[Lemma A.3]{deligiannidis2021randomized}.
    Further by \eqref{eq:LD}
    \begin{align} \label{eq:boundV}
        \Big|\sum_{i=0}^k\int_0^h V_{ih+s}&-V_{ih}\rmd s\Big|  = \Big|\sum_{i=0}^k\int_0^h \int_0^s -\nabla U(X_{ih+r})-\gamma V_{ih+r} \rmd r+ \sqrt{2\gamma}\int_{0}^s \rmd B_{ih+r} \rmd s\Big| \nonumber
        \\ & \le \sum_{i=0}^k\int_0^h \int_0^s (|\nabla U(X_{ih+r})| + \gamma|V_{ih+r}|) \rmd r \rmd s  
       +  \sqrt{2\gamma}\Big|\sum_{i=0}^k\int_0^h (B_{ih+s}-B_{ih})\rmd s\Big|.
    \end{align}
    We note that $\sum_{i=0}^k (B_{ih+s}-B_{ih})$ is a normally distributed random variable with mean zero and covariance matrix $((k+1)s) I_d$. Further, for $(X_{ih+r},V_{ih+r})\sim \mu_{\infty}$ for all $i\in\mathbb{N}$,  $\|V_{ih+r}\|_{L^1(\mu_{\infty})}\le \sqrt{d}$ and $$\|\nabla U(X_{ih+r})\|_{L^1(\mu_{\infty})}\le \sqrt{Ld}$$ by \cite[Lemma A.3]{deligiannidis2021randomized}.
    Then since $k<l$,
    \begin{align*}
        \mathbb{E}[|\sum_{i=0}^k\int_0^h V_{ih+s}-V_{ih}\rmd s|] &\le \mathbb{E}[\sum_{i=0}^k\int_0^h \int_0^s (|\nabla U(X_{ih+r})| + \gamma|V_{ih+r}|) \rmd r \rmd s  
       \\ & +\sqrt{2\gamma}\Big|\sum_{i=0}^k\int_0^h (B_{ih+s}-B_{ih})\rmd s\Big|]
       \\ & \le \frac{h^2}{2}l(\sqrt{Ld}+\gamma \sqrt{d})+ \sqrt{2\gamma dlh}h.
    \end{align*}

\end{proof}

\begin{proof}[Proof of \Cref{thm:strong_accuracy}]
Recall that we use $(\mathbf{X},\mathbf{V})$ to refer to the discretization and $(X,V)$ to refer to the continuous dynamics.

Set $ l=\lceil \frac{1}{h \gamma\max(\alpha, 1+ \alpha/2)} \rceil$. 
    By \eqref{eq:contr_k_step}, we have
    \begin{align*}
        \mathcal{W}_{\rho}(\mu_{\infty},\mu_{h,\infty})&=\mathcal{W}_{\rho}(\mu_{\infty}\pi^l,\mu_{h,\infty}\pi_h^l)\le \mathcal{W}_{\rho}(\mu_{\infty}\pi^l,\mu_{\infty}\pi_h^l)+\mathcal{W}_{\rho}(\mu_{\infty}\pi_h^l,\mu_{h,\infty}\pi_h^l)
        \\ & \le \mathcal{W}_{\rho}(\mu_{\infty}\pi^l,\mu_{\infty}\pi_h^l)+(1-ch)^l\mathcal{W}_{\rho}(\mu_{\infty},\mu_{h,\infty}).
    \end{align*}
    Hence, by $(1-ch)\le e^{-ch}$, $lh\ge \frac{1}{ \gamma\max(\alpha, 1+ \alpha/2)}$ and \eqref{eq:rho},
    \begin{align*}
        \mathcal{W}_{\rho}(\mu_{\infty},\mu_{h,\infty})&\le \frac{1}{1-(1-ch)^l}\mathcal{W}_{\rho}(\mu_{\infty}\pi^l,\mu_{\infty}\pi_h^l)\le (1-e^{-chl})^{-1}\mathcal{W}_{\rho}(\mu_{\infty}\pi^l,\mu_{\infty}\pi_h^l)
        \\ & \le (1-e^{-\frac{c}{h \gamma\max(\alpha, 1+ \alpha/2)}})^{-1}\mathbb{E}[\alpha|X_{hl}-\mathbf{X}_l|+|(X_{hl}-\mathbf{X}_l)+ \gamma^{-1}(V_{hl}-\mathbf{V}_l)|],
    \end{align*}
    with $(X_0,V_0)=(\mathbf{X}_0,\mathbf{V}_0)\sim \mu_{\infty}$ and where $(\mathbf{X}_k,\mathbf{V}_k)_{k\in \mathbb{N}}$ and $(X_s,V_s)_{s\geq 0}$ are synchronously coupled, i.e., $\xi_k=\int_{h(k-1)}^{hk}\rmd B_s$. Define the sequences $(a_k)_{k=0}^l$ and $(b_k)_{k=0}^l$ by 
    \begin{align*}
        a_k=\mathbb{E}[\alpha|X_{hk}-\mathbf{X}_k|]  \qquad \text{and} \qquad  b_k=\mathbb{E}[|(X_{hk}-\mathbf{X}_k)+ \gamma^{-1}(V_{hk}-\mathbf{V}_k)|].
    \end{align*}
    By using \eqref{eq:LD} and \eqref{eq:EM} iteratively and $a_0=b_0=0$, it holds
    \begin{align*}
        a_{k+1}&\le \mathbb{E}\Big[\alpha\Big|X_{kh}-\mathbf{X}_k+ \int_0^h (V_{hk+s}-\mathbf{V}_k)\rmd s\Big|\Big]
        %\\ & \le \mathbb{E}\Big[\alpha\Big|X_{0}-\mathbf{X}_0+\sum_{i=0}^k\int_0^h (V_{hi+s}-\mathbf{V}_{i})\rmd s  \Big|\Big]
        \\ & \le \mathbb{E}\Big[\alpha\Big|X_{0}-\mathbf{X}_0+\sum_{i=0}^k\int_0^h (V_{hi+s}-V_{hi})\rmd s +\sum_{i=0}^k h(V_{hi}-\mathbf{V}_{i}) \Big|\Big]
        \\ & \le \mathbb{E}\Big[\alpha\Big|\sum_{i=0}^k\int_0^h (V_{hi+s}-V_{hi})\rmd s\Big|\Big] +\alpha\sum_{i=0}^k h\mathbb{E}\Big[\Big|V_{hi}-\mathbf{V}_{i} \Big|\Big] \le h\mathbf{M}_2+h\gamma  \sum_{i=1}^k( \alpha b_i+a_i)
    \end{align*}
    where using \Cref{lem:apriori_bounds} and $l h \le 2/(\gamma(1+\alpha))$, the constant $\mathbf{M_2}$ is given by
    \begin{align}
        h^{-1} \mathbb{E}\Big[\alpha\Big|\sum_{i=0}^k\int_0^h (V_{hi+s}-V_{hi})\rmd s\Big|\Big] &\le \frac{\alpha h}{2}l(\sqrt{L}+\gamma) \sqrt{d}+  \alpha \sqrt{2\gamma dlh}  \nonumber
        \\ &  \le \frac{\alpha}{\gamma(1+\alpha)}(\sqrt{L}+\gamma) \sqrt{d}+ 2\alpha\sqrt{\frac{d}{1+\alpha}}  =:\mathbf{M_2}. \label{eq:M2}
        %\\ & \le \alpha h l \Big(\frac{L}{4}\sqrt{\frac{(L+\kappa)R^2 + d}{\kappa} } + \frac{(Lh+\gamma)}{4}\sqrt{d} +\frac{(Lh+\gamma)}{4}\sqrt{2\gamma hd} + \sqrt{2\gamma d h l }\Big)  \nonumber
        %\\ & \le \frac{\alpha}{\gamma(1+\alpha)}  \Big(\frac{L}{2}\sqrt{\frac{(L+\kappa)R^2 + d}{\kappa} } + \Big(\frac{\sqrt{L}}{2}+\gamma \Big)\sqrt{(3/2)d} + 4\sqrt{\frac{ d}{1+\alpha} }\Big)  =:\mathbf{M_2}. \label{eq:M2}
    \end{align}
    Further,
    \begin{align}
        b_{k+1} & \le b_k+ \mathbb{E}\Big[\int_0^h \gamma^{-1}  |\nabla U(X_{hk+s})-\nabla U(\mathbf{X}_k)|\rmd s \Big] \le b_k+\mathbb{E}\Big[\int_0^h \gamma^{-1} L |X_{hk+s}-\mathbf{X}_k|\rmd s \Big] \nonumber
        \\ & \le b_k+ \frac{\alpha\gamma h}{2}a_k+  \mathbb{E}\Big[\int_0^h \gamma^{-1} L |X_{hk+s}-X_{kh}|\rmd s \Big] \nonumber
        \\ & \le b_0 + \sum_{i=0}^k \frac{\alpha\gamma h}{2}a_i+  \sum_{i=0}^k\mathbb{E}\Big[\int_0^h \gamma^{-1} L |X_{hi+s}-X_{ih}|\rmd s \Big]. \label{eq:seq_b_l}
    \end{align}
    Then by \Cref{lem:apriori_bounds} which provides bounds for the continuous time process and by the definition of $l$, 
    we have for all $1\le k<l$
    \begin{align}
        \sum_{i=0}^k\mathbb{E}\Big[\int_0^h |X_{ih+s}  -X_{ih}|\rmd s\Big] &\le  \frac{h^2}{2} l\sqrt{d}\le \frac{h}{\gamma(1+\alpha)}\sqrt{d} %\Big( \frac{Lh}{2}\sqrt{\frac{(L+\kappa)R^2 + d}{\kappa} }+ \frac{\sqrt{d}}{2}+\frac{\sqrt{2\gamma hd}}{2} \Big) \nonumber
        %\\ & \le h \frac{1}{\gamma(1+\alpha)}\Big( \frac{\sqrt{L}}{4}\sqrt{\frac{(L+\kappa)R^2 + d}{\kappa} }+ \frac{\sqrt{(3/2)d}}{2} \Big) 
        =: h \mathbf{M}_1 \label{eq:M1}.
    \end{align}
    Since $a_0=0$,
    \begin{align*}
        b_{k+1}\le \sum_{i=1}^k \frac{\alpha\gamma h}{2}a_i+ \gamma^{-1} L h \mathbf{M}_1.
    \end{align*}
    Putting the bound for $a_{k+1}$ and $b_{k+1}$ together, we obtain
    \begin{align*}
        (a_{k+1}+b_{k+1})\le h(\gamma^{-1} L \mathbf{M}_1+\mathbf{M}_2)+ h\gamma \max(\alpha, 1+\alpha/2)\sum_{i=1}^k (a_{i}+b_{i}).
    \end{align*}
    We note that the sequence $(a_k+b_k)_{k\in \mathbb{N}}$ is bounded from above by the sequence $(c_k)_{k\in \mathbb{N}}$ (i.e., $a_k+b_k\le c_k$ for all $k$) satisfying 
    \begin{align*}
        c_{k+1}= h(\gamma^{-1} L \mathbf{M}_1+\mathbf{M}_2)+ h\gamma \max(\alpha, 1+\alpha/2)\sum_{i=1}^k c_i.
    \end{align*}
    and $c_1=h(\gamma^{-1} L \mathbf{M}_1+\mathbf{M}_2)$.
    Set $\lambda=\max(\alpha, 1+\alpha/2)$. For $(c_k)_{k\in \mathbb{N}}$ we observe
    \begin{align*}
        a_{k+1}+b_{k+1}& \le c_{k+1} = h(\gamma^{-1} L \mathbf{M}_1+\mathbf{M}_2)+ h\gamma \lambda\sum_{i=1}^{k-1} c_i +h\gamma \lambda c_k
        \\ & = (1+ h\gamma \lambda)c_k=(1+ h\gamma \lambda)^k c_1\le e^{h\gamma \lambda k}c_1= e^{h\gamma \lambda k} h(\gamma^{-1} L \mathbf{M}_1+\mathbf{M}_2) .
    \end{align*}
    
    Then, by the choice of $l$, \eqref{eq:M1}, \eqref{eq:M2} and \eqref{eq:alpha} , it holds
    \begin{align*}
        e^{h\gamma \lambda k} h(\gamma^{-1} L \mathbf{M}_1+\mathbf{M}_2)&= e^{1}h \Big(\gamma^{-1} L \frac{\sqrt{d}}{\gamma(1+\alpha)}+\frac{\alpha}{\gamma(1+\alpha)}(\sqrt{L}+\gamma)\sqrt{d}+2\alpha \sqrt{\frac{d}{1+\alpha}}\Big)
        \\ & = e^1 h (L\gamma^{-2} \sqrt{d})\Big(\frac{3 +\sqrt{L\gamma^{-2}}}{1+2L\gamma^{-2}}+ \frac{4}{\sqrt{1 +2L\gamma^{-2}}}\Big)\le 20 h L \gamma^{-2} \sqrt{d}
    \end{align*}
    and hence
    \begin{align}
        \mathcal{W}_{\rho}(\mu_{\infty},\mu_{h,\infty}) \le  (1-e^{-\frac{c}{ \gamma\max(\alpha, 1+ \alpha/2)}})^{-1}20 h L \gamma^{-2} \sqrt{d}\le  \Big(1+\frac{ \gamma (1+\alpha)}{c}\Big) 20 h L \gamma^{-2} \sqrt{d}.
    \end{align}
    where we used in the second step that $1/(1-e^{-x})\le 1+1/x$ for $x>0$.
    
\end{proof}

\begin{proof}[Proof of \Cref{thm:complexity_guarantee}]
    Applying the triangle inequality and combining \Cref{thm:contra_EM} and \Cref{thm:strong_accuracy}, it holds for all $k\in\mathbb{N}$
    \begin{align*}
        \mathcal{W}_{\rho}(\mu_{\infty}, \nu \pi_h^k)&\le \mathcal{W}_{\rho}(\mu_{\infty},\mu_{h,\infty})+\mathcal{W}_{\rho}(\mu_{h,\infty}, \nu\pi_h^k)
        \\ & \le h\Big(1+\frac{ \gamma(1+\alpha)}{c}\Big)20 L \gamma^{-2} \sqrt{d} +e^{-chk}\mathcal{W}_{\rho}(\mu_{h,\infty}, \nu).
    \end{align*}
    %where $\mathbf{M}_3$ is given in \eqref{eq:M3}. 
    The bound in $\mathcal{W}_1$ is obtained by using the equivalence of the distance $\rho$ and the Euclidean distance in $\mathbb{R}^{2d}$.
\end{proof}

\subsection{BU and UBU} \label{sec:proofs_UBU}
To show \Cref{thm:BU_contr}, we first prove a local contraction result for the distance $r_l$.
\begin{proposition}\label{prop:case1_BU}
Let the potential $U$ be of the form $U := x^{T}Kx + G(x)$, where the symmetric and positive definite matrix $K$ satisfies $\kappa I_{d}\prec K \prec L_{K}I_{d}$ and $ G$ is convex outside a Euclidean ball, i.e. $(\nabla G(x)-\nabla G(y))\cdot (x-y)\ge 0$ for all $x,y \in \mathbb{R}^{d}$ such that $|x-y| > R$, now consider two iterates of the BU scheme $(\mathbf{X}_{k},\mathbf{V}_{k})_{k \in \mathbb{N}}$ and $(\mathbf{X}'_{k},\mathbf{V}'_{k})_{k \in \mathbb{N}}$ with synchronously coupled noise increments and metric $r_{l}$ between the iterates.  If $r^{2}_{l}((\mathbf{X}_{k},\mathbf{V}_{k}),(\mathbf{X}'_{k},\mathbf{V}'_{k}))\ge \mathcal{R}$ at iteration $k \in \mathbb{N}$, $ h < \min\{\frac{\gamma}{55L_K},\frac{1}{15\gamma }\}$ and $L_G\gamma^{-2}  \le \kappa/(13 L_G)$ we have that
\[
r^{2}_{l}((\mathbf{X}_{k+1},\mathbf{V}_{k+1}),(\mathbf{X}'_{k+1},\mathbf{V}'_{k+1})) \leq \Big(1-\frac{7}{8}\tau\gamma h\Big)r^{2}_{l}((\mathbf{X}_{k},\mathbf{V}_{k}),(\mathbf{X}'_{k},\mathbf{V}'_{k})),
\]
where $\tau = \min\{\frac{\kappa}{6\gamma^{2}},\frac{1}{16}\}$. If $R=0$, $\mathcal{R}=0$ and the restriction on $\gamma$ improves to $L_G \gamma^{-2}\le \frac{1}{6}$.
\end{proposition}
\begin{proof}
We have that
\[
r_{l}^{2}((\mathbf{X}_{k+1},\mathbf{V}_{k+1}),(\mathbf{X}'_{k+1},\mathbf{V}'_{k+1})) = (Z_{k},W_{k}) P^{T} M P \cdot (Z_{k},W_{k}),
\]
where 
\[
M = \begin{pmatrix}
  \gamma^{-2} K + (1-2\tau)^{2}/2I_d & (1-2\tau)/2\gamma I_d\\
      (1-2\tau)/2\gamma I_d &  \gamma^{-2}I_d
\end{pmatrix},
\textnormal{ and }
P = \begin{pmatrix}
  I_d -\frac{1-\eta}{\gamma}h(K+Q) & \left(\frac{1-\eta}{\gamma}\right)I_d\\
      -h\eta(K + Q) &  \eta I_d
\end{pmatrix},
\]
% \kati{Isn't there the red term missing? It should not cause any problems in the calculation since it of order $h^2$. }
% \pete{Yes you are right I copied the bit from Section 4.3, I think the same term is missing there, I added it in red. I will correct the rest.}
% \kati{Thanks! I only changed it the proof of theorem 4, but I forgot to change the coupling.}
where $\eta = \exp{\left(-\gamma h\right)}$, $K$ is the matrix defined by the quadratic term in the potential and $Q$ is defined by 
\[
Q = \int^{1}_{t=0}\nabla^{2}G(\mathbf{X}_{k} + t(\mathbf{X}'_{k}-\mathbf{X}_{k}))\rmd t,
\]
where $G$ is the non-quadratic term in the potential and $Q \succ 0$ for $|\mathbf{X}_{k}-\mathbf{X}'_{k}|>R$ and $-L_{G}I_{d} \prec Q \prec L_{G} I_{d}$ otherwise. 
It holds 
\begin{align*}
    r_{l}^{2}((\mathbf{X}_{k+1},\mathbf{V}_{k+1}),(\mathbf{X}'_{k+1},\mathbf{V}'_{k+1})) &  = r_l^2((\mathbf{X}_{k},\mathbf{V}_{k}),(\mathbf{X}'_{k},\mathbf{V}'_{k}))
    \\ & +(Z_k,W_k)^T ( h (M P_1+P_1^T M)+ h^2 P_1^T M P_1  ) (Z_k,W_k)
\end{align*}
with 
\begin{align} \label{eq:P_1}
    P_1=  \begin{pmatrix}
        -\frac{1-\eta}{\gamma}(K+Q) & \frac{1-\eta}{\gamma h}I_d \\ -\eta \left(K + Q \right)& -\frac{1-\eta}{h} I_d
    \end{pmatrix}.
\end{align}
It is sufficient to show that for all $(z,w)\in \mathbb{R}^{2d}$ with $r_l^2((z,w))\ge \mathcal{R}$, $(z,w)^T ( h (M P_1+P_1^T M)+ h^2 P_1^T M P_1  ) (z,w)\le - \gamma \tau h (z,w)^TM(z,w) $.
It holds
\begin{align*}
    &h (M P_1+P_1^T M)= h \begin{pmatrix}
        \frac{1-2\tau}{\gamma}(-K-Q) & -\gamma^{-2}Q - \tau(1-2\tau) I_{d}\\  -\gamma^{-2}Q - \tau(1-2\tau)I_{d}& -\frac{1 + 2 \tau}{\gamma} I_d
    \end{pmatrix}+ \begin{pmatrix}
        A_{h^{2}} & B_{h^{2}}\\ B_{h^{2}}& C_{h^{2}}
    \end{pmatrix}\\
    &\prec h \begin{pmatrix}
        \frac{1-2\tau}{\gamma}(-K-Q)+ 2\gamma^{-3} Q^2 & -\tau (1-2\tau)I_d \\ -\tau (1-2\tau)I_d& (-2\tau \gamma^{-1} -\frac{\gamma^{-1}}{2}) I_d
    \end{pmatrix} + \begin{pmatrix}
        A_{h^{2}} & B_{h^{2}}\\ B_{h^{2}}& C_{h^{2}}
    \end{pmatrix},
\end{align*}
since for all $(z,w)\in \mathbb{R}^{2d}$, $z^T(-Q\gamma^{-2})w\le \gamma^{-3}z^TQ ^2 z + 1/4 \gamma^{-1}  |w|^2 $ and  $w^T(-Q\gamma^{-2})z\le \gamma^{-3}z^TQ ^2 z + 1/4 \gamma^{-1} |w|^2$ and where
\begin{align*}
    A_{h^{2}} &=h(1-\eta)\frac{1-2\tau}{\gamma }(K+Q) -h\frac{1-\eta}{\gamma}\left(\gamma^{-2}K + \frac{(1-2\tau)^{2}}{2} I_{d}\right)(K+Q)\\
    &-h\frac{1-\eta}{\gamma}(K+Q)\left(\gamma^{-2}K + \frac{(1-2\tau)^{2}}{2} I_{d}\right)\\
    B_{h^{2}} &= \frac{(1-\eta - \eta \gamma h)}{\gamma^{3}}K + h(1-\eta)Q/\gamma^{2} + \tau(1-2\tau)\frac{\gamma h - 1 +\eta}{\gamma }I_{d} -h\frac{(1-\eta)(1-2\tau)}{2\gamma^{2}}(K+Q)\\
    C_{h^{2}} &= \frac{(h\gamma - 1 +\eta)(1+2\tau)}{\gamma^{2}}I_{d}
\end{align*}
and using that $|h\gamma -1 + \eta| \leq \frac{h^{2}\gamma^{2}}{2}$, $|h\gamma\eta -1 + \eta| \leq |h\gamma -1 + \eta| + h\gamma(1-\eta) \leq  \frac{3h^{2}\gamma^{2}}{2}$ for $h< \frac{1}{2\gamma}$, $w^TQz\le \frac{1}{2}\gamma^{-1}z^TQ ^2 z + \frac{1}{2} \gamma |w|^2$ and $w^TKz\le \frac{1}{2}\gamma^{-1}z^TK^2 z + \frac{1}{2} \gamma |w|^2$ we have
\begin{align*}
    &\begin{pmatrix}
        A_{h^{2}} & B_{h^{2}}\\ B_{h^{2}}& C_{h^{2}}
    \end{pmatrix} \prec \begin{pmatrix}
       1 & 0\\ 0& 0
    \end{pmatrix}  \otimes \Bigg( A_{h^{2}}+ \frac{3h^{2}Q^{2}}{2\gamma^{2}} + 2\frac{h^{2}K^{2}}{\gamma^{2}} + \frac{h^{2}\tau(1-2\tau)\gamma^{2}}{2} I_{d} \Bigg)
    \\
    &+\begin{pmatrix}
       0 & 0\\ 0& 1
    \end{pmatrix} \otimes \left(\frac{h^{2}(1+2\tau)}{2}I_{d} +  {h^{2}}\left(\frac{7}{2} + \tau(1-2\tau)/2\right) I_{d}\right)
\end{align*}

We have that $z^{T}(K+Q)w \leq \frac{1}{2\gamma}z^{T}(K+Q)^{2}z + \frac{\gamma}{2}|w|^{2}$, $(K+Q)^{T}K(K+Q) \prec L_{K}(K+Q)^{2}$ as $K$ is symmetric and positive definite and hence we can take the square root of $K$. Using these identities and also using the aforementioned inequalities for $h$ and $\gamma$ and $h<\frac{1}{12\gamma}$ we have
\begin{align*}
    &h^2 P_1^T M P_1 =\\
    &\begin{pmatrix}
        \frac{h^{2}(K+Q)\left(2(1-\eta)^{2}K + \gamma^{2}\left((1-2\tau)^{2} + 4\eta\tau(1-2\tau) + \eta^{2}(1+4\tau^{2})\right)\right)(K+Q)}{2\gamma^{4}}& \frac{-h(1-\eta)(K+Q)((1-\eta)K + \gamma^{2}(\tau(2\tau-1) - \eta/2-2\eta\tau^{2})I_{d})}{\gamma^{4}}\\
     \frac{-h(1-\eta)((1-\eta)K + \gamma^{2}(\tau(2\tau-1) - \eta/2-2\eta\tau^{2})I_{d})(K+Q)}{\gamma^{4}} & \frac{(1-\eta)^{2}(2K + \gamma^{2}(1+4\tau^{2})I_{d})}{2\gamma^{4}}
    \end{pmatrix}\\
    &\prec 
     h^{2}\begin{pmatrix}
       \frac{5}{2}(K+Q)^{2}\gamma^{-2} & 0 \\  0 & K\gamma^{-2} + \frac{5}{2}I_{d}
    \end{pmatrix}
\end{align*}
assuming that $\tau \leq 1/16$.
Using the fact that $(K+Q)^{2} \prec 2K^{2} + 2Q^{2} \prec 2L_{K}K + 2Q^{2}$ and $KQ,QK \prec 
L_{K}K+Q^{2}$. Combining the previous estimates we obtain
\begin{align*}
    &h(MP_{1} + P^{T}_{1}M) + h^{2}P^{T}_{1}MP_{1} \prec h \begin{pmatrix}
        \frac{1-2\tau}{\gamma}(-K-Q)+ 2\gamma^{-3} Q^2 & -\tau (1-2\tau)I_d \\ -\tau (1-2\tau)I_d& (-2\tau \gamma^{-1} -\frac{\gamma^{-1}}{2}) I_d
    \end{pmatrix}\\
    &+ h^{2}\begin{pmatrix}
       (1-2\tau)\left(2\tau\frac{1-\eta}{\gamma h}(K+Q) +\frac{\tau \gamma^{2}}{2}I_{d}\right) + 9(Q^{2} + L_{K}K)\gamma^{-2} & 0 \\ 0& K\gamma^{-2} + 7I_{d}
    \end{pmatrix}
\end{align*}
We define $C_{\eta,\tau}:= 1-2\tau(1-\eta)$, where $\frac{7}{8}\leq C_{\eta,\tau}\leq 1$.
By the conditions on $h$ it holds that $h^{2}\frac{9L_{K}}{\gamma^{2}}K \prec h\frac{C_{\eta,\tau}(1-4\tau)}{4\gamma}K$, $h(L_{k}\gamma^{-2}+7)\leq \frac{1}{2\gamma} $ and $9hL^{2}_{G}\gamma^{-2} \leq 2\gamma^{-3}L^{2}_{G}$
\begin{align*}
    &\prec h \begin{pmatrix}
        \frac{C_{\eta,\tau}(1-2\tau)}{\gamma}(-K-Q)+\frac{C_{\eta,\tau}(1-4\tau)}{4\gamma}K+4\gamma^{-3} L^2_{G}I_{d} +h\tau \gamma^{2}(1-2\tau)/2 I_{d} & -\tau (1-2\tau)I_d \\ -\tau (1-2\tau)I_d& -2\tau \gamma^{-1} I_d
    \end{pmatrix}.
\end{align*}
By assumption on $G$ and $\gamma$ and the choice of $\tau$, we observe  
\begin{align*}
    \frac{C_{\eta,\tau}(1-2\tau)}{\gamma}z^T (-Q) z \le \frac{C_{\eta,\tau}(1-2\tau)}{\gamma} L_G \1_{|z|\le R}|z|^2 \le C_{\eta,\tau}L_G \gamma^{-1} R^2
\end{align*}
and due to the condition on $\gamma$ and $\tau < \kappa/6\gamma^{2}$
\begin{align*}
    &-\frac{3C_{\eta,\tau}(1-4\tau)}{4\gamma}z^T K z +4\gamma^{-3} L_G^2|z|^2  + \frac{h\tau}{2} \gamma^{2}(1-2\tau)|z|^{2}\\
    &\le -\frac{3C_{\eta,\tau}}{8\gamma}z^T K z + \frac{h\tau}{2} \gamma^{2}(1-2\tau)|z|^{2} \le -2 \tau \gamma \frac{(1-2\tau)^2}{2}|z|^{2}.
\end{align*}
Hence, 
\begin{align*}
    &(z,w)^T\Big(h (M P_1+P_1^T M)+ h^2 P_1^T M P_1\Big) (z,w)
    \\ & \le C_{\eta,\tau} h(z,w)^T  \begin{pmatrix}
        -2\tau \gamma \Big( \gamma^{-2}K +\frac{(1-2\tau)^2}{2}\Big)  & -\tau (1-2\tau)I_d \\ -\tau (1-2\tau)I_d& -2\tau \gamma^{-1} I_d
    \end{pmatrix} (z,w) + C_{\eta,\tau}hL_G \gamma^{-1} R^2
    \\ & \le -2C_{\eta,\tau}\tau \gamma h  (z,w)^TM(z,w)+ C_{\eta,\tau}hL_G \gamma^{-1} R^2 \le -\frac{7}{8}\tau \gamma h  (z,w)^TM(z,w)
\end{align*}
% \pete{Sorry, I think there should be a factor of 7/8 here. This is because of $7/8 <C_{\eta,\tau} < 1$. I added the extra preconstants above in red to improve this to the 7/8 constant.}
% \kati{I changed it in the theorem}
% \pete{thank you!}
where the last step holds by the definition of $\mathcal{R}$. We note that for $R=0$, $Q$ is positive definite and we can bound directly $-\frac{C_{\eta,\tau}(1-2\tau)}{\gamma}Q+4\gamma^{-3} Q^2\prec 0$, which yields the result.
\end{proof}

\begin{proof}[Proof of \Cref{thm:BU_contr}]
    Fix $k\in \mathbb{N}$. Consider $\mathbf{X}_k,\mathbf{V}_k,\mathbf{X}_k', \mathbf{V}_k'\in \mathbb{R}^d$. As for the Euler scheme, we write $Z_k= \mathbf{X}_k-\mathbf{X}_k'$, $W_k=\mathbf{V}_k-\mathbf{V}_k'$ and $q_k=Z_k+\gamma^{-1} W_k$. 
    We show contraction separately for the synchronous coupling and the coupling given by \eqref{eq:coupl_contr_BU}.
    
    \textit{Case 1:} If $D_{\mathcal{K}}+\epsilon r_l((\mathbf{X}_k, \mathbf{V}_k),(\mathbf{X}_k', \mathbf{V}_k'))\le r_s((\mathbf{X}_k, \mathbf{V}_k),(\mathbf{X}_k', \mathbf{V}_k')$, then the synchronous coupling is applied and it holds $\rho((\mathbf{X}_k,\mathbf{V}_k),(\mathbf{X}_k',\mathbf{V}_k'))=f(D_{\mathcal{K}}+\epsilon r_l((\mathbf{X}_k,\mathbf{V}_k),(\mathbf{X}_k',\mathbf{V}_k')))$.

    We observe
    \begin{align*}
        \rho_{k+1}:=\rho((\mathbf{X}_{k+1},\mathbf{V}_{k+1}),(\mathbf{X}_{k+1}',\mathbf{V}_{k+1}'))\le f(D_{\mathcal{K}}+\epsilon r_l(k+1)), 
    \end{align*}
    where $r_l(k+1)=r_l((\mathbf{X}_{k+1},\mathbf{V}_{k+1}),(\mathbf{X}_{k+1}',\mathbf{V}_{k+1}'))$
    By Proposition~\ref{prop:case1_BU} and \eqref{eq:condition_h_BU} it holds
    \begin{align*}
        r_l(k+1) & \le \sqrt{1-\tau \gamma h}r_l(k) \leq \left(1-\frac{7\tau \gamma h}{16}\right)r_l(k).
    \end{align*}
    %Note that the condition on $h$ follows directly from \eqref{eq:condition_h_BU}.
    By concavity of $f$, we obtain
    \begin{equation} \label{eq:contr_largedist_BU}
    \begin{aligned}
        \mathbb{E}[\rho_{k+1}-\rho_k]&\le \mathbb{E}[f(D_{\mathcal{K}}+\epsilon r_l(k+1))-f(D_{\mathcal{K}}+\epsilon r_l(k))]
        \\ & \le f'(R_1)\epsilon\mathbb{E}[r_l(k+1)-r_l(k)]
        \le -f'(R_1) c_1 h r_l(k),
    \end{aligned}
    \end{equation}
    where $c_{1} = 7\epsilon \tau \gamma/16$ with $\tau$ given in \Cref{prop:case1_BU}. By \eqref{eq:equi_dist} as well as $f(r)\le r$,
    \begin{align*}
        -f'(R_1) c_1 h r_l(k)\le -f'(R_1)h \frac{c_1 r_l(k)}{D_{\mathcal{K}}+\epsilon r_l(k)}\rho_k\le -f'(R_1) c_1 h \mathcal{E}\rho_k.
    \end{align*}
    Hence, 
    \begin{align*}
        \mathbb{E}[\rho_{k+1}]\le (1-f'(R_1) c_1 \mathcal{E} h) \rho_k.
    \end{align*}
    
    \textit{Case 2:} If $D_{\mathcal{K}}+\epsilon r_l((\mathbf{X}_k, \mathbf{V}_k),(\mathbf{X}_k', \mathbf{V}_k'))> r_s((\mathbf{X}_k, \mathbf{V}_k),(\mathbf{X}_k', \mathbf{V}_k'))$, the coupling \eqref{eq:coupl_contr_BU} is applied. 
    By the definition of the BU-scheme and \eqref{eq:coupl_contr_BU}, it holds for the process $(Z_k,q_k)_{k\in \mathbb{N}}$
    \begin{align}
        \begin{cases}
            Z_{k+1} &= Z_k+ (1-\eta ) (q_k-Z_k)-\frac{h(1-\eta)}{\gamma} (\nabla U(\mathbf{X}_k)-\nabla U(\mathbf{X}_k'))  
            \\ & + \sqrt{2 \gamma ^{-1} h}\Big( 1- \frac{1- \eta }{\gamma h }\Big)  \Xi_{k+1}
            \\ q_{k+1}&= q_k - h\gamma^{-1} (\nabla U(\mathbf{X}_k)-\nabla U(\mathbf{X}_k'))+ \sqrt{2\gamma^{-1} h}\Xi_{k+1}
        \end{cases}
    \end{align}
    with $\eta=e^{-\gamma h}$.
    %\kati{I think I forgot a term in the equation of $Z_{k+1}$}
    To show contraction for this scenario, we split the proof in three steps:

    \textit{Step 1:} Let $|Z_k|\ge 4 |q_k|$. 
    For $|Z_{k+1}|$ and $|q_k|$ it holds
    \begin{align*}
        |Z_{k+1}|&\le |Z_k|\eta + \Big|\eta- \frac{1-\eta}{\gamma h}\Big||q_k|+\frac{h(1-\eta)}{\gamma}L|Z_k| 
    + \sqrt{2\gamma^{-1} h}\Big( 1- \frac{1-\eta}{\gamma h }\Big) |\hat{q}_k+ \Xi_{k+1}|, 
        \\  |q_{k+1}|&\le h\gamma^{-1} L|Z_{k}| + \sqrt{2\gamma^{-1} h}|\hat{q}_k+ \Xi_{k+1}|.
    \end{align*}
    We observe
    \begin{align*}
        \mathbb{E}[\rho_{k+1}-\rho_k]&\le f'(r_s(k)) \mathbb{E}[r_s(k+1)-r_s(k)]
        \\ & \le f'(r_s(k))\mathbb{E}\Big[\alpha |Z_k|\Big( \eta + \frac{h\gamma}{2}+ \frac{h\gamma \alpha (1-\eta)}{2}\Big)  + \alpha \Big|\eta-\frac{1- \eta}{\gamma h}\Big| |q_k|
        \\ & \qquad +\sqrt{2 \gamma ^{-1} h}\Big( 1+ |1- \frac{1- \eta }{\gamma h }|\alpha \Big)|\hat{q}_k+  \Xi_{k+1}|-|q_k|-\alpha |Z_k| \Big]
    \end{align*}
    By \eqref{eq:exp_1moment}, it holds $\mathbb{E}[|\hat{q}_k+\Xi_{k+1}|]=\hat{q}_k$. Hence, 
    \begin{align*}
        \mathbb{E}[\rho_{k+1}-\rho_k]&\le f'(r_s(k))\Big[\alpha |Z_k|\Big[\Big( \eta + \frac{h\gamma}{2}+ \frac{h\gamma \alpha (1-\eta)}{2}\Big) -1\Big] 
        \\ & +  \alpha \Big(\Big|\eta-\frac{1- \eta}{\gamma h}\Big|+\Big|1- \frac{1- \eta }{\gamma h }\Big| \Big)|q_k|\Big]
        \\ & \le f'(r_s(k))\Big[\alpha |Z_k|\Big[\Big( \eta + \frac{h\gamma}{2}+ \frac{h\gamma \alpha (1-\eta)}{2}\Big) -1\Big] + \alpha(1-\eta ) |Z_k|/4\Big],
    \end{align*}
    where we used in the second step $\eta=e^{-\gamma h} \le \frac{1-\eta}{\gamma h}\le 1$ and $|q_k|\le |Z_k|/4$.
    Using that by \eqref{eq:condition_h_BU}, $ \frac{h\gamma}{2}+ \frac{h\gamma \alpha}{2}\le \frac{1}{12}+\frac{1}{24}=\frac{1}{8}$, we obtain
    \begin{align*}
        \mathbb{E}[\rho_{k+1}-\rho_k]&\le f'(r_s(k))\alpha |Z_k|\Big[\frac{h\gamma}{2}\eta+\Big(\frac{3}{4}-\frac{h\gamma}{2}- \frac{h\gamma \alpha}{2}\Big)(\eta -1)\Big]
        \\ & \le f'(r_s(k))\alpha |Z_k|\Big(\frac{h\gamma}{2}\eta+\frac{5}{8}\eta (-\gamma h)\Big)\le f'(r_s(k))\alpha |Z_k|\Big(-\frac{h\gamma}{8}\Big)\eta
        \\ & \le -f'(R_1) \eta \min\Big(\frac{h\gamma}{16},\frac{h\gamma \alpha}{4}\Big) \rho_k.
    \end{align*}
    
    \textit{Step 2:} Let $|Z_k|<4 |q_k|$ and $|q_k|\ge \sqrt{2\gamma^{-1} h}$. 
    First, using \eqref{eq:condition_h_BU} we observe that similarly to the first case
    \begin{align*}
        \alpha |Z_k|[(\eta-1)+ \frac{h\gamma }{2} + \frac{\alpha \gamma h }{2}(1-\eta)-1 ]\le 0 
    \end{align*}
    and hence, we can bound
    \begin{align}
        r_s(k+1)&\le \alpha |Z_k|+ \alpha \Big|\frac{1-\eta}{\gamma h}-\eta\Big| |q_k| + \sqrt{2\gamma^{-1} h } \Big(1+ \Big|1-\frac{1-\eta}{\gamma h}\Big|\alpha\Big) |\hat{q}_k+\Xi_{k+1}|  =:\bar{r}_s(k). \label{eq:bar_r_s}
    \end{align}
    Then, by concavity of $f$
    \begin{align}
        \mathbb{E}[\rho_{k+1}-\rho_k]&\le \mathbb{E}[f(\bar{r}_s(k))-f(r_s(k))] \nonumber
        \\ & \le f'(r_s(k))\mathbb{E}[\bar{r}_s(k)-r_s(k)]+ \mathbb{E}\Big[\int_{r_s(k)}^{\bar{r}_s(k)} (\bar{r}_s(k)-t) f''(t) \rmd t\Big] \label{eq:expec_case2}
    \end{align}
    For the first term, it holds by \eqref{eq:exp_1moment} and $\eta \le \frac{1-\eta}{\gamma h}\le 1$,
    \begin{align*}
        f'(r_s(k))\mathbb{E}[\bar{r}_s(k)-r_s(k)]\le f'(r_s(k))\alpha (1-\eta)|q_k|.
    \end{align*}
    To bound the second term we define
    \begin{align}
        \bar{A}= \{\Xi_{k+1} =-\hat{q}_k\} \quad \text{and} \quad \bar{A}^c = \{\Xi_{k+1} \neq -\hat{q}_k\}=\{\Xi_{k+1} = 2(e_k\cdot \xi_{k+1}^{(1)})e_k\}
    \end{align}
    and split the term in 
    \begin{align*}
        \mathbb{E}\Big[\int_{r_s(k)}^{\bar{r}_s(k)} & (\bar{r}_s(k)-t) f''(t) \rmd t\Big]
        \\ & = \mathbb{E}\Big[\int_{r_s(k)}^{\bar{r}_s(k)} (\bar{r}_s(k)-t) f''(t) \rmd t \1_{\bar{A}}\Big]+\mathbb{E}\Big[\int_{r_s(k)}^{\bar{r}_s(k)} (\bar{r}_s(k)-t) f''(t) \rmd t\1_{\bar{A}^c}\Big].
    \end{align*}
    Since $f''$ is non-positive, the first term satisfies
    \begin{align}
        \mathbb{E}\Big[&\int_{r_s(k)}^{\bar{r}_s(k)} (\bar{r}_s(k)-t) f''(t) \rmd t \1_{\bar{A}}\Big]\le \mathbb{E}\Big[\int_{r_s(k)}^{\frac{1}{2}(\bar{r}_s(k)+r_s(k))} (\bar{r}_s(k)-t) f''(t) \rmd t \1_{\bar{A}}\Big] \nonumber
        \\ &  \le \mathbb{E}\Big[\frac{3(\bar{r}_s(k)-r_s(k))^2}{8}\max_{t\in [\frac{1}{2}(\bar{r}_s(k)+r_s(k)), r_s(k)]} f''(t) \1_{\bar{A}}\Big] \nonumber
        \\&  \le \frac{3(\alpha|\eta- \frac{1-\eta}{\gamma h}|-1)^2|q_k|^2}{8}\max_{t\in [r_s(k)+(\alpha|\eta- \frac{1-\eta}{\gamma h}|-1)/2|q_k|, r_s(k)]} f''(t) \mathbb{E}[\1_{\bar{A}}]. \label{eq:expec_case2_1}
    \end{align}
    For the second term we define the set
    \begin{align*}
        A'=\{\sqrt{2\gamma h^{-1} }|\hat{q}_k+ \Xi_{k+1} |\le |q_k|-\sqrt{2\gamma^{-1} h}\}.
    \end{align*}
    Then,
    \begin{align}
        \mathbb{E}\Big[\int_{r_s(k)}^{\bar{r}_s(k)} (\bar{r}_s(k)-t) f''(t) \rmd t\1_{\bar{A}^c}\Big]\le \mathbb{E}\Big[\frac{3(\bar{r}_s(k)-r_s(k))^2}{8}\max_{t\in [\frac{1}{2}(\bar{r}_s(k)+r_s(k)),r_s(k)]}f''(t)  \1_{\bar{A}^c}\1_{A'}\Big].\label{eq:expec_case2_2}
    \end{align}
    Next, we use the observation \eqref{eq:f_conseq2} for $f''(t)$ for $t\in [\frac{1}{2}(r_s(k)+\bar{r}_s(k)),r_s(k)]$.
    % we make the following observation for the second derivative of $f$.
    % For $t\in [\frac{1}{2}(r_s(k)+\bar{r}_s(k)),r_s(k)]$ it holds $t\ge \frac{1}{2}r_s(k)$, $\exp(-128 \alpha \gamma^2 t^2/2)\ge \exp(-128 \alpha \gamma^2 r_s(k)/2)$ and  $\Psi(t)\ge \Psi(r_s(k))/2$. For $s>t$, $\exp(-128\alpha \gamma^2 t^2/2)\ge \exp(-128 \alpha \gamma^2 s^2/2)$. Hence, for $t\in [\frac{1}{2}(r_s(k)+\bar{r}_s(k)),r_s(k)]$
    % \begin{align*}
    %     f''(t)&= -128 \alpha \gamma^2 t f'(t)- \frac{\hat{c}\gamma}{2}\int_0^t \phi(s)\rmd s  \le -128 \alpha \gamma^2 \frac{r_s(k)}{2}f'(t) -  \frac{\hat{c}\gamma}{2}\int_0^{\frac{r_s(k)}{2}} \phi(s)\rmd s 
    %     \\ & \le -128 \alpha \gamma^2 \frac{r_s(k)}{4}f'(r_s(k)) -  \frac{\hat{c}\gamma}{4}\int_0^{\frac{r_s(k)}{2}} \phi(s)\rmd s  \le -128 \alpha \gamma^2 \frac{r_s(k)}{4}f'(r_s(k)) -  \frac{\hat{c}\gamma}{4}f(r_s(k))
    % \end{align*}
    By inserting the estimate in \eqref{eq:expec_case2_1} and \eqref{eq:expec_case2_2} and using $|q_k|\ge \sqrt{2\gamma^{-1} h }$, $|\eta- \frac{1-\eta}{\gamma h}|\le 1-\eta\le \gamma h$ and $\alpha \gamma h \le 1/2$, we obtain
    \begin{align}
        \mathbb{E}\Big[\int_{r_s(k)}^{\bar{r}_s(k)} (\bar{r}_s(k)-t) f''(t) \rmd t \1_{\bar{A}}\Big] 
         %& \le \frac{3(1/2)^22\gamma^{-1} h}{8}\Big(-128 \alpha \gamma^2 \frac{r_s(k)}{4}f'(r_s(k)) -  \frac{\hat{c}\gamma}{4}f(r_s(k))\Big) \mathbb{E}[\1_{\bar{A}}] \nonumber
         & \le \frac{3\gamma^{-1} h}{16}\Big(-128 \alpha \gamma^2 \frac{r_s(k)}{4}f'(r_s(k)) -  \frac{\hat{c}\gamma}{4}f(r_s(k))\Big) \mathbb{E}[\1_{\bar{A}}] \label{eq:expec_case2_3}
    \end{align}
    and 
    \begin{align*}
         \mathbb{E}\Big[&\int_{r_s(k)}^{\bar{r}_s(k)}(\bar{r}_s(k)-t) f''(t) \rmd t\1_{\bar{A}^c}\Big]
         \\ & \le \Big(-128 \alpha \gamma^2 \frac{r_s(k)}{4}f'(r_s(k)) -  \frac{\hat{c}\gamma}{4}f(r_s(k))\Big)\mathbb{E}\Big[\frac{3(\bar{r}_s(k)-r_s(k))^2}{8}\1_{\bar{A}^c}\1_{A'}\Big].
    \end{align*}
    The expectation on the right-hand side is bounded by 
    \begin{align*}
        \mathbb{E}&\Big[\frac{3(\bar{r}_s(k)-r_s(k))^2}{8}\1_{\bar{A}^c}\1_{A'}\Big]
        \\ & = \int_{-\infty}^{\infty} \frac{3(\alpha|\eta - \frac{1-\eta}{\gamma h }||q_k|+ \sqrt{2\gamma^{-1} h}(1+|1-\frac{1-\eta}{\gamma h}|\alpha)||\hat{q}_k|+2u|-|q_k|)^2}{8}\1_{\{||\hat{q}_k|+2u|\le |\hat{q}_k|-1\}}
        \\ & \qquad \qquad \cdot(\varphi(u)-\varphi(u+|\hat{q}_k|))^+ \rmd u
        \\ & = \int_{-\frac{|\hat{q}_k|}{2}}^{-\frac{1}{2}} \frac{3(\alpha|\eta - \frac{1-\eta}{\gamma h }||q_k|+ \sqrt{2\gamma^{-1} h}(1+|1-\frac{1-\eta}{\gamma h}|\alpha)2u)^2}{8}(\varphi(u)-\varphi(u+|\hat{q}_k|))^+ \rmd u
        \\ & \ge \frac{3\gamma^{-1}h}{4}\int_{-\frac{|\hat{q}_k|}{2}}^{-\frac{1}{2}}  \Big(\alpha|\eta - \frac{1-\eta}{\gamma h }|\frac{R_1}{\sqrt{2\gamma^{-1} h}}+ 2u\Big)^2(\varphi(u)-\varphi(u+|\hat{q}_k|))^+ \rmd u
        \\ & \ge \frac{3\gamma^{-1}h}{4}\int_{-\frac{|\hat{q}_k|}{2}}^{-\frac{1}{2}}  \Big(\alpha(1-\eta)\frac{R_1}{\sqrt{2\gamma^{-1} h}}+ 2u\Big)^2(\varphi(u)-\varphi(u+|\hat{q}_k|))^+ \rmd u
        \\ & \ge \frac{3\gamma^{-1}h}{4}\int_{-\frac{|\hat{q}_k|}{2}}^{-\frac{1}{2}}  \Big(1/2+ 2u\Big)^2(\varphi(u)-\varphi(u+|\hat{q}_k|))^+ \rmd u
        \\ & \ge \frac{3\gamma^{-1}h}{4}\int_{-\frac{|\hat{q}_k|}{2}}^{-\frac{1}{2}}  \frac{1}{4}(\varphi(u)-\varphi(u+|\hat{q}_k|))^+ \rmd u,
    \end{align*}
    where in the second last step we used $\frac{\alpha(1-\eta) R_1}{\sqrt{2\gamma^{-1} h}}\le \frac{\alpha\gamma h R_1}{\sqrt{2\gamma^{-1} h}} \le 1/2$. 
    Combining this bound and the expectation in \eqref{eq:expec_case2_3} yields
    \begin{align*}
        \mathbb{E}\Big[\frac{3(\bar{r}_s(k)-r_s(k))^2}{8}\1_{\bar{A}^c}\1_{A'}\Big]+ \frac{3\gamma^{-1} h}{16}\mathbb{E}[\1_{\bar{A}}]\ge \frac{3}{16} \gamma^{-1} h \int_{-\infty}^{-1/2} \varphi(u) \rmd u \ge  \frac{9}{160} \gamma^{-1} h.
    \end{align*}
    Inserting this bound back into the sum of \eqref{eq:expec_case2_1} and \eqref{eq:expec_case2_2} and inserting the sum back into \eqref{eq:expec_case2}, we obtain
    \begin{align*}
        \mathbb{E}[\rho_{k+1}-\rho_k]&\le f'(r_s(k)) \alpha (1-\eta)|q_k|+ \frac{9}{160}\gamma^{-1} h\Big(-128 \alpha \gamma^2 \frac{r_s(k)}{4}f'(r_s(k)) -  \frac{\hat{c}\gamma}{4}f(r_s(k))\Big)
        \\ & \le -\frac{9}{640}\hat{c}h f(r_s(k))= -\frac{9}{640}\hat{c}h \rho_k.
    \end{align*}

    \textit{Step 3:} Let $|Z_k|< 4 |q_k|$ and $|q_k|< \sqrt{2\gamma^{-1} h}$.  As in the second step it holds
    \begin{align}
        \mathbb{E}[\rho_{k+1}-\rho_k] %\le \mathbb{E}[f(\bar{r}_s(k))-f(r_s(k))] & 
        \le f'(r_s(k)) \alpha \gamma h |q_k|+ \mathbb{E}\Big[\int_{r_s(k)}^{\bar{r}_s(k)} (\bar{r}_s(k)-t)f''(t) \rmd t\Big] ,\label{eq:expec_case3}
    \end{align}
    where $\bar{r}_s(k)$ is given by \eqref{eq:bar_r_s}. 
    The last term is bounded by 
    \begin{align*}
        \mathbb{E}\Big[\int_{r_s(k)}^{\bar{r}_s(k)} (\bar{r}_s(k)-t)f''(t) \rmd t\Big] &\le \mathbb{E}\Big[\int_{\frac{1}{2}(\bar{r}_s(k)+r_s(k))}^{\bar{r}_s(k)} (\bar{r}_s(k)-t)f''(t) \rmd t \1_{A}\Big]
        \\ & \le \mathbb{E}\Big[ \frac{(\bar{r}_s(k)-r_s(k))^2}{8} \max_{t\in [r_s(k), \frac{1}{2}(\bar{r}_s(k)+r_s(k))]}f''(t)  \1_{A}\Big],
    \end{align*}
    where the set $A$ is given by 
    \begin{align} \label{eq:A_BU}
        A=\{|q_k|+2\sqrt{2\gamma^{-1} h} \le \sqrt{2\gamma^{-1}  h}|\hat{q}_k+\Xi_{k+1}|\le |q_k|+ 6 \sqrt{2\gamma^{-1} h }\}.
    \end{align}
    By construction of $A$, it holds
    \begin{align*}
        \frac{1}{2}(r_s(k)+\bar{r}_s(k))\ge \frac{1}{2}|q_k|+ \frac{1}{2}\sqrt{2\gamma^{-1} h} \Big(1+\Big|1-\frac{1-\eta}{\gamma h}\Big| \alpha\Big) |\hat{q}_k + \Xi_{k+1}|\ge \sqrt{2\gamma^{-1} h}.
    \end{align*}
    % Further, by \eqref{eq:f} it holds for all $t\in [\frac{1}{2}(\bar{r}_s(k)+r_s(k)),\bar{r}_s(k)]$
    % \begin{align}
    %     f''(t)&= -f'(t) 128 \alpha \gamma^2 t - \frac{\hat{c}\gamma }{2} \int_0^t \phi(s)\rmd s \nonumber
    %     \\ & \le - \frac{1}{2}\psi(r_s(k)) \phi(t) 128 \alpha \gamma^2 \sqrt{2\gamma^{-1} h}-\frac{\hat{c}\gamma }{2}\int_0^{r_s(k) + \sqrt{2\gamma^{-1} h}}\phi(x) \rmd s \nonumber
    %     \\ & \le -\frac{1}{2}\psi(r_s(k)) \phi(r_s(k)) 128 \alpha \gamma^2 \sqrt{2\gamma h} \frac{\phi(\bar{r}_s(k))}{\phi(r_s(k))} - \frac{\hat{c}\gamma}{2}\int_0^{r_s(k) + \sqrt{2\gamma^{-1} h}}\phi(x) \rmd s. \label{eq:bound_f''}
    % \end{align}
    By \eqref{eq:condition_h_BU}, it holds $8(7\alpha\gamma h +6)\le 50$ and $(7\alpha \gamma h +6)(7\alpha \gamma h+8)\le 50$ and therefore by \eqref{eq:A_BU}
    \begin{align*}
        &\bar{r}_s(k)^2 -r_s(k)^2
        \\ & = 2(\alpha |Z_k|+|q_k|)\Big[\alpha \Big|\eta-\frac{1-\eta}{\gamma h}\Big||q_k| + \Big(1+\alpha \Big|1- \frac{1-\eta}{\gamma h}\Big|\Big)\sqrt{2\gamma^{-1} h }|\hat{q}_k+\Xi_{k+1}|-|q_k|\Big]
        \\ & \quad + \Big(\alpha \Big|\eta-\frac{1-\eta}{\gamma h}\Big||q_k| + \Big(1+\alpha \Big|1- \frac{1-\eta}{\gamma h}\Big|\Big)\sqrt{2\gamma^{-1} h }|\hat{q}_k+\Xi_{k+1}|-|q_k|\Big)^2
        \\ & \le 2(4\alpha+1)|q_k|\Big(\alpha (1-\eta)|q_k|+ 6 \alpha (1-\eta)\sqrt{2\gamma^{-1} h}+ 6 \sqrt{2\gamma^{-1} h}\Big)
        \\ & \quad +\Big(\alpha (1-\eta)|q_k|+ 6 \alpha (1-\eta)\sqrt{2\gamma^{-1} h}+ 6 \sqrt{2\gamma^{-1} h}\Big)^2
        \\ & \le (2\gamma^{-1} h)\Big( 2(4\alpha+1)(7\alpha\gamma h +6)+ (7\alpha\gamma h +6)^2\Big)
        \\ & \le  (2\gamma^{-1} h)\Big( 8\alpha(7\alpha\gamma h +6)+ (7\alpha\gamma h +6)(7\alpha\gamma h +8)\Big)\le (2\gamma^{-1} h)(50 \alpha + 50).
    \end{align*}
    By \eqref{eq:condition_h_BU}, we observe
    \begin{align*}
         \frac{\phi(\bar{r}_s(k))}{\phi(r_s(k))}&=\exp\Big(-128\alpha \gamma^2 \frac{\bar{r}_s(k)^2 -r_s(k)^2}{2}\Big) \ge \exp\Big(-128\alpha \gamma^2 \frac{(2\gamma^{-1} h)(50 \alpha + 50)}{2}\Big) 
         \\ & \ge \exp(-2/3)\ge 1/2
    \end{align*}
    and also \eqref{eq:f_conseq2} holds. Using these bounds and \eqref{eq:f_conseq3}, yields for all $t\in [\frac{1}{2}(\bar{r}_s(k)+r_s(k)),\bar{r}_s(k)]$
    \begin{align*}
        f''(t)\le -\frac{1}{4}f'(r_s(k)) 128 \alpha \gamma^2 \sqrt{2\gamma^{-1} h} -\frac{\hat{c}\gamma}{4}\frac{\sqrt{2 \gamma^{-1} h}}{(4\alpha +1)|q_k|}f(r_s(k)),
    \end{align*}
    where we used that $r_s(k)\le (4\alpha+1)|q_k|$. Hence, since on the set $A$, $(\bar{r}_s(k)-r_s(k))^2\ge 8\gamma^{-1} h$, it holds
    \begin{align*}
         \mathbb{E}&\Big[\int_{r_s(k)}^{\bar{r}_s(k)} (\bar{r}_s(k)-t)f''(t) \rmd t\Big]
\le \frac{4(2\gamma^{-1}h)^{3/2}}{8} \Big( -f'(r_s(k)) 32 \alpha \gamma^2  -\frac{\hat{c}\gamma}{4}\frac{1}{(4\alpha +1)}f(r_s(k))\Big)\mathbb{E}[\1_{A}].
    \end{align*}
    The term $\mathbb{E}[\1_{A}]$ is bounded as in \eqref{eq:expecA}. Hence, 
    \begin{align*}
        \mathbb{E}&\Big[\int_{r_s(k)}^{\bar{r}_s(k)} (\bar{r}_s(k)-t)f''(t) \rmd t\Big] \le -4\gamma^{-1}  h \alpha \gamma^2 f'(r_s(k)) |q_k|- \frac{1}{32(4\alpha +1)} h \hat{c} f(r_s(k)).
    \end{align*}
    Inserting this bound back into \eqref{eq:expec_case3}, we obtain 
    \begin{align*}
        \mathbb{E}[\rho_{k+1}- \rho_k]\le - \frac{1}{32(4\alpha +1)} h \hat{c} f(r_s(k)).
    \end{align*}
    Putting the three steps together and combining the result with the first case, we obtain
    \begin{align*}
        \mathbb{E}[\rho_{k+1}- \rho_k]\le (1- ch) \rho_k
    \end{align*}
    with $c$ given in \eqref{eq:c_BU}.
\end{proof}

\begin{proof}[Proof of \Cref{thm:UBU_contr}]
    Since the scheme satisfies $(\mathcal{U}_{1/2}\mathcal{B}\mathcal{U}_{1/2})^{n} = \mathcal{U}_{1/2}(\mathcal{B}\mathcal{U})^{n-1}\mathcal{B}\mathcal{U}_{1/2}$, where $\mathcal{U}_{1/2}$ is a half step, we use the contraction result of \Cref{thm:BU_contr} for the steps $(\mathcal{B}\mathcal{U})^{n-1}$ and it remains to control the error induced by the steps $\mathcal{U}_{1/2}$ and $\mathcal{B}\mathcal{U}_{1/2}$. %For the first step $\mathcal{U}$ and  the last step $\mathcal{B}\mathcal{U}$ (here $\mathcal{U}$ is a half step), we use a synchronous coupling to show that they introduce an error which can be controlled for sufficiently $n\in\mathbb{N}$, such that we get a contraction result for $(\mathcal{U}\mathcal{B}\mathcal{U})^{n}$ steps.
    Using a synchronous coupling, we show that there exists a constant $\mathbf{C}$ such that for any $x,v,x',v'\in \mathbb{R}^d$ and $\xi^{(1)}, \xi^{(2)}\sim \mathcal{N}(0,I_d)$ 
    \begin{align*}
        &\rho((\mathbf{X}_U, \mathbf{V}_U), (\mathbf{X}_U', \mathbf{V}_U'))\le \mathbf{C}\rho((x,v),(x',v')), \qquad \text{and}
        \\ & \rho((\mathbf{X}_{BU}, \mathbf{V}_{BU}), (\mathbf{X}_{BU}', \mathbf{V}_{BU}'))\le \mathbf{C}\rho((x,v),(x',v')), 
    \end{align*}
    where 
    \begin{align*}
        & (\mathbf{X}_U, \mathbf{V}_U)=\mathcal{U}\Big(x,v,\frac{h}{2},\xi^{(1)}, \xi^{(2)}\Big), && (\mathbf{X}_{BU}, \mathbf{V}_{BU})=\mathcal{B}\Big(\mathcal{U}\Big(x,v,\frac{h}{2},\xi^{(1)}, \xi^{(2)}\Big),h\Big), 
        \\  & (\mathbf{X}_U', \mathbf{V}_U')=\mathcal{U}\Big(x',v',\frac{h}{2},\xi^{(1)}, \xi^{(2)}\Big), && (\mathbf{X}_{BU}', \mathbf{V}_{BU}')=\mathcal{B}\Big(\mathcal{U}\Big(x',v',\frac{h}{2},\xi^{(1)}, \xi^{(2)}\Big),h\Big).
    \end{align*}
    We write
    \begin{align*}        
        &Z_U=\mathbf{X}_U-\mathbf{X}_U', \quad  W_U=\mathbf{V}_U-\mathbf{V}_U', && Z_{BU}=\mathbf{X}_{BU}-\mathbf{X}_{BU}', \quad W_{BU}=\mathbf{V}_{BU}-\mathbf{V}_{BU}',
        \\ & q_U=Z_U+\gamma^{-1} W_U, && q_{BU}=Z_{BU}+\gamma^{-1} W_{BU} 
    \end{align*}
    and $z=x-x'$, $w=v-v'$ and $q=z+\gamma^{-1} w$.

    First, we assume that $x,v,x',v'\in \mathbb{R}^d$ are such that $\rho((x,v),(x',v))=f(r_s((x,v),(x',v)))$.
    %Then, 
    %\begin{align*}
    %    & \rho ((\mathbf{X}_U,\mathbf{V}_U),(\mathbf{X}_U',\mathbf{V}_U'))\le f(r_s((\mathbf{X}_U,\mathbf{V}_U),(\mathbf{X}_U',\mathbf{V}_U'))\le \alpha |Z_{U}|+|q_U| \quad \text{and}
    %    \\ & \rho ((\mathbf{X}_{BU},\mathbf{V}_{BU}),(\mathbf{X}_{BU}',\mathbf{V}_{BU}'))\le f(r_s((\mathbf{X}_{BU},\mathbf{V}_{BU}),(\mathbf{X}_{BU}',\mathbf{V}_{BU}'))\le \alpha |Z_{BU}|+|q_{BU}|.
    %\end{align*}
    By the construction of the $\mathcal{U}$ step it holds
    \begin{align*}
        \alpha |Z_{U}|+|q_U|&= \alpha|z+ (1-\exp^{-\gamma h /2})(q-z)|+|q|
        \\ &\le \alpha |z|\exp^{-\gamma h /2}+(1+\alpha\gamma h/2)|q|\le (1+\alpha \gamma h/2) (\alpha |z|+|q|).
    \end{align*}
    Since $f$ is concave and $(1+\alpha \gamma h/2) \ge 1$, it holds
    %Hence, if $\rho((x,v),(x',v))=f(r_s((x,v),(x',v)))$ holds,
    \begin{align*}
        \rho ((\mathbf{X}_U,\mathbf{V}_U),(\mathbf{X}_U',\mathbf{V}_U'))&\le f(r_s((\mathbf{X}_U,\mathbf{V}_U),(\mathbf{X}_U',\mathbf{V}_U'))
         \le f((1+\alpha \gamma h/2)r_s((x,v),(x',v')))
         \\ &\le (1+\alpha \gamma h/2)f(r_s((x,v),(x',v')))=(1+\alpha \gamma h/2)\rho((x,v),(x',v)).
        %f'(R_1)^{-1} (1+\alpha \gamma h/2) \rho((x,v),(x',v')).
    \end{align*}
    Analogously, it holds
    \begin{align*}
        \alpha |Z_{BU}|+|q_{BU}|&= \alpha|z+ (1-e^{-\gamma h /2})(q-z)|+|q-h\gamma^{-1} e^{-\gamma h/2}(\nabla U(x)-\nabla U(x'))|
        \\ &\le \alpha |z|e^{-\gamma h/2}+ h\gamma^{-1}L e^{-\gamma h/2}|z|+(1+\alpha\gamma h/2)|q|
        \\ & \le \alpha(1+h\gamma/2) e^{-\gamma h/2}|z|+ (1+\alpha\gamma h/2)|q| \le (1+\alpha \gamma h/2) (\alpha |z|+|q|).
    \end{align*}
    Hence, since $f$ is concave and $(1+\alpha \gamma h/2) \ge 1$, %for $\rho((x,v),(x',v))=f(r_s((x,v),(x',v)))$ 
    \begin{align*}
        \rho ((\mathbf{X}_{BU},\mathbf{V}_{BU}),(\mathbf{X}_{BU}',\mathbf{V}_{BU}')) & \le f(r_s((\mathbf{X}_{BU},\mathbf{V}_{BU}),(\mathbf{X}_{BU}',\mathbf{V}_{BU}')))
         %\le f((1+\alpha \gamma h/2)r_s((x,v),(x',v')))
         \\ &\le (1+\alpha \gamma h/2)\rho((x,v),(x',v)).
    \end{align*}
    Next, let $x,v,x',v'\in \mathbb{R}^d$ be such that $\rho((x,v),(x',v'))=f(D_{\mathcal{K}}+\epsilon r_l((x,v),(x',v')))$ holds.
    We observe that
    \begin{align*}
        r_l^2 ((\mathbf{X}_U,\mathbf{V}_U),(\mathbf{X}_U',\mathbf{V}_U'))= (z,w)^T (I_{2d} +P_2)^T M (I_{2d}+P_2) (z,w),
    \end{align*}
    with $M$ given in the proof of \Cref{prop:case1_BU} and
    \begin{align*}
        P_2 = \begin{pmatrix}
            0_d & (1-e^{\gamma h/2})\gamma^{-1} I_d \\ 0_d & (e^{-\gamma h/2}-1) I_d
        \end{pmatrix}
    \end{align*}
    where $0_d$ is the $d\times d$ zero matrix.
    We want to show that $(z,w)^T (P_2^T M+ MP_2 + P_2MP_2) \cdot (z,w)\le C r_l^2(z,w) $ for some positive constant $C>0$.
    It holds
    \begin{align*}
        P_2^T M+ MP_2 + P_2MP_2= \frac{1-e^{-\gamma h/2}}{\gamma}\begin{pmatrix}
           0_d & \gamma^{-2} K + \frac{-2\tau(1-2\tau)}{2}I_d \\  \gamma^{-2} K + \frac{-2\tau(1-2\tau)}{2}I_d & -\frac{2\tau+1}{\gamma}e^{-\gamma h/2} I_d
        \end{pmatrix}.
    \end{align*}
    For $(z,w)\in\mathbb{R}^{2d}$ it holds
    \begin{align*}
        w^T \gamma^{-2} K\cdot z+ z^T \gamma^{-2} K\cdot w \le \gamma^{-1} z^T K\cdot z+ \gamma^{-3} w^T K \cdot w. 
    \end{align*}
    Hence, 
    \begin{align*}
        &(z,w)^T (P_2^T M+ MP_2 + P_2MP_2) \cdot (z,w) 
        \\& \le \frac{1-e^{-\gamma h/2}}{\gamma}\Big( \gamma^{-1} z^T K\cdot z+ \gamma^{-3} w^T K \cdot w - 2 \tau (1-2\tau)z\cdot w -\frac{2\tau+1}{\gamma}e^{-\gamma h/2}|w|^2\Big) 
        \\& \le \frac{1-e^{-\gamma h/2}}{\gamma}\Big( -2\tau \gamma (z,w)^T M \cdot (z,w) + 2\gamma^{-1} z^T K\cdot z+ \gamma^{-3} w^T K \cdot w\Big)
        \\ & \le \frac{1-e^{-\gamma h/2}}{\gamma}2\gamma\max(1, L_K \gamma^{-2})(z,w)^T M \cdot (z,w) \le h \gamma\max(1, L_K \gamma^{-2})(z,w)^T M \cdot (z,w).
    \end{align*}
    Thus,
    \begin{align*}
        D_{\mathcal{K}}+\epsilon r_l((\mathbf{X}_U,\mathbf{V}_U),(\mathbf{X}_U',\mathbf{V}_U')) &\le D_{\mathcal{K}}+ \sqrt{1+h \gamma\max(1, L_K \gamma^{-2})} \epsilon r_l((x,v),(x',v')) 
        \\ & \le   \sqrt{1+h \gamma\max(1, L_K \gamma^{-2})}(D_{\mathcal{K}}+ \epsilon r_l((x,v),(x',v'))).
    \end{align*}
    By concavity of $f$
    \begin{align*}
        \rho ((\mathbf{X}_U,\mathbf{V}_U),(\mathbf{X}_U',\mathbf{V}_U'))& \le   \sqrt{1+h \gamma\max(1, L_K \gamma^{-2})} \rho((x,v),(x',v')).
    \end{align*}
    Similarly, we observe that
    \begin{align*}
        r_l^2 ((\mathbf{X}_{BU},\mathbf{V}_{BU}),(\mathbf{X}_{BU}',\mathbf{V}_{BU}'))&= (z,w)^T (P_{BU_{1/2}})^T M (P_{BU_{1/2}}) (z,w)
        \\ & =(z,w)^T ((I_{2d}+P_3)P_{BU})^T M ((I_{2d}+P_3)P_{BU}) (z,w),
    \end{align*}
    where $P_{BU}$ corresponds to a $\mathcal{BU}$ step and $P_{BU_{1/2}}$ to a $\mathcal{BU}_{1/2}$ and $P_3$ is given by
    \begin{align*}
        P_3 = \begin{pmatrix}
             0_d & \gamma^{-1} (1-\tilde{\eta}/\eta)I_d \\ 0_d &  (\tilde{\eta}/\eta-1)I_d
         \end{pmatrix}
    \end{align*}
    with $\eta= e^{-\gamma h}$ and $\tilde{\eta}=e^{-\gamma h/2}$.
    We observe that 
    \begin{align*}
        (I_{2d}+P_3)^T M (I_{2d}+P_3)&= M+ \begin{pmatrix}
            0_d & \gamma^{-1}(\frac{\tilde{\eta}}{\eta}-1)(-\gamma^{-2} K+ \tau(1-2\tau)I_d) \\ 0_d & \gamma^{-2}(\frac{\tilde{\eta}}{\eta}-1)\frac{1+2\tau}{2}I_d
        \end{pmatrix}
        \\ & +\begin{pmatrix}
            0_d & 0_d \\ \gamma^{-1}(\frac{\tilde{\eta}}{\eta}-1)(-\gamma^{-2} K+ \tau(1-2\tau)I_d) & \gamma^{-2}(\frac{\tilde{\eta}}{\eta}-1)\frac{1+2\tau}{2}I_d
        \end{pmatrix}
        \\ & + \begin{pmatrix}
            0_d & 0_d \\ 0_d & \gamma^{-2}(1-\frac{\tilde{\eta}}{\eta})^2(K\gamma^{-2} +2\tau^2I_d)
        \end{pmatrix}
        \\ & \prec \Big(1+ 2\tau (\frac{\tilde{\eta}}{\eta}-1)\Big)M+ \begin{pmatrix}
            (\frac{\tilde{\eta}}{\eta}-1)\gamma^{-2} K  & 0_d \\ 0_d & \gamma^{-2}(\frac{\tilde{\eta}}{\eta}-1)(\gamma^{-2} K + I_d)
        \end{pmatrix}
        \\ & + \begin{pmatrix}
            0_d & 0_d \\ 0_d & \gamma^{-2}(1-\frac{\tilde{\eta}}{\eta})^2(K\gamma^{-2} +2\tau^2I_d)
        \end{pmatrix},
    \end{align*}
    since 
    \begin{align}
        \gamma^{-1}\Big(\frac{\tilde{\eta}}{\eta}-1\Big)\gamma^{-2} z^T K w \le \gamma^{-2}\Big(\frac{\tilde{\eta}}{\eta}-1\Big)\frac{1}{2} z^T K z+\gamma^{-4}\Big(\frac{\tilde{\eta}}{\eta}-1\Big)\frac{1}{2}w^T K w.
    \end{align}
    By \eqref{eq:condition_h_BU} and \eqref{eq:tau}, it holds $(\frac{\tilde{\eta}}{\eta}-1)\le e^{\gamma h/2}\frac{\gamma h}{2}\le \frac{\gamma h}{2}e^{1/30}$ and $\tau\le 1/8$ and hence
    \begin{align*}
        (I_{2d}+P_3)^T M (I_{2d}+P_3)&\prec\Big(1+ 2\tau \frac{\gamma h e^{\frac{1}{30}}}{2}\Big)M+\begin{pmatrix}
            \frac{h\gamma^{-1}}{2}e^{\frac{1}{30}} K  & 0_d \\ 0_d & \frac{h}{2}e^{\frac{1}{30}}(1+ \frac{e^{\frac{1}{30}}}{30})(\gamma^{-3} K + \gamma^{-1} I_d)
        \end{pmatrix}
        \\ & \prec \Big(1+h\gamma \max(1, L_k\gamma^{-2}) \Big)M. 
    \end{align*}
    By \Cref{prop:case1_BU}, it holds 
    \begin{align*}
         r_l^2 ((\mathbf{X}_{BU},\mathbf{V}_{BU}),(\mathbf{X}_{BU}',\mathbf{V}_{BU}'))&\le \Big(1+h\gamma \max(1, L_k\gamma^{-2}) \Big)(z,w)^T(P_{BU})^T M P_{BU}(z,w).
         \\&\le \Big(1+h\gamma \max(1, L_k\gamma^{-2}) \Big)(z,w)^T M (z,w).
    \end{align*}
    By concavity of $f$, it holds
    \begin{align*}
        \rho ((\mathbf{X}_{BU},\mathbf{V}_{BU}),(\mathbf{X}_{BU}',\mathbf{V}_{BU}'))& \le   \sqrt{1+h \gamma\max(1, L_K \gamma^{-2})} \rho((x,v),(x',v')).
    \end{align*}

    These estimates for the $\mathcal{U}_{1/2}$-step and the $\mathcal{BU}_{1/2}$-step combined with the contraction result for the $\mathcal{BU}$-scheme (\Cref{thm:BU_contr})
    concludes the proof since by \eqref{eq:c_BU} $\frac{1}{1-ch}\le \frac{1}{1-e^{-\gamma h}\gamma h/16}\le 1+\gamma h/16$ and hence $\frac{\max( (1+\alpha\gamma h/2)^2, 1+\gamma h \max(1, L_k\gamma^{-2}))}{1-ch}\le \mathbf{C}$.

\end{proof}

To prove \Cref{thm:compl_UBU} and \Cref{thm:compl_UBU_2}, we first introduce an auxiliary result bounding the difference between the continuous time kinetic Langevin dynamics and the UBU discretization scheme.

\begin{lemma}\label{lem:UiT_UBU}
    Assume that $h <\min\left\{\frac{1}{2\gamma},\frac{\gamma}{2L}\right\}$. Consider the kinetic Langevin dynamics and the UBU scheme with synchronously coupled Brownian motion and with initial points $(\mathbf{X}_{0},\mathbf{V}_{0}) = (X_{0},V_{0}) \sim \mu_{\infty}$. Assume that Assumption \ref{ass} and Assumption \ref{ass:hessian_lip} is satisfied and, then for $k \in \mathbb{N}$ we have that
    \begin{align*}
        &\alpha\mathbb{E}\left[\left|\mathbf{X}_{k} - X_{kh}\right|\right] + \mathbb{E}\left[\left|\mathbf{X}_{k} - X_{kh} + \gamma^{-1}\left(\mathbf{V}_{k} - V_{kh}\right)\right|\right] \leq e^{3\max\left\{\gamma,\frac{2L}{\gamma}\right\}hk}(\mathbf{M}_{b} + \alpha(k+1)\mathbf{M}_{a}),
    \end{align*}
    where
    \begin{align*}
        \mathbf{M}_{a} &= \frac{h^{3}}{24}\left(\left(\frac{\sqrt{42}}{2} + 1\right)L + \gamma L^{1/2}\right)d^{1/2},\\
        \mathbf{M}_{b} &= \frac{\gamma^{-1}(k+1)Lh^{2}\sqrt{d}}{4},
    \end{align*}
    and if Assumption \ref{ass:hessian_lip} is satisfied $\mathbf{M}_{b}$ is refined to
    \begin{align*}
        \mathbf{M}_{b} &= \gamma^{-1}(k+1)\frac{h^{3}\sqrt{d}}{24}\left(\sqrt{3}L_{1}\sqrt{d} + L^{3/2} + \gamma L\right) + \sqrt{2\gamma^{-1}}\sqrt{\frac{(k+1)h^{5}L^{2}d}{192}},
    \end{align*}
    further if Assumption \ref{ass:strong_hessian_lip} is satisfied this $\mathbf{M}_{b}$ is refined to
    \begin{equation*}
        \mathbf{M}_{b} = \gamma^{-1}(k+1)\frac{h^{3}\sqrt{d}}{24}\left(\sqrt{3}L^{s}_{1} + L^{3/2} + \gamma L\right) + \sqrt{2\gamma^{-1}}\sqrt{\frac{(k+1)h^{5}L^{2}d}{192}}.
    \end{equation*}
\end{lemma}

% \kati{where in the proof do you use the condition on $h$? To my knowledge I got rid of this condition in the Euler discretization .}
% \pete{I used it in one place to simplify one of the bounds, however after looking at it again I don't think it is needed as in the Euler case. Do you want me to remove this condition?}
% \kati{thanks. I think we can leave the condition otherwise the constants become very long and in the main theorem this condition is already covered by \eqref{eq:condition_h_BU}.}
\begin{proof}
\cite{sanz2021wasserstein} use a different formulation of the solution of kinetic Langevin dynamics, which is derived by using It\^o's formula on the product $e^{\gamma t}V_{t}$. For initial condition $(X_0,V_0) \in \mathbb{R}^{2d}$ the solution of \eqref{eq:LD} can be written as:
\begin{align}
V_{t} &= \mathcal{E}(t)V_{0} - \int^{t}_{0}\mathcal{E}(t-s)\nabla U(X_{s})\rmd s + \sqrt{2\gamma}\int^{t}_{0}\mathcal{E}(t-s)\rmd B_s, \label{eq:cont_v}\\
X_{t} &= X_0 + \mathcal{F}(t)V_{0} - \int^{t}_{0}\mathcal{F}(t-s)\nabla U(X_{s})\rmd s + \sqrt{2\gamma}\int^{t}_{0}\mathcal{F}(t-s)\rmd B_{s} \label{eq:cont_x},
\end{align}
where
\begin{equation}\label{eq:EFdef}
\mathcal{E}(t) = e^{-\gamma t}\qquad \mathcal{F}(t) = \frac{1-e^{-\gamma t}}{\gamma}.
\end{equation}
Further the UBU scheme (as in \cite{sanz2021wasserstein}) can be expressed as 
\begin{align}
\mathbf{V}_{k+1} &= \mathcal{E}(h)\mathbf{V}_k - h\mathcal{E}(h/2)\nabla U(\overline{\mathbf{X}}_{k}) + \sqrt{2\gamma}\int^{(k+1)h}_{kh}\mathcal{E}((k+1)h-s)\rmd B_s,\label{eq:disc_v}\\
\overline{\mathbf{X}}_{k} &= \mathbf{X}_{k} + \mathcal{F}(h/2)\mathbf{V}_{k} + \sqrt{2\gamma}\int^{(k+1/2)h}_{kh}\mathcal{F}((k+1/2)h-s)\rmd B_s, \label{eq:disc_y}\\ 
\mathbf{X}_{k+1} &= \mathbf{X}_{k} + \mathcal{F}(h)\mathbf{V}_{k} - h\mathcal{F}(h/2)\nabla U(\overline{\mathbf{X}}_{k}) + \sqrt{2\gamma}\int^{(k+1)h}_{kh} \mathcal{F}((k+1)h-s)\rmd B_s,\label{eq:disc_x}
\end{align}
this is convenient for comparison with \eqref{eq:cont_v} and \eqref{eq:cont_x}. 

If we consider synchronously coupled Brownian motion with $(X_{0},V_{0}) = (\mathbf{X}_{0},\mathbf{V}_{0}) \sim \mu_{\infty}$, let us define
\begin{align*}
    a_{k} &:= \alpha\mathbb{E}\left[\left|\mathbf{X}_{k} - X_{kh}\right|\right], \qquad b_{k} := \mathbb{E}\left[\left|\mathbf{X}_{k} - X_{kh} + \gamma^{-1}\left(\mathbf{V}_{k} - V_{kh}\right)\right|\right],
\end{align*}
then we first have 
\begin{align*}
   &a_{k+1} \\
   &= \alpha\mathbb{E}\left[\left|\mathbf{X}_{k}-X_{kh} + \mathcal{F}(h)\left(\mathbf{V}_{k} - V_{kh}\right) - h\mathcal{F}(h/2)\nabla U(\overline{\mathbf{X}}_{k}) + \int^{h}_{0}\mathcal{F}(h-s)\nabla U(X_{kh+s})\rmd s\right|\right]\\
   &\leq \mathbb{E}\left[\alpha\left|\mathbf{X}_{k}-X_{kh} \right| + h\alpha\left|\mathbf{V}_{k}-V_{kh}\right| + \alpha\left|h\mathcal{F}(h/2)\nabla U(\overline{\mathbf{X}}_{k}) - \int^{h}_{0}\mathcal{F}(h-s)\nabla U(X_{kh+s})\rmd s\right|\right]\\
   &\leq a_{0} + \sum^{k}_{i=1}h\gamma(\alpha b_{i} + a_{i}) +  \alpha\sum^{k}_{i=0}\mathbb{E}\left[\left|h\mathcal{F}(h/2)\nabla U(\overline{\mathbf{X}}_{i}) - \int^{h}_{0}\mathcal{F}(h-s)\nabla U(X_{ih+s})\rmd s\right|\right],
\end{align*}
where following Section 7.6 of \cite{sanz2021wasserstein} we expand $\mathcal{F}(h-s)\nabla U(X_{kh+s})$ by the fundamental theorem of calculus and get 
\[
h\mathcal{F}(h/2)\nabla U(\overline{\mathbf{X}}_{i}) - \int^{h}_{0}\mathcal{F}(h-s)\nabla U(X_{ih+s})\rmd s = -h\mathcal{F}(h/2)(\nabla U(X_{ih+1/2}) - \nabla U(\overline{\mathbf{X}}_{i})) + I_{1} + I_{2},
\]
where $I_{1}$ and $I_{2}$ are defined as in \cite{sanz2021wasserstein}
\begin{align*}
    I_{1} &= -\int^{(i+1)h}_{ih}\int^{s}_{(i+1/2)h}\mathcal{F}((i+1)h-s')\nabla^{2}U(X_{s'})V_{s'}\rmd s'\rmd s\\
    I_{2} &= -\int^{(i+1)h}_{ih}\int^{s}_{(i+1/2)h}\mathcal{E}((i+1)h-s')\nabla U(X_{s'})\rmd s'\rmd s,
\end{align*}
and they show the following bounds
\begin{align*}
    \mathbb{E}\left|I_{1} + I_{2}\right| &\leq \left(\mathbb{E}\left|I_{1}\right|^{2}\right)^{1/2} + \left(\mathbb{E}\left|I_{2}\right|^{2}\right)^{1/2}\leq \frac{h^{3}}{24}\left(\left(\frac{\sqrt{42}}{2} + 1\right)L + \gamma L^{1/2}\right)d^{1/2} := \mathbf{M}_{a}.
\end{align*} We can also bound
\begin{align*}
    \mathbb{E}\left[\left|-h\mathcal{F}(h/2)(\nabla U(X_{(i+1/2)h}) - \nabla U(\overline{\mathbf{X}}_{i}))\right|\right]\leq \frac{h^{2}}{2}L\mathbb{E}\left(\left|X_{ih} - \mathbf{X}_{i}\right| + \frac{h}{2}\left| V_{ih}-\mathbf{V}_{i}\right|\right),
\end{align*}
and combining these results and using that $h<\min\left\{\frac{1}{2\gamma},\frac{\gamma}{2L}\right\},$ we have 
\begin{align*}
    &a_{k+1} \leq 2h\gamma\sum^{k}_{i=1}(\alpha b_{i} + a_{i}) + \alpha (k+1)\mathbf{M}_{a}.
\end{align*}

If we consider $b_{k}$ we have
\begin{align*}
   &b_{k+1} \leq b_{0} + \gamma^{-1}\mathbb{E}\left[\left|\sum^{k}_{i=0}\int^{h}_{0}(\nabla U(X_{ih+s}) - \nabla U(\overline{\mathbf{X}}_{i}))\rmd s\right|\right]\\
   &\leq b_{0} + \gamma^{-1}\mathbb{E}\left[\left|\sum^{k}_{i=0}\int^{h}_{0}(\nabla U(X_{ih+s}) - \nabla U(X_{(i+1/2)h}))\rmd s\right|\right] 
   \\ & + h\gamma^{-1}\sum^{k}_{i=0}\mathbb{E}\left[\left|\nabla U(\overline{\mathbf{X}}_{i}) - \nabla U(X_{(i+1/2)h})\right|\right] 
\end{align*}
where we now want to bound the second term as we have an estimate for the third term by the previous arguments. Using \^Ito-Taylor expansion twice as in \cite{sanz2021wasserstein} we have that
\begin{align*}
    \int^{(i+1)h}_{ih}\nabla U(X_{s})\rmd s = h\nabla U(X_{(i+1/2)h}) +  \int^{(i+1)h}_{(i+1/2)h}\int^{s}_{(i+1/2)h}\int^{s'}_{2(i+1/2)h-s'}\rmd (\nabla^{2} U(X_{s''})v)\rmd s'\rmd s,
\end{align*}
where we can use \^Ito's formula on $\nabla^{2} U(X_{s''})v$ to get that
\[
\int^{(i+1)h}_{ih}(\nabla U(X_{s}) - \nabla U(X_{(i+1/2)h}))\rmd s =  I_{3} + I_{4} + I_{5} + I_{6},
\]
where
\begin{align*}
    I_{3} &= \int^{(i+1)h}_{(i+1/2)h}\int^{s}_{(i+1/2)h}\int^{s'}_{2(i+1/2)h-s'}\nabla^{3}U(X_{s''})[V_{s''},V_{s''}]\rmd s''\rmd s'\rmd s\\
    I_{4} &= -\gamma\int^{(i+1)h}_{(i+1/2)h}\int^{s}_{(i+1/2)h}\int^{s'}_{2(i+1/2)h-s'}\nabla^{2}U(X_{s''})V_{s''}\rmd s''\rmd s'\rmd s\\
    I_{5} &= -\int^{(i+1)h}_{(i+1/2)h}\int^{s}_{(i+1/2)h}\int^{s'}_{2(i+1/2)h-s'}\nabla^{2}U(X_{s''})\nabla U(X_{s''})\rmd s''\rmd s'\rmd s\\
    I_{6} &=\sqrt{2\gamma}\int^{(i+1)h}_{(i+1/2)h}\int^{s}_{(i+1/2)h}\int^{s'}_{2(i+1/2)h-s'}\nabla^{2}U(X_{s''})\rmd B_{s''} \rmd s'\rmd s.
\end{align*}
We first bound $I_{3},I_{4}$ and $I_{5}$ as follows. For $I_{3}$ we use the fact that each scalar component of $V_{s''} \sim \mathcal{N}(0,I_{d})$ and therefore under Assumption \ref{ass:strong_hessian_lip} and \cite[Lemma 7]{paulin2024}
\[
\mathbb{E}\left[\left\|\nabla^{3}U(X_{s''})[V_{s''},V_{s''}]\right\|\right]\leq L^{s}_{1}\sqrt{3d}
\]
and therefore
\begin{align*}
\mathbb{E}\left[\left| I_{3}\right|\right]& \leq  \frac{L^{s}_{1}\sqrt{3d}h^{3}}{24}.
\end{align*}
Under Assumption \ref{ass:hessian_lip} similarly we have
\begin{align*}
\mathbb{E}\left[\left| I_{3}\right|\right]& \leq  \frac{\sqrt{3}L_{1}dh^{3}}{24}.
\end{align*}
To bound $I_{4}$ and $I_{5}$, we observe 
%For $I_{4}$ we have that
\begin{align*}
    \mathbb{E}\left[\left\|\nabla^{2}U(X_{s''})V_{s''}\right\|\right]&\leq L\mathbb{E}\left(\left\|V_{s''}\right\|\right) \leq L\sqrt{d}, && \text{and}
    \\ \mathbb{E}\left[\left\|\nabla^{2}U(X_{s''})\nabla U(X_{s''})\right\|\right]&\leq L\mathbb{E}\left(\|\nabla U(X_{s''})\|\right) \leq L^{3/2}\sqrt{d}
\end{align*}
and therefore
\begin{align*}
    \mathbb{E}\left[\left| I_{4}\right|\right]& \leq \frac{\gamma L\sqrt{d}h^{3}}{24}, 
%\end{align*}
% For $I_{5}$ we have that
% \begin{align*}
%     \mathbb{E}\left[\left\|\nabla^{2}U(X_{s''})\nabla U(X_{s''})\right\|\right]\leq L\mathbb{E}\left(\|\nabla U(X_{s''})\|\right) \leq L^{3/2}\sqrt{d}
% \end{align*}
% and therefore
%\begin{align*}
    \mathbb{E}\left[\left| I_{5}\right|\right] \leq \frac{L^{3/2}\sqrt{d}h^{3}}{24}.
\end{align*}
Then we have that 
\begin{align*}
    &\gamma^{-1}\mathbb{E}\left[\left|\sum^{k}_{i=0}\int^{h}_{0}(\nabla U(X_{ih+s}) - \nabla U(X_{(i+1/2)h}))\rmd s\right|\right] \\
    &\leq \gamma^{-1}(k+1)\frac{h^{3}\sqrt{d}}{24}\left(L_{1} + L^{3/2} + \gamma L \right) \\
    &+ \sqrt{2\gamma^{-1}}\mathbb{E}\left[\left|\sum^{k}_{i=0}\int^{(i+1)h}_{(i+1/2)h}\int^{s}_{(i+1/2)h}\int^{s'}_{2(i+1/2)h - s'}\nabla^{2}U(X_{s''})\rmd B_{s''}\rmd s'\rmd s\right|\right].
\end{align*}
We can estimate the final expectation by using Jensen's inequality and considering the following estimate of the second moment
\begin{align*}
    &\mathbb{E}\left[\left|\sum^{k}_{i=0}\int^{(i+1)h}_{(i+1/2)h}\int^{s}_{(i+1/2)h}\int^{s'}_{2(i+1/2)h - s'}\nabla^{2}U(X_{s''})\rmd B_{s''}\rmd s'\rmd s\right|^{2}\right] \\
    \intertext{using the fact that the Brownian motions are over disjoint time intervals corresponding to each $0\leq i \leq k$  we can apply Fubini's theorem and \^Ito's isometry to get}&=\sum^{k}_{i=0}\mathbb{E}\left[\left|\int^{(i+1)h}_{(i+1/2)h}\int^{s}_{(i+1/2)h}\int^{s'}_{2(i+1/2)h - s'}\nabla^{2}U(X_{s''})\rmd B_{s''}\rmd s'\rmd s\right|^{2}\right] \leq \frac{(k+1)h^{5}L^{2}d}{192}.
\end{align*}
We define  $\mathbf{M}_{b} := \gamma^{-1}(k+1)\frac{h^{3}\sqrt{d}}{24}\left(\mathbf{M}_c + L^{3/2} + \gamma L\right) + \sqrt{2\gamma^{-1}}\sqrt{\frac{(k+1)h^{5}L^{2}d}{192}}$ with $\mathbf{M}_c=\sqrt{3}L_1^s$ under Assumption \ref{ass:strong_hessian_lip} and $\mathbf{M}_c=\sqrt{3}L_1\sqrt{d}$ under Assumption \ref{ass:hessian_lip}.
%Under Assumption \ref{ass:strong_hessian_lip} we define $\mathbf{M}_{b} := \gamma^{-1}(k+1)\frac{h^{3}\sqrt{d}}{24}\left(L^{s}_{1} + L^{3/2} + \gamma L\right) + \sqrt{2\gamma^{-1}}\sqrt{\frac{(k+1)h^{5}L^{2}d}{192}}$, and under Assumption \ref{ass:hessian_lip} we define
%$\mathbf{M}_{b} := \gamma^{-1}(k+1)\frac{h^{3}\sqrt{d}}{24}\left(L_{1}\sqrt{d} + L^{3/2} + \gamma L\right) + \sqrt{2\gamma^{-1}}\sqrt{\frac{(k+1)h^{5}L^{2}d}{192}}$.

Further when Assumption \ref{ass:hessian_lip} and Assumption \ref{ass:strong_hessian_lip} are not satisfied we use the fundamental theorem of calculus to have that $\nabla U(X_{ih+s}) = \nabla U (X_{(i+1/2)h}) + \int^{ih +s}_{(i+1/2)h}\nabla^{2}U(X_{s'})V_{s'}\rmd s'$ and therefore we estimate 
\begin{align*}
    &\gamma^{-1}\mathbb{E}\left[\left|\sum^{k}_{i=0}\int^{h}_{0}(\nabla U(X_{ih+s}) - \nabla U(\overline{\mathbf{X}}_{i}))\rmd s\right|\right] \\
    &\leq h\gamma^{-1}\sum^{k}_{i=0}\mathbb{E}\left[\left|\nabla U(X_{(i+1/2)h}) -\nabla U(\overline{\mathbf{X}}_{i})\right|\right] + \frac{\gamma^{-1}(k+1)Lh^{2}\sqrt{d}}{4}.
\end{align*}
In this case, we define $\mathbf{M}_{b} := \frac{\gamma^{-1}(k+1)Lh^{2}\sqrt{d}}{4}$.

Combining the estimates we  have
\begin{align*}
    b_{k+1} &\leq  h\sum^{k}_{i=1}\left(\gamma a_{i} + \frac{hL}{2}b_{i}\right) + \mathbf{M}_{b},
\end{align*}
and combining this with the iteration inequality for $a_{k+1}$ we obtain
\begin{align*}
    a_{k+1} + b_{k+1} &\leq h\sum^{k}_{i=1}\left(3\gamma a_{i} + \frac{6L}{\gamma}b_{i}\right)+ (\mathbf{M}_{b} + \alpha(k+1)\mathbf{M}_{a}) \\
    &\leq 3h\max\left\{\gamma,\frac{2L}{\gamma}\right\}\sum^{k}_{i=1}\left(a_{i} + b_{i}\right) + (\mathbf{M}_{b} + \alpha(k+1)\mathbf{M}_{a}).
\end{align*}
Now we note that the sequence $\left(a_{k} + b_{k}\right)_{k \in \mathbb{N}}$ is bounded from above by the sequence $(c_{k})_{k\in \mathbb{N}}$ satisfying
\[
c_{k+1} := 3h\max\left\{\gamma,\frac{2L}{\gamma}\right\}\sum^{k}_{i=1}c_{i} + (\mathbf{M}_{b} + \alpha(k+1)\mathbf{M}_{a}),
\]
and $c_{1} = \mathbf{M}_{b} + \alpha(k+1)\mathbf{M}_{a}$. Then, we have
\begin{align*}
    a_{k+1} + b_{k+1} &\leq c_{k+1} = \left(1 + 3h\max\left\{\gamma,\frac{2L}{\gamma}\right\}\right)c_{k}\\
    &\leq \left(1 + 3h\max\left\{\gamma,\frac{2L}{\gamma}\right\}\right)^{k}c_{1} \leq e^{3\max\left\{\gamma,\frac{2L}{\gamma}\right\}hk}(\mathbf{M}_{b} + \alpha(k+1)\mathbf{M}_{a}).
    %&= h^{2}\sqrt{d}\Bigg(\frac{1}{\max\left\{\gamma,\frac{2L}{\gamma}\right\}}\left(\gamma^{-1}\left(L_{1} + L^{3/2} + \gamma L\right) + \alpha\left((\frac{\sqrt{42}}{2}+1)L + \gamma L^{1/2}\right)\right) + \sqrt{\frac{\gamma^{-1}L^{2}}{\max\left\{\gamma,\frac{2L}{\gamma}\right\}}}\Bigg) .
\end{align*}
\end{proof}

\begin{proof}[Proof of Theorem~\ref{thm:compl_UBU}]
    Inspired by the interpolation argument used in \cite{leimkuhler2023contractiona} we define 
    $(\mathbf{X}_{l},\mathbf{V}_{l})$ as $l$ steps of the UBU scheme and $(X_{lh},V_{lh})$ is defined by \eqref{eq:LD} at time $lh \geq 0$, where these are both initialized at $ (X_{0},V_{0}) = (\mathbf{X}_{0},\mathbf{V}_{0}) \sim \mu_{\infty}$ and have synchronously coupled Brownian motion. We further define a sequence of interpolating variants $(\mathbf{X}^{(k)}_{l},\mathbf{V}^{(k)}_{l})$ for every $k = 0,...,l$ all initialized $(\mathbf{X}^{(k)}_{0},\mathbf{V}^{(k)}_{0}) = (\mathbf{X}_{0},\mathbf{V}_{0})$, where we define $(\mathbf{X}^{(k)}_{i},\mathbf{V}^{(k)}_{i})^{k}_{i=1} := (X_{ih},V_{ih})^{k}_{i=1}$ and  $(\mathbf{X}^{(k)}_{i},\mathbf{V}^{(k)}_{i})^{l}_{i=k+1}$ by UBU steps and for $k = l$ we simply have just the continuous diffusion \eqref{eq:LD}. I.e. $(\mathbf{X}^{(k)}_{l},\mathbf{V}^{(k)}_{l})$ is defined by $k$ steps of the continuous time process, followed by $l-k$ steps of the discretization. Using Lemma \ref{lem:UiT_UBU} we split up the steps into blocks of size $\Tilde{l} = \left\lceil \frac{1}{3h\max\left\{\gamma,\frac{2L}{\gamma}\right\}}\right\rceil$ as 
    \begin{align*}
        \mathcal{W}_{\rho}&(\mu_{\infty}\pi^{l},\mu_{\infty}\pi^{l}_{h}) 
        \\ &\leq  \mathcal{W}_{\rho}\left(\mu_{\infty}\pi^{\lfloor l/\Tilde{l}\rfloor \Tilde{l}}\pi^{l-\lfloor l/\Tilde{l}\rfloor \Tilde{l}}_{h},\mu_{\infty}\pi^{l}\right) + \sum^{\lfloor l/\Tilde{l} \rfloor - 1}_{j = 0} \mathcal{W}_{\rho} \left(\mu_{\infty}\pi^{j\Tilde{l}}\pi^{l-j\Tilde{l}}_{h} ,\mu_{\infty} \pi^{(j+1)\Tilde{l}}\pi_{h}^{l-(j+1)\Tilde{l}}\right).
    \end{align*}
    Using the fact that \eqref{eq:LD} preserves $\mu_{\infty}$ and the remaining steps follow the UBU scheme for which we have contraction we can use Lemma \ref{lem:UiT_UBU} and Theorem \ref{thm:UBU_contr} to achieve
    \[
     \mathcal{W}_{\rho} \left(\mu_{\infty}\pi^{j\Tilde{l}}\pi^{l-j\Tilde{l}}_{h} ,\mu_{\infty} \pi^{(j+1)\Tilde{l}}\pi_{h}^{l-(j+1)\Tilde{l}}\right) \leq \mathbf{C}e^{3\max\left\{\gamma,\frac{2L}{\gamma}\right\}h\Tilde{l}}(\mathbf{M}_{b} + \alpha (\Tilde{l}+1)\mathbf{M}_{a})(1-ch)^{l-(j+1)\Tilde{l}},
    \]
    where $\mathbf{M}_{b}$ depends on $\Tilde{l}$ and $\mathbf{C}$ is defined as in \eqref{eq:contr_UBU_C}.
    Summing up the terms we have that
    \begin{align*}
        \mathcal{W}_{\rho}(\mu_{\infty}\pi^{l},\mu_{\infty}\pi^{l}_{h}) \leq \frac{15\mathbf{C}(\mathbf{M}_{b} + \alpha (\Tilde{l}+1)\mathbf{M}_{a})}{1-(1-ch)^{\Tilde{l}}}
    \end{align*}
    and taking the limit as $l \to \infty$ in the following estimate we have
    \begin{align*}
        \mathcal{W}_{\rho}(\mu_{\infty},\mu_{h,\infty})&=\mathcal{W}_{\rho}(\mu_{\infty}\pi^l,\mu_{h,\infty}\pi_h^l)\le \mathcal{W}_{\rho}(\mu_{\infty}\pi^l,\mu_{\infty}\pi_h^l)+\mathcal{W}_{\rho}(\mu_{\infty}\pi_h^l,\mu_{h,\infty}\pi_h^l)
        \\ & \le \mathcal{W}_{\rho}(\mu_{\infty}\pi^l,\mu_{\infty}\pi_h^l)+\mathbf{C}(1-ch)^l\mathcal{W}_{\rho}(\mu_{\infty},\mu_{h,\infty})\\
        &\leq \frac{15\mathbf{C}(\mathbf{M}_{b} + \alpha (\Tilde{l}+1)\mathbf{M}_{a})}{1-(1-ch)^{\Tilde{l}}} \leq 15\mathbf{C}(\mathbf{M}_{b} + \alpha (\Tilde{l}+1)\mathbf{M}_{a})\left(1 + \frac{1}{ch\Tilde{l}}\right)\\
        &\leq 15\mathbf{C}(\mathbf{M}_{b} + \alpha (\Tilde{l}+1)\mathbf{M}_{a})\left(1 + \frac{3\max{\left\{\gamma, \frac{2L}{\gamma}\right\}}}{c}\right),
    \end{align*}
    where we have used that $1/(1-e^{-x}) \leq 1 + 1/x$ for all $x > 0$ and the result follows, where we simplify the estimate.
\end{proof}

\begin{proof}[Proof of Theorem~\ref{thm:compl_UBU_2}]
    The results follows by applying triangle inequality, \Cref{thm:compl_UBU}, \Cref{thm:UBU_contr} and using \eqref{eq:dist_equiv}.
\end{proof}

%\bibliography{JMLRSubmission/sample} 
\bibliographystyle{plain}
\bibliography{sample.bib}

\end{document}